\documentclass[12pt]{amsart}
\usepackage[english]{babel}
\usepackage{a4wide,color,graphicx,amsmath,amssymb,verbatim,mathrsfs,latexsym}

\allowdisplaybreaks
\newtheorem{theorem}{Theorem}
\newtheorem{definition}{Definition}
\newtheorem{prop}{Proposition}
\newtheorem{lemma}{Lemma}

\newcommand{\vc}[1]{{\mathbf {#1}}} 
\newcommand{\cc}{\vec{c}}
\newcommand{\vv}{\vc{v}}
\newcommand{\vw}{\vc{w}}
\newcommand{\vzeta}{\vec{\zeta}}
\newcommand{\R}{\mathbb{R}}	
\newcommand{\N}{\mathbb{N}}	
\newcommand{\eps}{\varepsilon}	
\newcommand{\pa}{\partial}		
\newcommand{\Div}{\mathop{\mathrm{div}}\nolimits}	

\newcommand{\na}{\nabla}		
\newcommand{\OmT}{			
\Omega\times (0,T) }

\def\dx{\; {\rm d}x}
\def\dS{\; {\rm d}S}
\def\dt{\; {\rm d}t}
\def\dtau{\; {\rm d}\tau}

\title[Incompressible electrically charged chemically reacting mixtures]{Existence analysis for incompressible fluid model of electrically charged chemically reacting and heat conducting mixtures}
\author{Miroslav Bul\'i\v{c}ek}
\address{Charles University, Faculty of Mathematics and Physics, Sokolovsk\'a 83, 186 75 Praha 8, Czech Republic}
\email{mbul8060@karlin.mff.cuni.cz}
\author{Milan Pokorn\'y}
\address{Charles University, Faculty of Mathematics and Physics, Sokolovsk\'a 83, 186 75 Praha 8, Czech Republic}
\email{pokorny@karlin.mff.cuni.cz}
\author{Nicola Zamponi}
\address{Institute for Analysis and Scientific Computing, Vienna University of
	Technology, Wiedner Hauptstra\ss e 8--10, 1040 Wien, Austria}
\email{nicola.zamponi@tuwien.ac.at}
\date{\today}

\numberwithin{equation}{section}

\begin{document}

\begin{abstract}
We deal with a three dimensional model based on the use of barycentric velocity that describes unsteady flows of a heat conducting electrically charged multicomponent chemically reacting non-Newtonian fluid. We show that under certain assumptions on data and the constitutive relations for such a fluid there exists a global in time and large data weak solution. The paper has two key novelties. The first one is that we present a model that is thermodynamically and mechanically consistent and that is able to describe the cross effects in a  generality never considered before, i.e., we cover the so-called Soret effect, Dufour effect, Ohm law, Peltier effect, Joul heating, Thompson effect, Seebeck effect and also the generalized Fick law. The second key novelty is that contrary to the previous works on the similar topic, we do not need to deal with the energy equality method and therefore we are able to cover a large range of power-law parameters in the Cauchy-stress. In particular, we cover even the Newtonian case (which is the most used model), for which the existence analysis was missing.
\end{abstract}



\maketitle

\section{Introduction}

We consider a model of an incompressible, multicomponent, heat conducting, electrically charged fluid that is capable to capture the cross-diffusion effects and that is mechanically and thermodynamically consistent. The main goal of this paper is to propose such a model which is able to describe many physical phenomena and  to establish the global in time and large data existence of weak solutions to this model. The general framework we deal with is as follows. For a bounded open set
$\Omega \subset \mathbb{R}^3$ with a boundary of the class $C^{1,1}$ and given the length of the time interval $T>0$, we denote $\Omega_T:=(0,T)\times \Omega$. We consider a flow of an incompressible mixture in $\Omega_T$ that consists of $L\in \mathbb{N}$ constituents with $L\ge 2$. This problem is described by the following set of equations
\begin{align}
 \label{eq.c}
 \mbox{ for } i=1,2,\dots,L\colon \qquad \pa_t c_i + \Div(c_i \vc{v} + \vc{q}_{\cc}^i) - r_i &= 0,\\
 \label{eq.v}
 \pa_t \vc{v} + \Div(\vc{v}\otimes \vc{v} - \mathcal{T}) + Q\na\varphi &= 0,\\
 \label{eq.divv}
 \Div \vc{v} &= 0,\\
 \label{eq.poi}
 -\Delta\varphi &= Q,\\
 \pa_t E  +\Div \left(\left(\frac{|\vc{v}|^2}{2}+e+Q\varphi\right)\vc{v}  +\varphi\sum_{i=1}^L z_i \vc{q}_{\cc}^i  + \vc{q}_e- \mathcal{T} \vc{v} -\varphi\na\pa_t\varphi
 \right) &= 0.\label{eq.E}
\end{align}
Here, $\vec{c}:=(c_1,\ldots, c_L)$: $\Omega_T \to [0,1]^L$ are the species concentrations, $\vc{v}=(v_1,v_2,v_3)$: $\Omega_T \to\R^3$ is the velocity field, $\varphi$: $\Omega_T \to\R$ is the electric potential,  $\vc{q}_{\cc}^i=(q_{\cc}^{i1}, q_{\cc}^{i2}, q_{\cc}^{i3})$: $\Omega_T \to \R^3$ are the fluxes of the corresponding concentrations $c_i$, $\vc{q}_e=(q_e^1, q_e^2, q_e^3)$: $\Omega_T\to \R^3$ denotes the heat flux, $\vec{r}=(r_1,\ldots, r_L)$: $\Omega_T \to \R^L$ are the reaction/productions terms for the concentration $\vec{c}$, $\vec{z}=(z_1,\ldots, z_L)\in \R^L$ are  the specific electric charges of $\vec{c}$, $Q:=\sum_{i=1}^L c_i z_i=\cc\cdot \vec{z}$ is the total electric charge,  $\mathcal{T}$: $\Omega_T \to \R^{3\times 3}$ is the Cauchy stress tensor and $E:\Omega_T \to \R_+$ is the total energy of the fluid given as a sum of the kinetic, the internal and the electrostatic energy, i.e., $E:=|\vv|^2/2 + e + |\nabla \varphi|^2/2$ with $e$: $\Omega_T \to \R_+$ the internal energy. The  terms inside the divergence in \eqref{eq.E} represent the fluxes of the total energy due to the dissipation, the heat transfer, the flux of all charged constituent and the flux due to the time evolution of the electrostatic potential.

System \eqref{eq.c}--\eqref{eq.E} is completed by the initial conditions
\begin{equation}
\cc(0)=\cc^{\ 0}, \quad \vv(0)=\vv^0, \quad e(0)=e^0,\label{IC}
\end{equation}
and by the following set of boundary conditions on $\Gamma:=(0,T)\times \Omega$
\begin{align}
 \vv\cdot\nu &= 0,\quad (\mathcal{I} - \nu\otimes\nu)\mathcal{T}\nu = -\gamma \vv, \label{bc.vA}\\
 \vc{q}_{\cc}^i\cdot\nu &= q^i_{\cc\ \Gamma}, \quad
 \vc{q}_e\cdot\nu = q_{e\Gamma}, \quad
 \na\varphi\cdot\nu = q_{\varphi \Gamma},\label{bc.phiA}
\end{align}
where $\nu$ denotes the unit outward normal vector to $\pa\Omega$. All parameters, i.e., $\gamma$, $q^i_{\cc\ \Gamma}$, $q_{e\Gamma}$, $q_{\varphi \Gamma}$ may depend on $(t,x)$ and also on some unknowns and their form will be specific later. We would like to underline that while in \eqref{bc.vA}, which are the Navier slip boundary conditions, we do not permit the total mass to flow through the boundary, the boundary conditions \eqref{bc.phiA} allow to consider the flux of the single species, heat and charge.

System \eqref{eq.c}--\eqref{eq.E} contains a lot of simplifications. First, we consider the volume additivity and model the whole mixture as being incompressible, and we set the density to be equal to one. Second,  the magnetic field and polarization are neglected, which reduces the Maxwell equations to \eqref{eq.poi}, where again without loss of generality but for simplicity of further explanation, we set the permittivity equal to one. The Lorenz force in the momentum equation \eqref{eq.v} is then reduced to $Q\nabla \varphi$. The justification of such simplifications can be found e.g. in  \cite{Eckar:1940,Prigo:1967,RaYaWi94,RajagTao:1995}. On the other hand, the above system is still rich enough to describe many phenomena presented in the electrically charged and heat conducting mixture flow, as for example the Peltier effect, the Joule heat, the Fourier, the Fick and the Ohm laws, the Soret and the Dufour effects. The goal of the paper is to introduce such constitutive assumptions on parameters that will lead to this generality on one hand and will be mathematically  treatable  on the other hand. The main concept and also the key tool, used in this paper, is the mechanical and thermodynamic compatibility of all assumptions, that will lead to the energy balance, and the entropy inequality which will provide us  with the key estimates allowing us to develop the existence theory.

\subsection*{Notation}
Throughout the paper we use the following notation. A scalar or a scalar-valued function is printed with the usual font, a vector or a vector-valued function corresponding to the three-dimensional space are printed in bold-face, the vectors and vector-valued functions connected with the number of species ($L$) are printed with the sign $\vec{\ }$ above the letter, the tensors and tensor-valued functions in three space dimensions (i.e., with 9 components) are printed with a special font  (as e.g. $\mathcal{I}$ for $I$) and the tensors and tensor-valued functions connected at least with one component to the number of species are printed with another special font (as e.g. $\mathfrak{I}$ for $I$).

For any nonzero vector $\vec{a}\in\R^L$ we introduce the matrix  $\mathfrak{P}_{\vec{a}}\in\R^{L\times L}$ by the formula
$$
\mathfrak{P}_{\vec{a}}:=\mathfrak{I}-\frac{\vec{a}\otimes \vec{a}}{|\vec{a}|^2}, \qquad \textrm{ i.e. } (\mathfrak{P}_{\vec{a}})_{ij} := \delta_{ij} - \frac{a_i a_j}{|\vec{a}|^2}.
$$
The role of $\mathfrak{P}_{\vec{a}}$ is that it generates the projection onto the space orthogonal to $\vec{a}$, i.e., for any $\vec{x}\in \R^L$ there holds $\mathfrak{P}_{\vec{a}} \vec{x} \cdot \vec{a}=0$. We also define the vector
$$
\vec{\ell}:= (1,\ldots,1)\in \mathbb{R}^L.
$$

Concerning the function spaces,  we will use standard notation for the Lebesgue spaces ($L^p$) and the Sobolev spaces ($W^{k,p}$) endowed with the standard norms. We denote the corresponding norm with the lower index $p$ and, $k,p$, respectively. For vector-valued functions on $\Omega$ we will use notation $L^p(\Omega;\R^3)$; similarly for tensor-valued functions, Sobolev spaces etc. We further define for $1<p<\infty$
$$
L^p_{0,\Div}(\Omega):= \overline{C^{\infty}_{0,\Div}(\Omega; \R^3)}^{\|\cdot\|_{L^p(\Omega)}}
$$
and
$$
W^{1,p}_{\Div}(\Omega):= \overline{C^{\infty}_{\Div}(\overline{\Omega}; \R^3)}^{\|\cdot\|_{W^{1,p}(\Omega)}},
$$
where $C^{\infty}_{0,\Div}(\Omega; \R^3)$ denotes the set of smooth compactly supported functions in $\Omega$ with zero divergence. Next, we define
$$
W_{0,\nu}^{1,p}(\Omega)= \{\vc{u}\in L^p(\Omega;\R^3); \na \vc{u} \in L^p(\Omega;\R^{3\times 3}), \vc{u} \cdot \nu = 0 \mbox{ on }\partial \Omega\},
$$
where $\nu$ denotes the unit outer normal vector on $\partial \Omega$. For the functions defined on the time--space we use the notation $L^p(0,T; L^q(\Omega))$ and similarly for other spaces. For the notation of the norms, we will use the full description. For the corresponding dual spaces we will use the standard notation.

\section{Constitutive relations and main assumptions}

\subsection{Mechanical and thermodynamical consistency}
First, it is natural to assume that the chemical reactions change neither the total mass nor the total charge, i.e.,
\begin{align}
 \sum_{i=1}^L r_i = \sum_{i=1}^L z_i r_i = 0.\label{eq.r}
\end{align}
Second, since $\vc{q}_{\cc}^i$ model the flux of $c_i$, that is the difference of $c_i\vc{v}^i$ and $c_i\vc{v}$, where $\vc{v}^i$ is the velocity of the i-th species, we also impose the assumption
\begin{align}
 \sum_{i=1}^L \vc{q}_{\cc}^i=\vc{0}\label{eq.qc}
\end{align}
in order to obtain a compatible model, see \cite{BuHa15} for more details. Note that summing \eqref{eq.c} over $i=1,\ldots, L$ then leads under the assumptions \eqref{eq.r}--\eqref{eq.qc} to the following transport equation
\begin{equation}
\pa_t (\cc\cdot \vec{\ell})  +\Div \left((\cc\cdot \vec{\ell})\vv\right)=0.\label{eq.transport}
\end{equation}

Next, we focus on the equation for the internal energy and then also on the entropy inequality. Thus, we multiply the $i$-th equation in \eqref{eq.c} by $z_i$, sum over $i=1,\ldots, L$ and use \eqref{eq.r} to obtain the following identity for the total charge $Q=\sum_{i=1}^L c_i z_i$
\begin{align}
 \pa_t Q + \Div\left(Q \vv + \sum_{i=1}^L z_i \vc{q}_{\cc}^i\right) = 0, \label{eq.Q}
\end{align}
which can be in view of \eqref{eq.poi} rewritten as
\begin{align}
- \Delta \pa_t \varphi + \Div\left(Q \vv + \sum_{i=1}^L z_i \vc{q}_{\cc}^i\right) = 0. \label{eq.Q2}
\end{align}
It is worth noticing that this is also equivalent to
\begin{align*}
 \Div\left(Q \vv + \sum_{i=1}^L z_i \vc{q}_{\cc}^i -\nabla \pa_t \varphi\right) =0.
\end{align*}
Next, we multiply \eqref{eq.Q2} by $\varphi$ and obtain
\begin{align}
 \pa_t\frac{|\na\varphi|^2}{2}
 + \Div\left(Q\varphi \vv + \varphi\sum_{i=1}^L z_i \vc{q}_{\cc}^i- \varphi\na\pa_t\varphi\right)-Q\nabla \varphi\cdot \vv-\nabla \varphi \cdot  \sum_{i=1}^L z_i \vc{q}_{\cc}^i = 0.\label{eq.phi}
\end{align}
Finally, taking the scalar product of \eqref{eq.v} with  $\vv$, we deduce the following identity for the kinetic energy
\begin{align}
 \pa_t\frac{|\vv|^2}{2} + \Div\left( \frac{|\vv|^2 \vv}{2}- \mathcal{T}\vv\right)+ \mathcal{T}: \na \vv + Q \vv\cdot\na\varphi = 0.\label{eq.v2}
\end{align}
Consequently, subtracting \eqref{eq.phi} and \eqref{eq.v2} from \eqref{eq.E}, we obtain the equation for the internal energy
\begin{align}
\label{eq.e}
 \pa_t e + \Div(e \vv + \vc{q}_e) - \mathcal{T}:\na \vv + \sum_{i=1}^L z_i \vc{q}_{\cc}^i\cdot\na\varphi &=0.
\end{align}

To end this part we derive the identity for the entropy. We assume that the \emph{entropy density} associated to system \eqref{eq.c}--\eqref{eq.E} is a function of the internal energy and the concentration vector $\cc$, i.e., $s := s^*(\cc,e):(0,1)^L \times \R_{+} \to \R_{+}$. Next, we define  the {\em chemical potential\footnote{Usually the chemical potential is also given as $\frac{\mu}{\theta}$, where $\mu$ is the vector of mobilities.}} $\vzeta$: $\Omega_T \to \R^L$ and the {\em temperature} $\theta$: $\Omega_T \to \R_+$ as
\begin{align}
\vzeta=\vzeta^*(\cc,e):= -\pa_{\cc} s^*(\cc,e),\qquad \theta = \theta^*(\cc,e):=\frac{1}{\pa_e s^*(\cc,e)}. \label{zetatheta}
\end{align}
Then, multiplying \eqref{eq.c} by $\zeta_i$, \eqref{eq.e} by $\theta^{-1}$ and using the fact that $\Div \vv=0$ (see \eqref{eq.divv}) we get
\begin{align*}
 \pa_t s &= -\sum_{i=1}^L\zeta_i\pa_t c_i + \theta^{-1}\pa_t e\\
 &= \sum_{i=1}^L\zeta_i\left( \Div(c_i \vv + \vc{q}_{\cc}^i) - r_i \right)
 + \frac{1}{\theta}\left( -\Div(e \vv + \vc{q}_e) + \mathcal{T}:\na \vv - \sum_{i=1}^L z_i \vc{q}_{\cc}^i\cdot\na\varphi \right)\\
 &= \vv\cdot\left( \sum_{i=1}^L\zeta_i\na c_i - \frac{\na e}{\theta} \right)+\Div\left( \sum_{i=1}^L\zeta_i \vc{q}_{\cc}^i - \frac{\vc{q}_e}{\theta} \right)
 -\sum_{i=1}^L \vc{q}_{\cc}^i \cdot \na\zeta_i+  \vc{q}_e\cdot \na\frac{1}{\theta}\\
 &\quad -\sum_{i=1}^L\zeta_i r_i + \frac{\mathcal{T}:\na \vv}{\theta} - \frac{1}{\theta}\sum_{i=1}^L z_i \vc{q}_{\cc}^i\cdot\na\varphi\\
 &= \Div\left( \sum_{i=1}^L\zeta_i \vc{q}_{\cc}^i - \frac{\vc{q}_e}{\theta} \right) - \vv\cdot\na s
 -\vzeta \cdot \vec{r} + \frac{\mathcal{T}:\na \vv}{\theta}\\
 &\quad -\sum_{i=1}^L \vc{q}_{\cc}^i\cdot\left( \na\zeta_i + \frac{z_i}{\theta}\na\varphi \right) + \vc{q}_e\cdot\na\frac{1}{\theta},
\end{align*}
which, again due to \eqref{eq.divv}, leads to the final entropy identity
\begin{equation}\label{eq.s}
\begin{split}
 \pa_t s &+ \Div\left( s \vv -\sum_{i=1}^L\zeta_i \vc{q}_{\cc}^i +\frac{\vc{q}_e}{\theta}\right)\\
 &= -\vzeta\cdot \vec{r} + \frac{\mathcal{T}:\na \vv}{\theta}
 -\sum_{i=1}^L \vc{q}_{\cc}^i\cdot\left( \na\zeta_i + \frac{z_i}{\theta}\na\varphi \right) + \vc{q}_e\cdot\na\frac{1}{\theta}.
\end{split}
\end{equation}
The Second law of thermodynamics dictates that the right-hand side of \eqref{eq.s} has to be nonnegative and in what follows we introduce the constitutive relations for parameters that will be designed to satisfy this constraint.

\subsection{Constitutive equations and assumptions}

Next, we specify the assumptions on the constitutive equations. We also denote $C_1$ and $C_2$ some positive constants that will appear in the assumptions below.
\subsubsection*{Assumptions on the reaction term $\vec{r}$}
We assume that there exists a continuous functions $\vec{r}^{\, *}:\R^L \times \R_{+,0} \times \R^{L} \to \R$ such that for all $(\cc,\theta,\vzeta)\in \R^L \times \R_{+,0} \times \R^{L}$, we have
\begin{align}
 |\vec{r}^{\,*}(\cc,\theta,\vzeta)|\leq C_2,\qquad \vzeta \cdot \vec{r}^{\,*}(\cc,\theta,\vzeta) &\leq 0,\qquad
 \vec{r}^{\,*}(\cc,\theta,\vzeta)\cdot \vec{\ell} = \vec{z}\cdot \vec{r}^{\, *}(\cc,\theta,\vzeta) = 0.\label{hp.r}
\end{align}
Recall that $\vec{\ell} = (1,\dots,1) \in \R^L$.
We also point out  that in general, it is enough to assume that $r^*$ is defined only for $\cc$ from  $[0,1]^L$ and $\theta \in \R_+$, however, it will be easier for us to assume that this function is extended onto the whole $\R^L$ or to $\R_{+,0}$; the same will be used  several times below.
Then in \eqref{eq.c} we set
$$
\vec{r} := \vec{r}^{\,*}(\cc,\theta,\vzeta).
$$
Note that \eqref{eq.r}  immediately follows from \eqref{hp.r}. Moreover, we see that the first term on the right-hand side of \eqref{eq.s} is nonnegative. It is worth noticing that we do not consider any coercivity for $\vec{r}^{\,*}$, which was the essential assumption in \cite{BuHa15} in order to get a~priori estimate on $\vzeta$.
\subsubsection*{Assumptions on the fluxes $\mathfrak{q}_{\cc}$ and $\vc{q}_e$}
For simplicity, we focus in this paper only on the linear dependence of fluxes on the gradients of the temperature, chemical and electrostatic potentials and through the paper we shall assume that the fluxes are of the following form
\begin{align}
 \mathfrak{q}_{\cc} = \{\vc{q}_{\cc}^i\}_{i=1}^L \mbox{ with } \vc{q}_{\cc}^i &:= -\sum_{j=1}^L M_{ij}\left( \na\zeta_j + \frac{z_j}{\theta}\na\varphi \right) - m_i\na\frac{1}{\theta},\label{qc}\\
 \vc{q}_e &:= -\kappa\na\theta - \sum_{i=1}^L m_i\left( \na\zeta_i + \frac{z_i}{\theta}\na\varphi \right).\label{qe}
\end{align}
Next, we introduce the assumptions on further parameters appearing in \eqref{qc} and \eqref{qe}. We assume that $\kappa= \kappa^*(\cc,\theta)$: $\R^{L}\times \R_+ \to \R_+$ is a continuous function satisfying for some $\beta \in [0,2]$ (here the upper bound for $\beta$ is taken just for simplicity of the presentation and it is possible to allow higher values of $\beta$) and all $(\vec{c},\theta) \in \R^L\times \R_+$
\begin{align}
 C_1 \leq \frac{\kappa^*(\cc,\theta)}{1 + \theta^{-\beta}}\leq C_2.\label{hp.kappa}
\end{align}
Notice here that $\kappa^*(\cc,\theta)\na\theta$ represents the \emph{Fourier law} with the heat conductivity $\kappa^*(\cc,\theta)$ which is assumed to be uniformly bounded with respect to chemical concentration but can possibly blow up if the temperature $\theta$ tends to zero which is a natural\footnote{It should be also emphasized here that for many fluids this behavior is irrelevant since much before the temperature reaches zero  a phase transition will occur. On the other hand, for liquid gasses with the temperature near the absolute zero, which can be modeled as incompressible fluids, this behavior is quite typical.} requirement.

Second, we assume that $\mathfrak{M} = \mathfrak{M}^*(\cc,\theta) = \{M_{ij}^*(\cc,\theta)\}_{i,j=1}^L$: $\R^L\times \R_{+,0} \to \R^{L\times L}_{sym}$ is a continuous symmetric\footnote{Here, the space $\R^{L\times L}_{sym}$ denotes the space of all symmetric $L\times L$ matrices.} matrix-valued mapping and $\vec{m}=\vec{m}^*(\cc,\theta)= \{m_i^*(\cc,\theta)\}_{i=1}^L$: $\R^L\times \R_{+,0} \to \R^L$ is a continuous vector-valued mapping satisfying for all $(\cc,\theta)\in \R^L\times \R_{+,0}$
\begin{align}
 \sum_{i=1}^L M^*_{ij}(\cc,\theta) = \sum_{i=1}^L m_i^*(\cc,\theta) = 0,\qquad \textrm{ for all } j=1,\ldots,L,\label{hp.mM}
\end{align}
for all $\vec{w}\in \R^L$ and $M(\theta)\geq 0$
\begin{align}
 C_1M(\theta)|\mathfrak{P}_{\vec{\ell}}\, \vec{w}|^2\leq\sum_{i,j=1}^L M_{ij}^*(\cc,\theta)w_i w_j\leq C_2M(\theta)|\mathfrak{P}_{\vec{\ell}}\, \vec{w}|^2,\label{hp.M}
\end{align}
and for some $\varepsilon_0>0$
\begin{equation}
\begin{aligned}
C_1\min(1,\theta^{\beta-\varepsilon_0})&\le M(\theta)\le C(1+\theta)^{\frac53 - \varepsilon_0}, \\
|\vec{m}^*(\cc,\theta)|^2&\leq C_2\left\{\begin{aligned}
&\min\{M(\theta) \theta^{-\beta+\varepsilon_0},\theta^{-2(\beta-1)+\varepsilon_0}\}&&\textrm{for }\theta<1,\\
&M(\theta)\theta &&\textrm{for } \theta\ge 1.  \label{hp.m}
\end{aligned}\right.
\end{aligned}
\end{equation}

Let us make few comments about the model represented by \eqref{qc} and \eqref{qe}. First, condition \eqref{eq.qc} is automatically satisfied due to  assumption \eqref{hp.mM}. Therefore, we see that the matrix $\mathfrak{M}$ cannot be positive definite. On the other hand, due to  assumption \eqref{hp.M}, we see that it is positive definite on the  subspace of $\R^L$ which is orthogonal to $\vec{\ell}$. In addition, this subspace is evidently also the range of this matrix. The part of $\mathfrak{q}_{\cc}$ with $\mathfrak{M}\nabla \vzeta$ represents the generalized \emph{Fick law} and we see that due to the possible non-diagonal form of  $\mathfrak{M}$ and due to the presence of $\na \vzeta$ instead of $\na \cc$, we are able to cover the so-called cross effects in a big generality. In addition, the presence of $M(\theta)$ in \eqref{hp.M} and the fact that it is possibly vanishing at zero and tending to infinity for the temperature growing to infinity, makes the model also fully compatible with the \emph{Einstein law}. Further, the part of the flux with $\mathfrak{M} \vec{z}\,\na \varphi /\theta$ models the \emph{Ohm law}. Next, the term $\mathfrak{M} \na \vzeta :(\vec{z}\, \na \varphi)$ in \eqref{eq.E} models the \emph{Peltier effect}, and  the term $\mathfrak{D}(\vec{z}\,\na \varphi): (\vec{z}\,\na \varphi)/\theta$ models\footnote{$\mathfrak{D}$ will be introduced later, in connection with the boundary conditions.} the \emph{power of the Joul heat}. Finally, the possible dependence  of $\mathfrak{q}_{\cc}$ on $\nabla \theta$ then models the so-called \emph{Soret effect} and conversely, the dependence of  $\vc{q}_e$ on $\nabla \vzeta$ models the \emph{Dufour effect}. It is noticeable here that the second law of thermodynamics requires also the presence of $\na \varphi$ in the equation for $\vc{q}_e$, as will be seen later. This represents the \emph{Tomphson effect} and the \emph{Seebeck effect}. We refer at this point to \cite{Roubi:2007} and \cite{BuHa15} for more detailed explanation of these phenomena and for further physical models.

\subsubsection*{Assumptions on the Cauchy stress $\mathcal{T}$}
We assume that the Cauchy stress is decomposed into two parts
\begin{equation}\label{Cauchy}
\mathcal{T}=-p\mathcal{I} + \mathcal{S},
\end{equation}
where $p:\Omega_T\to \R$ is the mean normal stress --- the pressure --- which is in the incompressible setting also an unknown function, and $\mathcal{S}$ is the constitutively determined part of the Cauchy stress given by
\begin{align}
 \mathcal{S} = \mathcal{S}^*(\cc,\theta,\mathcal{D}(\vv)).\label{S}
\end{align}
Here, $\mathcal{D}(\vv)$ denotes the symmetric part of the velocity gradient, i.e., $\mathcal{D}(\vv):=(\nabla \vv + (\nabla \vv)^T)/2$, and the mapping $\mathcal{S}^*$: $\R^L \times \R_{+,0} \times \R^{3\times 3}_{sym}\to \R^{3\times 3}_{sym}$ is assumed to be continuous. Further, we assume that for some $r\in (1,\infty)$ and for all $(\cc,\theta, \mathcal{D},\mathcal{B})\in \R^L \times \R_{+,0} \times \R^{3\times 3}_{sym}\times \R^{3\times 3}_{sym}$ we have
\begin{align}
 \label{hp.S}
  \mathcal{S}^*(\cc,\theta,\mathcal{D}):\mathcal{D} &\geq C_1 |\mathcal{D}|^r-C_2 ,\\
 \label{hp.S.2}
  |\mathcal{S}^*(\cc,\theta,\mathcal{D})|&\leq C_2(1+|\mathcal{D}|^{r-1}),\\
 \label{hp.S.3}
  \mathcal{S}^*(\cc,\theta,0)&=0,\quad (\mathcal{S}^*(\cc,\theta,\mathcal{D})-\mathcal{S}^*(\cc,\theta,\mathcal{B})) : (\mathcal{D}-\mathcal{B})\geq 0.
\end{align}
Such a general framework allows to cover also non-Newtonian behavior of the fluid. We would like to emphasize here that the existence theory built in \cite{BuHa15,Roubi:2007} holds true for $r\ge 11/5$ and therefore does not include  the Newtonian fluids and considers only the shear-thickening fluids. It is however not the case in our setting and we will be able to treat more physically relevant cases, including the Newtonian and the non-Newtonian shear-thinning fluids. Finally, the monotonicity of $\mathcal{S}^*$, see \eqref{hp.S.3}, and the fact that $\Div \vv=0$ lead to the nonnegativity of the second term on the right hand side of \eqref{eq.s}, which is the entropy production due to the mechanical dissipation.

A typical example  we have in mind is
$$
\mathcal{S}^*(\cc,\theta,\mathcal{D}(\vv))= g^*(\cc,\theta) (1+ |\mathcal{D}(\vv)|^{r-1})\mathcal{D}(\vv),
$$
where $g^*(\cc,\theta)$ is a continuous function  bounded  from below (by a positive constant) and from above.

\subsubsection*{Assumptions on the entropy}
We assume that the entropy $s$ decomposes as a sum of two contributions, one from the internal energy $e$ and another from the concentration vector $\cc$, i.e., we assume that
\begin{align}
 s &= s_e^*(e) + s_{\cc}^*(\cc),
 \label{hp.s.dec}
\end{align}
where $s_e^*$: $\R_{+,0}\to \R_{+,0}$ and $s_{\cc}^*$: $\R^L \to \R_{+,0}$ are strictly concave $C^2$  functions.   Concerning $s^*_{\cc}$, we assume that for all $\cc$, $\vec{x}\in \R^L$ there holds
\begin{align}
  -\sum_{i,j=1}^L x_i x_j\pa^2_{c_i c_j}s_{\cc}^*(\cc) &\geq C |\vec{x}|^2 \label{hp.sc.1}\\
  \intertext{and that for all $K>0$ there exists $\eps>0$ such that for all $\cc\in \R^L$ and all $i=1,\ldots, L$ we have}
 |\pa_{\cc} s_{\cc}^*(\cc)|&\leq K \implies c_i\geq\eps. \label{hp.sc.2}
\end{align}
This assumption was firstly used in \cite{BuHa15} and plays the key role in proving the minimum principle for chemical concentration. The singularity of the entropy derivative at $0$ prevents each concentration from vanishing.
A model example for $s_{\cc}$ used frequently in practice is
\begin{align*}
& s_{\cc}^*(\cc) = \sum_{i=1}^L \big(c_i - c_i\log(c_i)\big),
\end{align*}
where $0<c_i<1$, $i=1,\dots,L$. Indeed, this function may be extended to $\R^L$ fulfilling all the assumptions above.

For $s_e^*$ we assume that it is a strictly increasing nonnegative function fulfilling for all $e>1$
\begin{align}
  C_1 &\leq -\frac{(s_e^*)''(e)}{(s_{e}^*)'(e)^{2}} \leq C_{2}. \label{hp.se.1}\\
  \intertext{In addition, concerning its behavior near zero, we assume that}
\lim_{e\to 0_{+}}\frac{1}{(s_{e}^*)(e)}&=\lim_{e\to 0_{+}}(s_e^*)'(e) =\lim_{e\to 0_+} -\frac{(s^*_e)''(e)}{(s^*_{e})'(e)^{2}} = \infty . \label{hp.se.2}
\end{align}
It also follows from \eqref{hp.se.1} and \eqref{hp.se.2} that $\theta^*(0)=0$ and that the heat capacity of the fluid vanishes as the temperature tends to zero, which is nothing else but the \emph{Third law of thermodynamics}. Furthermore, we see that for $e>1$ (with possibly changed constants $C_1$ and $C_2$)
\begin{align}
& C_1\le \frac{\theta^*(e)}{e}\le C_2 \label{prop.se}
\end{align}
and in addition (again for some possibly enlarged $C_2>0$) that for all $e\ge 0$
\begin{align}
& e-2s_e^*(e) +C_2\ge 0.\label{prop.se2}
\end{align}
A possible example fulfilling \eqref{hp.se.1}--\eqref{prop.se2} is
$$
s_e^*(e) = \left\{
\begin{array}{rl}
C_1 +C_2 \ln(e+C_3), & e> 1\\
C_4e^a, & 0<e\leq 1,
\end{array}
\right.
$$
where $0<a<1$  and $C_1$, \dots, $C_4$ are suitably chosen constants ensuring the smoothness of $s_e(e)$.

\subsubsection*{Second law of thermodynamics}
Having introduced all constitutive relations and assumptions, we focus finally on the validity of the Second law of thermodynamics. By replacing the terms on the right-hand side of \eqref{eq.s} by the corresponding terms from \eqref{qc}, \eqref{qe} and \eqref{S} and recalling that $\Div \vv=0$, we get
\begin{equation}\label{eq.s.2}
\begin{split}
 \pa_t s &+ \Div\left( s \vv -\sum_{i=1}^L\zeta_i \vc{q}_{\cc}^i +\frac{\vc{q}_e}{\theta}\right)= -\vzeta\cdot \vec{r} + \frac{\mathcal{S}: \mathcal{D}(\vv)}{\theta}\\
 &+\mathfrak{M}\left( \na\vzeta + \frac{\vec{z}}{\theta}\na\varphi \right) \cdot\left( \na\vzeta + \frac{\vec{z}}{\theta}\na\varphi \right) +\frac{\kappa|\na\theta|^2}{\theta^2}\ge 0,
\end{split}
\end{equation}
where the inequality follows from  assumptions \eqref{hp.r}, \eqref{hp.kappa}, \eqref{hp.M} and \eqref{hp.S.3}.

\subsection{Boundary and initial conditions}

First, we characterize the reasonable assumptions on the initial data, where the function spaces are chosen so that the initial energy and the entropy is finite.
\begin{equation}
\begin{array}{l}
\vec{c}^{\ 0}\in L^{\infty}([0,1]^L), \quad c^0_i > 0 \textrm{ a.e. in } \Omega \quad  \forall i \in \{1,2,\dots, L\}, \quad \vec{c}^{\ 0}\cdot \vec{\ell} = 1 \textrm{ a.e. in } \Omega,\\
\vv^0\in L^2_{0,\Div}(\Omega)\\
e^0\in L^1(\Omega).\label{IC.a}
\end{array}
\end{equation}
In what follows, we will use the following notation
$$
G:=\{\vec{x}\in \R^L; \; \forall i=1,2,\dots,L, \; x_i> 0, \; \vec{x}\cdot \vec{\ell}=1\}.
$$
Therefore, we may restate the first constraint in \eqref{IC.a} as
$$
\cc^{\ 0} \in G.
$$

Next,  we will introduce the assumptions on the boundary data on $\Gamma$. Concerning \eqref{bc.vA}, we use the structure of $\mathcal{T}$ and we see that it reduces to
\begin{align}
 \vv\cdot\nu &= 0,\quad (\mathcal{I} - \nu\otimes\nu)\mathcal{S}\nu = -\gamma \vv\quad \mbox{on }\Gamma, \label{bc.v}
 \end{align}
 where $\gamma=\gamma^*(\cc,\theta)$ is a nonnegative continuous function fulfilling for all $(c,\theta)\in \R^L\times \R_{+,0}$
 $$
 0\le \gamma^*(c,\theta)\le C_2.
 $$
Next, for the particular fluxes, we assume that the fluxes are proportional to a differences of corresponding quantities in the fluid and outside the domain $\Omega$, i.e.,
\begin{align}
 q_{\cc\,\Gamma}^i = q_{\cc\,\Gamma}^{*,i} (x,\cc,\theta,\vzeta,\varphi)&= \sum_{j=1}^L D_{ij}^*(x,\cc,\theta)\left( \zeta_j - \zeta_j^\Gamma(x) + \frac{z_j}{\theta^{\Gamma}(x)}(\varphi - \varphi^\Gamma(x))\right)
 &&\mbox{on }\Gamma,\label{bc.qc}\\
 q_{e\Gamma}= q_{e\Gamma}^*(x,\cc,\theta)&= -\kappa^*_{\Gamma}(x,\cc,\theta) \left( \frac{1}{\theta} - \frac{1}{\theta^\Gamma(x)} \right)
 &&\mbox{on }\Gamma,\label{bc.qe}\\
 q_{\varphi\Gamma}= q_{\varphi\Gamma}^*(x,\varphi)&= -\lambda^{\Gamma}(x)(\varphi - \varphi^{\Gamma}(x)) &&\mbox{on }\Gamma.\label{bc.phi}
\end{align}
Here $\mathfrak{D} = \mathfrak{D}^*(x,\cc,\theta)= \{D^*_{ij}(x,\cc,\theta)\}_{i,j=1}^L$: $\partial \Omega \times \R^L \times \R_{+,0} \to \R^{L\times L}$ is a Carath\'{e}odory mapping fulfilling for some measurable nonnegative $d$: $\partial \Omega \to \R_{,+0}$ almost all $x\in \partial \Omega$ and all $(\cc,\vec{w},\theta)\in \R^L \times \R^L \times \R_{+,0}$ the following inequality
\begin{align}\label{as.D}
 C_1 d(x)|\mathfrak{P}_{\vec{\ell}} \,\vec{w}|^2 \leq \sum_{i,j=1}^L D^*_{ij}(x,\cc,\theta)w_i w_j \leq C_2 d(x)|\mathfrak{P}_{\vec{\ell}}\, \vec {w}|^2.
\end{align}
Moreover, for $(x,\cc,\theta) \in \Omega \times \R^L\times \R_{+,0}$
\begin{equation} \label{as.D2}
\sum_{i=1}^L D_{ij}^*(x,\cc,\theta) = 0,
\end{equation}
and the non-negative function $d$ satisfies
\begin{equation}\label{ass.d}
\int_{\partial \Omega} d(x) \dS>0.
\end{equation}
The function $\kappa_{\Gamma}^*(x,\cc,\theta)$: $\partial \Omega \times \R^L \times \R_{+,0} \to \R_{+,0}$ is a nonnegative Carath\'{e}odory function satisfying
\begin{equation}
\label{as.kappa.G}
C_1 \overline{\kappa}(x)\le \kappa_{\Gamma}^*(x,\cc,\theta) \le C_2\overline{\kappa}(x)
\end{equation}
with some measurable nonnegative function $\overline{\kappa}$: $\partial \Omega \to \R_{+,0}$ fulfilling
\begin{equation}\label{as.kappa.2}
\int_{\partial \Omega} \overline{\kappa}(x)\dS >0.
\end{equation}
Finally, we require  $\lambda^{\Gamma}\in C^{1,1}(\partial \Omega)$ to be a nonnegative function that satisfies
\begin{equation}
\int_{\partial \Omega} \lambda^{\Gamma}(x) \dS >0. \label{as.lambda}
\end{equation}
We conclude by making the following assumptions on the regularity of the boundary data
\begin{equation}\label{bc.spaces}
(\theta^{\Gamma})^{-1}  \in L^2(\Gamma), \quad
\mathfrak{P}_{\vec{\ell}}\,\vzeta^{\ \Gamma}  \in L^2(\Gamma;\mathbb{R}^L), \quad
\varphi^{\Gamma}  \in C^{1,1}(\partial \Omega).
\end{equation}

Let us point out at this place that we assume on purpose that $\overline{\kappa}$ and $d$ can vanish on some part of the boundary, while they must be strictly positive on a set of nonzero measure. This should model two different phenomena. First, when the boundary consists of a wall preventing any energy, ion or mass transfer, and the second one, where there still is no mass transfer, but the transfer of the heat and the concentrations is driven by the difference of the corresponding quantities.


%
%

We now collect all our assumptions in order to have all of them at one place. Recall we consider system \eqref{eq.c}--\eqref{eq.E} with initial conditions \eqref{IC} and boundary conditions \eqref{bc.vA}--\eqref{bc.phiA}. \newline
{\bf Hypothesis (H1): chemical reactions} \newline
We assume that $\vec{r}= \vec{r}^{\,*}(\cc,\theta,\vzeta)$ is a bounded continuous function on $\R^L\times \R_{+,0} \times \R^L$, and
$$
\sum_{i=1}^L r_i^* =  \sum_{i=1}^L r_i^*z_i = 0
$$
such that for all $\vzeta \in \R^L$ it holds $\vec{r}^{\,*}\cdot \vzeta \leq 0$. \newline
{\bf Hypothesis (H2):  stress tensor} \newline
We assume that ${\mathcal S}= {\mathcal S}^*(\cc,\theta,{\mathcal D})$ is a continuous function  on $\R^L\times \R_{+,0} \times \R^{3\times 3}_{sym}$ such that  for some $r\in (1,\infty)$ and all $(\cc,\theta, {\mathcal D},{\mathcal B})\in \R^L \times \R_{+,0} \times \R^{3\times 3}_{sym}\times \R^{3\times 3}_{sym}$ we have
\begin{align*}
  {\mathcal S}^*(\cc,\theta,{\mathcal D}):{\mathcal D} &\geq C_1 |{\mathcal D}|^r-C_2,\\
  |{\mathcal S}^*(\cc,\theta,{\mathcal D})|&\leq C_2(1+|{\mathcal D}|^{r-1}),\\
  {\mathcal S}^*(\cc,\theta,0)&=0,\quad ({\mathcal S}^*(\cc,\theta,{\mathcal D})-{\mathcal S}^*(\cc,\theta,{\mathcal B}))\cdot({\mathcal D}-{\mathcal B})\geq 0.
\end{align*}
{\bf Hypothesis (H3):  flux of concentrations and heat flux} \newline
We assume that
$$
\vc{q}_{\cc}^i = \vc{q}_{\cc}^{*,i}(\cc,\theta,\na\vzeta, \na\theta,\na\varphi) := -\sum_{j=1}^L M_{ij}^*(\cc,\theta)\left( \na\zeta_j + \frac{z_j}{\theta}\na\varphi \right) - m_i^*(\cc,\theta)\na\frac{1}{\theta}
$$
and
$$
 \vc{q}_e=\vc{q}_e^*(\cc,\theta,\na\vzeta, \na\theta,\na\varphi) := -\kappa^*(\cc,\theta)\na\theta - \sum_{i=1}^L m_i^*(\cc,\theta)\left( \na\zeta_i + \frac{z_i}{\theta}\na\varphi \right)
$$
with the matrix-valued  function
$\mathfrak{M}=\mathfrak{M}^*(\cc,\theta)= \{M_{ij}^*(\cc,\theta)\}_{i,j=1}^L$: $\R^L\times \R_{+,0} \to \R^{L\times L}_{sym}$  and the vector-valued function $\vec{m} = \vec{m}^{\,*}(\cc,\theta) = \{m_i^*(\cc,\theta)\}_{i=1}^L$: $\R^L\times \R_{+,0} \to \R$ continuous mappings  such that for all $(\cc,\theta)\in \R^L\times \R_{+,0}$,
$$
 \sum_{i=1}^L M_{ij}^*(\cc,\theta) = \sum_{i=1}^L m_i^*(\cc,\theta) = 0,\qquad \textrm{ for all } j=1,\ldots,L,
$$
for all $\vec{w}\in \R^L$
$$
 C_1M(\theta)|\mathfrak{P}_{\vec{\ell}}\, \vec{w}|^2\leq\sum_{i,j=1}^L M_{ij}^*(\cc,\theta)w_i w_j\leq C_2M(\theta)|\mathfrak{P}_{\vec{\ell}}\, \vec{w}|^2,
$$
and  for some $\varepsilon_0>0$
$$
\begin{aligned}
C_1\min(1,\theta^{\beta-\varepsilon_0})&\le M(\theta)\le C(1+\theta)^{\frac53 - \varepsilon_0}, \\
|\vec{m}^{\, *}(\cc,\theta)|^2&\leq C_2\left\{\begin{aligned}
&\min\{M(\theta) \theta^{-\beta+\varepsilon_0},\theta^{-2(\beta-1)+\varepsilon_0}\}&&\textrm{for }\theta<1,\\
&M(\theta)\theta &&\textrm{for } \theta\ge 1.
\end{aligned}\right.
\end{aligned}
$$
{\bf Hypothesis (H4):  heat conductivity} \newline
We assume that $\kappa=\kappa^*(\cc,\theta)$: $\R^L\times \R_+ \to \R_+$ is a continuous function satisfying for some $\beta \in [0,2]$  and all $(\cc,\theta) \in \R^L\times \R_+$
$$
0< C_1 \leq \frac{\kappa^*(\cc,\theta)}{1 + \theta^{-\beta}}\leq C_2.
$$
{\bf Hypothesis (H5):  entropy} \newline
We assume that $s=s^*(e,\cc) = s_e^*(e) + s_{\cc}^*(\cc)$, where $s_e^*$: $\R_{+,0} \to \R_{+,0}$ and $s_{\cc}^*$: $\R^L\to \R_{+,0}$ are strictly concave $C^2$ functions. Moreover, for all $\cc$, $\vec{x} \in \R^L$
$$
  -\sum_{i,j=1}^L x_i x_j\pa^2_{c_i c_j}s_c^*(\cc) \geq C |\vec{x}|^2,
$$
and  for all $K>0$ there exists $\eps>0$ such that for all $\cc\in \R^L$ and all $i=1,\ldots, L$ we have
$$
 |\pa_{\cc} s_{\cc}^*(\cc)|\leq K \implies c_i\geq\eps.
$$
For $s_e^*$ we assume that it is a strictly increasing function fulfilling for all $e>1$
\begin{align*}
  C_1 &\leq -\frac{(s_e^*)''(e)}{(s_{e}^*)'(e)^{2}} \leq C_{2}\\
  \intertext{and that its behaviour near zero is described by the following limits as}
\lim_{e\to 0_{+}}\frac{1}{(s_{e}^*)(e)}&=\lim_{e\to 0_{+}}(s^*_{e})'(e) =\lim_{e\to 0_+} -\frac{(s^*_e)''(e)}{(s^*_{e})'(e)^{2}} = \infty.
\end{align*}
{\bf Hypothesis (H6): initial conditions}
$$
\cc^{\ 0} \in G, \quad
\vv^0\in L^2_{0,\Div}(\Omega), \quad
e^0\in L^1(\Omega).
$$
{\bf Hypothesis (H7): boundary conditions for the velocity} \newline
We assume that $\gamma=\gamma^*(\cc,\theta) $ is a nonnegative continuous functions on $\R^L \times \R_{+,0}$ such that
$$
0\leq \gamma^*(\cc,\theta) \leq C_2
$$
for all $(\cc,\theta)\in \R^L\times \R_{+,0}$.
\newline
{\bf Hypothesis (H8):  boundary conditions for the flux of concentrations} \newline
$$
q_{\cc\,\Gamma}^i = q_{\cc\,\Gamma}^{*,i}(x,\cc,\theta,\vzeta,\varphi):= \sum_{j=1}^L D^*_{ij}(x,\cc,\theta)\left( \zeta_j - \zeta_j^\Gamma + \frac{z_j}{\theta^{\Gamma}}(\varphi - \varphi^\Gamma)\right)
$$
with
$\mathfrak{D}=\mathfrak{D}^*=\{D^*_{ij}(x,\cc,\theta)\}_{i,j}^L$: $\partial \Omega \times \R^L \times \R_+ \to \R^{L\times L}$  a Carath\'{e}odory mapping fulfilling for some  nonnegative $d\in L^{\infty}(\partial \Omega)$, almost all $x\in \partial \Omega$ and all $(\cc,\vec{w},\theta)\in \R^L \times \R^L \times \R_{+,0}$ the following inequality
$$
 C_1 d(x)|\mathfrak{P}_{\vec{\ell}}\, \vec{w}|^2 \leq \sum_{i,j=1}^L D^*_{ij}(x,\cc,\theta)w_i w_j \leq C_2 d(x)|\mathfrak{P}_{\vec{\ell}} \, \vec{w}|^2.
$$
Moreover,
$$
\sum_{i=1}^L D^*_{ij} =0, \qquad \int_{\partial \Omega} d(x) \dS>0
$$
and
$$
(\theta^{\Gamma})^{-1}  \in L^2(\Gamma), \quad
\mathfrak{P}_{\vec{\ell}}\, \vzeta^{\ \Gamma}  \in L^2(\Gamma;\mathbb{R}^L),
\quad \varphi^{\Gamma} \in C^{1,1}(\partial \Omega).
$$
{\bf Hypothesis (H9): boundary conditions for the flux of internal energy} \newline
$$
 q_{e\Gamma}=q_{e\Gamma}^*(x,\cc,\theta):= -\kappa^*_{\Gamma}(x,\cc,\theta) \left( \frac{1}{\theta} - \frac{1}{\theta^\Gamma(x)} \right).
$$
The function $\kappa_{\Gamma}^*(x,\cc,\theta)$: $\partial \Omega \times \R^L \times \R_{+,0} \to \R_{+,0}$ is a nonnegative Carath\'{e}odory function satisfying
$$
C_1 \overline{\kappa}(x)\le \kappa_{\Gamma}^*(x,\cc,\theta) \le C_2\overline{\kappa}(x),
$$
where $\overline \kappa \in L^\infty(\partial \Omega)$, and
$$
\int_{\partial \Omega} \overline{\kappa}(x)\dS >0.
$$
{\bf Hypothesis (H10):  boundary conditions for the flux of the electric potential} \newline
$$
 q_{\varphi\Gamma}= q_{\varphi\Gamma}^*(x,\varphi):= -\lambda^{\Gamma}(x)(\varphi - \varphi^{\Gamma}(x)),
$$
where $\lambda^{\Gamma}\in C^{1,1}(\partial \Omega)$ is a nonnegative function that satisfies
$$
\int_{\partial \Omega} \lambda^{\Gamma}(x) \dS >0.
$$
Note that, for the sake of simplicity, we assume that all quantities depend explicitly only on $x$, but nor on $t$. We could, indeed, assume also the dependence on time, however, we prefer to omit it in order to simplify the presentation.

\subsection{Auxiliary results.}

Here we present some results that will be helpful in the following.

\begin{lemma}[Korn inequality]\label{lem.Korn}
Let $\Omega\subset\R^3$ be an open bounded Lipschitz domain, $p\in (1,\infty)$. Then:
 \begin{equation} \label{Korn1}
  \|\vv\|_{W^{1,p}(\Omega)}\leq C(p,\Omega)\|\mathcal{D}(\vv)\|_{L^p(\Omega)}\qquad\forall \vv\in W_0^{1,p}(\Omega;\R^3),
 \end{equation}
 \begin{equation} \label{Korn2}
  \|\vv\|_{W^{1,p}(\Omega)}\leq C(p,\Omega)\|\mathcal{D}(\vv)\|_{L^p(\Omega)} + \|\vv\|_{L^p(\Omega)} \qquad\forall \vv\in W_{0,\nu}^{1,p}(\Omega).
 \end{equation}
If, additionally, $\Omega$ is not axially symmetric, then \eqref{Korn1} holds also for  $v\in W_{0,\nu}^{1,p}(\Omega)$.
\end{lemma}
\begin{proof}
 See e.g. \cite{NoSt:2004}.
\end{proof}

\begin{lemma}[Interpolation inequality]\label{lem.interp}
Let $\Omega\subset\R^3$ be an open bounded Lipschitz domain, $p\in [\frac{6}{5},\infty)$ and $q\ge 2$. Then for all $u\in W^{1,p}(\Omega)$
\begin{align*}
 \|u\|_{L^q(\Omega)}^q\leq C(p,q,\Omega)\|u\|_{L^2(\Omega)}^{\alpha}\|u\|_{W^{1,p}(\Omega)}^{1-\alpha},
\end{align*}
provided that $\alpha \in [0,1]$ and $\frac 1q = \frac \alpha 2 + (1-\alpha)(\frac 1p-\frac 13)$.
\end{lemma}
\begin{proof}
 See \cite{Adams}.
\end{proof}

\begin{lemma}[Generalized Aubin's lemma]\label{lem.aubin}
Let $\{y_n\}_{n\in\N}$, $\{z_n\}_{n\in\N}$ be sequences in $L^{\infty}(\OmT)$ such that:
\begin{align}
 &\|y_n\|_{L^{\infty}(\OmT)} + \|z_n\|_{L^{\infty}(\OmT)} \leq C,\label{hp.aub.1}\\
 &\|y_n\|_{L^{2}(0,T; W^{1,2}(\Omega))} + \|y_n z_n\|_{L^{2}(0,T; W^{1,2}(\Omega))}
 + \|\pa_t z_n\|_{L^{2}(0,T; W^{-1,2}(\Omega))} \leq C,\label{hp.aub.2}\\
 & y_n\to y\quad\mbox{strongly in }L^2(\OmT)\mbox{ (up to subsequences),}\label{hp.aub.3}\\
 & z_n\rightharpoonup^* z\quad\mbox{weakly* in }L^\infty(\OmT)\mbox{ (up to subsequences).}\label{hp.aub.4}
\end{align}
Then:
\begin{align}
 y_n z_n\to yz\quad\mbox{strongly in }L^q(\OmT)\mbox{ (up to subsequences),}~~~\forall q\in [1,\infty).\label{lem.aub.stat}
\end{align}
\end{lemma}
\begin{proof}
The proof is almost identical to the one of \cite[Lemma 13]{Jue14}. The only difference is that the expression
$\|z_n(\cdot,\cdot + k) - z_n\|_{L^2(0,T; W^{-1,2}(\Omega))}$ (with $k\in\R$ arbitrary) appearing in the terms $J_1$, $J_2$
can be controlled by $k\|\pa_t z^{(n)}\|_{L^2(0,T; W^{-1,2}(\Omega))}$.
\end{proof}

\begin{lemma}[Lemma~3.6 in \cite{BuHa15}]
Let $s_{\cc}^*$ satisfy \eqref{hp.sc.1} and define $\vzeta^*(c):=-\partial_{\cc} s_{\cc}^*(\cc)$. Then
there exists $C >$ 0 such that for all $\cc\in [0,1]^L$ there holds
\begin{equation}
\begin{split}
\vzeta^{\,*}_i(\cc)\le C, \qquad |\vzeta^{\,*}(\cc)|\le \frac{C}{c_i} \quad \textrm{for all }i=1,\ldots,L.\label{as.zeta1}
\end{split}
\end{equation}
Moreover, there exists $C > 0$ such that for all $c\in G$, we have
\begin{equation}\label{as.zeta2}
|\vzeta^{\,*}(\cc)|\le C(1+|\mathfrak{P}_{\vec{\ell}}\,\vzeta^{\,*}(\cc)|).
\end{equation}
\end{lemma}

\section{A~priori estimates}
In this section we derive uniform estimates valid for any sufficiently smooth solution that also help us to introduce a proper notion of a solution.  In what follows, we omit writing dependence of quantities on spatial and time variables. From now on, the constant $C>0$ denotes some generic constant that can vary line to line, but can depend only on the data of our problem.

\subsection{Uniform bound for $\cc$ and $\varphi$}
We start with \eqref{eq.transport}. Since we assume, see \eqref{IC.a}, that $\cc(0)\cdot \vec{\ell} \equiv 1$ in $\Omega$, $\Div \vv=0$ and $\vv\cdot \nu =0$ on $\Gamma$, we get that $\cc\cdot \vec{\ell} \equiv 1$ almost everywhere in $\Omega_T$. Next, we assume that $\vzeta$ is well defined (possible unbounded), therefore it also follows from \eqref{hp.sc.2} that for all $i=1,\ldots, L$ we have $c_i \ge 0$ in $\Omega_T$ which implies $\cc \in L^\infty(\Omega_T;\R^L)$. To justify that
\begin{equation}
\cc\in G \textrm{ almost everywhere in } \Omega_T,\label{c.in.G}
\end{equation}
we need to verify that in fact $c_i >0$ for all $i \in \{1,2,\dots,L\}$ a.e. in $\Omega_T$. We show this at the end of this section by proving that $\vzeta$ is finite a.e.

Furthermore, using the fact that $Q:=\vec{z}\cdot \cc$, we get
\begin{equation}\label{ap.Q}
\sup_{t\in (0,T)} \|Q\|_{\infty} \le C(\vec{z}).
\end{equation}
Using now the theory for the Laplace equation \eqref{eq.poi}, the fact that $\Omega \in C^{1,1}$ and the  assumptions on $\lambda^\Gamma$ and $\varphi^\Gamma$ from Hypotheses (H8) and (H10), we deduce that for all $q\in[1,\infty)$
\begin{equation}\label{ap.phi.1}
\sup_{t\in (0,T)} \|\varphi\|_{2,q} \le C(\vec{z},q,\varphi^{\Gamma},\lambda^\Gamma).
\end{equation}

\subsection{Uniform bounds based on the total energy and the entropy}
First, we focus on information coming from \eqref{eq.E}. Thus, we integrate \eqref{eq.E} over $\Omega$ and use integration by parts to obtain (note here that several boundary integrals vanish due to \eqref{bc.v})
\begin{equation*}
\frac{{\rm d}}{{\rm d}t}\int_{\Omega} E \dx  +\int_{\partial \Omega} \Big(\varphi\sum_{i=1}^L z_i \vc{q}_{\cc}^i  + \vc{q}_e- \mathcal{T} \vv -\varphi\na\pa_t\varphi\Big)
 \cdot \nu \dS= 0.
\end{equation*}
Using \eqref{Cauchy}, \eqref{bc.v}, \eqref{bc.phi} and the symmetry of $S$, we find
$$
\begin{aligned}
 \int_{\partial \Omega} -\mathcal{T} \vv \cdot \nu\dS&=\int_{\partial \Omega} p\vv\cdot \nu -\mathcal{S}\vv \cdot \nu\dS\\
 &=\int_{\partial \Omega} -\mathcal{S}\nu \cdot \vv\dS=\int_{\partial \Omega} \gamma|\vv|^2 -\mathcal{S}(\nu\otimes \nu)\nu \cdot \vv\dS\\
 &=\int_{\partial \Omega} \gamma|\vv|^2 \dS, \\
 -\int_{\partial \Omega}\varphi\na\pa_t\varphi \cdot \nu \dS
&= \int_{\partial \Omega}\varphi \lambda^\Gamma \pa_t(\varphi-\varphi^{\Gamma}) \dS = \frac12 \frac{{\rm d}}{{\rm d}t} \int_{\partial \Omega}\lambda^{\Gamma}|\varphi|^2\dS.
\end{aligned}
$$
Hence, using also Hypotheses (H8) and (H9) we obtain
\begin{equation}
\begin{split}\label{eq.TE}
\frac{{\rm d}}{{\rm d}t}&\left(\int_{\Omega} E \dx+\int_{\partial \Omega} \frac{\lambda^\Gamma |\varphi|^2}{2}\dS \right)  +\int_{\partial \Omega} \gamma|\vv|^2\dS\\
&= \int_{\partial \Omega}\kappa_{\Gamma} \left( \frac{1}{\theta} - \frac{1}{\theta^\Gamma} \right) -\sum_{i,j=1}^L D_{ij} z_i \left( \zeta_j - \zeta_j^\Gamma + \frac{z_j}{\theta^\Gamma}(\varphi - \varphi^\Gamma)\right) \varphi \dS.
\end{split}
\end{equation}

Similarly, integration \eqref{eq.s.2} over $\Omega$ leads after using integration by parts and the fact that $\vv\cdot \nu$ vanishes on $\partial \Omega$ to
\begin{equation}\label{eq.s.3}
\begin{split}
\frac{{\rm d}}{{\rm d}t} & \int_{\Omega} -s \dx +\int_{\Omega} \frac{\mathcal{S}: \mathcal{D}(\vv)}{\theta} -\vzeta\cdot \vec{r}\dx \\
 &+\int_{\Omega}\mathfrak{M}\left( \na\vzeta + \frac{\vec{z}}{\theta}\na\varphi \right) \cdot\left( \na\vzeta + \frac{\vec{z}}{\theta}\na\varphi \right) +\frac{\kappa|\na\theta|^2}{\theta^2}\dx\\
 &=\int_{\partial \Omega} -\sum_{i=1}^L\zeta_i \vc{q}_{\cc}^i\cdot \nu  +\frac{\vc{q}_e\cdot \nu }{\theta}\dS\\
 &=-\int_{\partial \Omega}  \mathfrak{D}\left( \vzeta - \vzeta^{\ \Gamma} + \frac{\vec{z}}{\theta^\Gamma}(\varphi - \varphi^\Gamma)\right)\vzeta +\kappa_{\Gamma} \left( \frac{1}{\theta} - \frac{1}{\theta^\Gamma} \right)\frac{1}{\theta}\dS.
\end{split}
\end{equation}
Hence, summing \eqref{eq.TE} and \eqref{eq.s.3} and reordering the corresponding terms we find
\begin{equation}
\begin{split}\label{eq.TE.s}
&\frac{{\rm d}}{{\rm d}t}\left(\int_{\Omega} E-s \dx+\int_{\partial \Omega} \frac{\lambda^\Gamma|\varphi|^2}{2}\dS \right)  +\int_{\partial \Omega} \gamma|\vv|^2\dS+\int_{\Omega} \frac{\kappa|\na\theta|^2}{\theta^2}\dx \\
&\quad +\int_{\Omega} \frac{\mathcal{S}: \mathcal{D}(\vv)}{\theta} -\vzeta\cdot \vec{r}  +\mathfrak{M}\left( \na\vzeta + \frac{\vec{z}}{\theta}\na\varphi \right) \cdot\left( \na\vzeta + \frac{\vec{z}}{\theta}\na\varphi \right)\dx \\
&\quad  +\int_{\partial \Omega}\kappa_{\Gamma} \left| \frac{1}{\theta} - \frac{1}{\theta^\Gamma} \right|^2 +\mathfrak{D}\left(\vzeta-\vzeta^{\, \Gamma}+ (\varphi -\varphi^{\Gamma})\frac{\vec{z}}{\theta^\Gamma}\right)\cdot \left(\vzeta-\vzeta^{\ \Gamma}+ (\varphi -\varphi^{\Gamma})\frac{\vec{z}}{\theta^\Gamma}\right) \dS\\
&= -\int_{\partial \Omega}\kappa_{\Gamma} \left( \frac{1}{\theta} - \frac{1}{\theta^\Gamma} \right)\left(\frac{1}{\theta^{\Gamma}}-1\right) +\mathfrak{D}\left(\vzeta^{\,\Gamma}+\varphi^{\Gamma}\frac{\vec{z}}{\theta^\Gamma}\right)\cdot  \left( \vzeta - \vzeta^{\,\Gamma} + \frac{\vec{z}}{\theta^\Gamma}(\varphi - \varphi^\Gamma)\right) \dS \\
&\quad + \int_{\partial \Omega}\mathfrak{D}  \left( \vzeta - \vzeta^{\,\Gamma} + \frac{\vec{z}}{\theta^\Gamma}(\varphi - \varphi^\Gamma)\right)\cdot \vec{z} \varphi \Big(1-\frac{1}{\theta^\Gamma}\Big) \dS.
\end{split}
\end{equation}
This inequality is the starting point for getting a~priori estimates. First, due to Hypothesis (H1), 
the term with $\vec{r}$  on the left-hand side is nonnegative and therefore will be neglected.  Moreover, since the matrix $\mathfrak{D}$ is positive definite on the range of $\mathfrak{P}_{\vec{\ell}}$, we can use the Cauchy--Schwarz inequality to absorb the corresponding terms on the right-hand side by the left-hand side. To estimate the term with the heat conductivity, we use Hypothesis (H4), 
for the term involving the matrix $\mathfrak{M}$ we use Hypothesis (H3). 
Finally, for the boundary integral, we use the H\"{o}lder inequality and also Hypotheses (H5), (H8), (H9)
and the uniform bound \eqref{ap.phi.1} to obtain
\begin{equation}
\begin{split}\label{eq.TE.s2}
&\frac{{\rm d}}{{\rm d}t}\left(\int_{\Omega} E-s+C_2 \dx+\int_{\partial \Omega} \frac{\lambda^\Gamma|\varphi|^2}{2}\dS \right)  +\int_{\partial \Omega} \gamma|\vv|^2\dS \\
&\quad +\int_{\Omega} \frac{\mathcal{S}: \mathcal{D}(\vv)}{\theta}+ C_1M(\theta)\left|\mathfrak{P}_{\vec{\ell}}\left( \na\vzeta + \frac{\vec{z}}{\theta}\na\varphi \right)\right|^2+C_1\frac{(1+\theta^{-\beta})|\na\theta|^2}{\theta^2}\dx \\
&\quad  +C \int_{\partial \Omega}\overline{\kappa}\left| \frac{1}{\theta} - \frac{1}{\theta^\Gamma} \right|^2 +d\left|\mathfrak{P}_{\vec{\ell}}\left(\vzeta-\vzeta^{\,\Gamma}+ (\varphi -\varphi^{\Gamma})\frac{\vec{z}}{\theta^\Gamma}\right)\right|^2\dS\\
&\le C \Big( \int_{\partial \Omega}\overline{\kappa}\left|\frac{1}{\theta^{\Gamma}}-1\right|^2 +d\left|\mathfrak{P}_{\vec{\ell}}\left(\vzeta^{\,\Gamma}+\varphi^{\Gamma}\frac{\vec{z}}{\theta^\Gamma}\right)\right|^2 
\dS +1\Big).
\end{split}
\end{equation}
Finally, using Hypotheses (H8), (H9) and (H10) and the assumptions
on the initial data \eqref{IC.a}, the Young inequality and the estimate \eqref{ap.phi.1}, we have (notice that all terms appearing on the left-hand side are nonnegative)
\begin{equation}
\begin{split}\label{eq.TE.s3}
&\sup_{t\in (0,T)} \int_{\Omega} E(t)-s(t)+C_2 \dx +\int_{\Gamma} \gamma|\vv|^2+\overline{\kappa}\theta^{-2} +d\left|\mathfrak{P}_{\vec{\ell}}\vzeta\right|^2\dS \dt \\
&\quad +\int_{\Omega_T} \frac{\mathcal{S}: \mathcal{D}(\vv)}{\theta}+ M(\theta)\left|\mathfrak{P}_{\vec{\ell}}\left( \na\vzeta + \frac{\vec{z}}{\theta}\na\varphi \right)\right|^2+|\nabla \ln \theta|^2 + |\na \theta^{-\frac{\beta}{2}}|^2 \dx \dt \le C.
\end{split}
\end{equation}
Thus, employing Hypothesis (H5) 
and the definition of $E$, we see that
\begin{equation}
\label{ap.E}
\sup_{t\in (0,T)} \left( \|\vv(t)\|_2 + \|e(t)\|_1 + \|s(t)\|_1+\|\theta(t)\|_1\right) \le C.
\end{equation}
Next, using \eqref{ap.E}, Hypothesis (H9),
the fact that $\beta\le 2$ and the Poincar\'{e} inequality, we have
\begin{equation}
\int_{\Omega_T}\|\ln \theta\|_{1,2}^2+ \|\theta^{-\frac{\beta}{2}}\|^2_{1,2} \dt \le C. \label{ap.theta.1}
\end{equation}

\subsection{Uniform bounds based on the kinetic and the internal energy identity}

Integrating \eqref{eq.v2} over $\Omega$, using the fact that $\Div \vv=0$, assumptions \eqref{bc.vA} and \eqref{Cauchy}, we get
\begin{align}
 \frac{d}{dt}\int_{\Omega} \frac{|\vv|^2}{2}\dx  + \int_{\partial \Omega} \gamma|\vv|^2\dS + \int_{\Omega} \mathcal{S}: \mathcal{D}(\vv)\dx=-\int_{\Omega}Q \vv\cdot\na\varphi \dx.\label{eq.vv2}
\end{align}
Thus, applying Hypothesis (H2), 
the already obtained uniform bound \eqref{ap.Q} and \eqref{ap.phi.1}, and the standard interpolation inequality (see Lemma~\ref{lem.interp}), we deduce after integration over $(0,T)$
\begin{equation}
\int_0^T \|\sqrt{\gamma}\vv\|^2_{L^2(\partial \Omega)}+\|\vv\|_{\frac{5r}{3}}^{\frac{5r}{3}}+\|\vv\|_{1,r}^r + \|\mathcal{S}\|_{r'}^{r'}\dt \le C(T,\vv_0,\cc,\vec{z},\Omega). \label{ap.D.S}
\end{equation}
Consequently, following the procedure developed in \cite{BuMaRa09} and using the fact that $\Omega \in C^{1,1}$, we can decompose the pressure into four parts
$$
p=p_1+p_2+p_3+p_4,
$$
such that for all $\psi\in W^{2,\infty}(\Omega)$ fulfilling $\nabla \psi \cdot \nu =0$ on $\partial \Omega$ and for almost all $t\in (0,T)$
\begin{align}
\int_{\Omega} p_1 \Delta \psi \dx &= \int_{\Omega} \mathcal{S}: \nabla^2 \psi \dx,\label{def.p1}\\
\int_{\Omega} p_2 \Delta \psi \dx & = -\int_{\Omega} (\vv\otimes \vv) : \nabla^2 \psi\dx,\label{def.p2}\\
\int_{\Omega} p_3 \Delta \psi \dx &= \int_{\partial \Omega} \gamma \vv \cdot \nabla \psi \dS,\label{def.p3}\\
\int_{\Omega} p_4 \Delta \psi \dx & = \int_{\Omega} Q\nabla \varphi \cdot \nabla \psi\dx.\label{def.p4}
\end{align}
Consequently, it follows from \eqref{def.p1}--\eqref{def.p4} and the already stated a~priori estimates \eqref{ap.Q}, \eqref{ap.phi.1} and \eqref{ap.D.S} that (see \cite{BuMaRa09} for details) whenever $r>6/5$, we have\footnote{Indeed, we could (at least for $r>\frac 32$) improve the integrability of $p_3$, but we do not need it here.}
\begin{equation}
\sup_{t\in (0,T)}\|p_4(t)\|_{\infty} + \int_{\Omega_T} |p_1|^{r'}+ |p_2|^{\frac{5r}{6}} + |p_3|^{2} \dx\dt\le C.\label{ap.pres}
\end{equation}

Next, we introduce a nonnegative function $f(s) \in C^{\infty}(0,\infty)$, such that $|f(s)|\le 1$, which in addition will satisfy $f(s)=0$ for $s\in (0,1)$ and $f(s):= (1+s)^{-\lambda}$ for $s\ge 2$, where $\lambda \in (0,1)$ is arbitrary. Multiplying  \eqref{eq.e} by $f(e)$, integrating over $\Omega$ and using the fact that $\vv\cdot \nu=0$ on boundary and $\Div \vv=0$, we observe
\begin{equation*}
\begin{aligned}
&-\frac{{\rm d}}{{\rm d}t}\int_{\Omega}F(e)\dx  +\int_{\partial\Omega} f(e)\kappa^{\Gamma}\left( \frac{1}{\theta} - \frac{1}{\theta^\Gamma} \right)\dS+\int_{\Omega} f(e)\mathcal{S}: \mathcal{D}(\vv)\dx \\
 &-\int_{\Omega} f'(e)\left(\kappa\na\theta\cdot \nabla e + \sum_{i=1}^L m_i\left( \na\zeta_i + \frac{z_i}{\theta}\na\varphi \right)\cdot \nabla e\right)\dx \\
&+\int_{\Omega}  f(e)\sum_{i,j=1}^L  M_{ij}\left( \na\zeta_j + \frac{z_j}{\theta}\na\varphi \right)\cdot (z_i\nabla \varphi)  - \frac{f(e) \vec{m}\cdot \vec{z}}{\theta^2} \na\theta\cdot \na\varphi \dx =0,
\end{aligned}
\end{equation*}
where $F'=f$.
Next, using Hypotheses (H2), (H3) and (H9) 
together with the fact that $f(e)\le 1$  and the Cauchy--Schwarz and the Young inequalities, we find
\begin{equation}
\begin{aligned}
\label{eq.e1}
&-\frac{{\rm d}}{{\rm d}t}\int_{\Omega}F(e)\dx  +\lambda\int_{\{e\ge 2\}} \frac{\kappa\na\theta\cdot \nabla e }{(1+e)^{\lambda+1}}\dx -\int_{\{1\le e\le 2\}} f'(e)\kappa\na\theta\cdot \nabla e \dx\\
&\le C\int_{\partial\Omega} \overline{\kappa}\left| \frac{1}{\theta} - \frac{1}{\theta^\Gamma} \right|\dS+C\int_{\{e\ge 1\}} \frac{|\vec{m}|^2|\nabla e|^2}{M(\theta)(1+e)^{2\lambda+2}}\dx\\
&+C\int_{\{e\geq 1\}}  M(\theta)\left|\mathfrak{P}_{\vec{\ell}}\left( \na\vzeta + \frac{\vec{z}}{\theta}\na\varphi \right)\right|^2 +  M(\theta)|\mathfrak{P}_{\vec{\ell}}\,(\vec{z}\,\nabla \varphi)|^2  \dx \\
&+\int_{\{e\ge 1\}} \frac{\delta|\nabla \theta|^2}{(1+e)^{1+\lambda}}+  C(\delta^{-1})\frac{(1+e)^{1-\lambda}|\vec{m}\cdot \vec{z}|^2}{\theta^{4}} |\na\varphi|^2 \dx,
\end{aligned}
\end{equation}
where $\delta>0$ is arbitrary.
Then, the integration over $(0,T)$ and the use of Hypotheses (H4) and (H5)
and the already obtained estimates \eqref{ap.phi.1}, \eqref{eq.TE.s3} and \eqref{ap.E} together with the fact that $|F(e)|\le C(\lambda)(1+e)$ lead to (with a constant $C>0$ depending on the data and on $\lambda>0$)
\begin{equation*}
\begin{aligned}
&\int_{\{e\ge 2\}} \frac{|\na\theta|^2}{(1+\theta)^{\lambda+1}}\dx\dt \le C+C\int_{\{1\le e\le 2\}} |\nabla \theta|^2 \; \dx \dt\\
&+ C\int_{\{e\ge 2\}} \frac{|\vec{m}|^2|\nabla \theta|^2}{M(\theta)(1+\theta)^{2\lambda+2}}\dx\dt+C\int_{\{e\ge 1\}} M(\theta)\dx\dt \\
&+\int_{\{e\ge 2\}} \frac{C\delta|\nabla \theta|^2}{(1+\theta)^{1+\lambda}}+  \frac{C(\delta^{-1})|\vec{m}|^2}{(1+\theta)^{3+\lambda}} \dx\dt,
\end{aligned}
\end{equation*}
which, by using Hypothesis (H3),  can be simplified as
\begin{equation}
\begin{aligned}
\label{eq.e2}
&\int_{\{e\ge 2\}} \frac{|\na\theta|^2}{(1+\theta)^{\lambda+1}}\dx\dt \le C+C\int_{\{1\le e\le 2\}} |\nabla \theta|^2 \; \dx \dt\\
&+ C\int_{\{e\ge 2\}} \frac{|\nabla \theta|^2}{(1+\theta)^{1+\lambda}}\left(\delta+\frac{1}{(1+\theta)^{\lambda}}\right)+C(\delta^{-1})\int_{\{e\ge 1\}} M(\theta)\dx\dt\\
&\le C+C(\delta^{-1})\int_{\{e\ge 1\}} \frac{|\nabla \theta|^2}{\theta^2} + M(\theta)\dx \dt\\
&+ C\delta\int_{\{e\ge 2\}} \frac{|\nabla \theta|^2}{(1+\theta)^{1+\lambda}}\dx\dt.
\end{aligned}
\end{equation}
We see that for any fixed $\lambda\in (0,1)$ (so that the constant $C$ above, which may depend on $\lambda$ is now also fixed), we can find $\delta$ such that $C\delta\le 1/2$ and therefore the last integral in \eqref{eq.e2} can be absorbed by the left hand side. Thus, using in addition  Hypothesis (H3) and estimate \eqref{ap.theta.1}, we obtain
\begin{equation}
\begin{aligned}
\label{eq.e3}
&\int_{\Omega_T} \frac{|\na\theta|^2}{(1+\theta)^{\lambda+1}}\dx\dt \le C(\lambda)\left(1+\int_{\Omega_T} (1+\theta)^{\frac53-\varepsilon_0}\dx \dt\right).
\end{aligned}
\end{equation}
Finally, to estimate the right-hand side, we use  bound \eqref{ap.E}, the interpolation in Lebesgue spaces and the Sobolev embedding theorem. We have
$$
\begin{aligned}
&\int_{\Omega_T} (1+\theta)^{\frac53-\varepsilon_0}\dx \dt=\int_{0}^T \|(1+\theta)^{\frac{1-\lambda}{2}}\|^{\frac{2(5-3\varepsilon_0)}{3(1-\lambda)}}_{\frac{2(5-3\varepsilon_0)}{3(1-\lambda)}}\dt\\
&\quad \le C\int_{0}^T \|(1+\theta)^{\frac{1-\lambda}{2}}\|^{\frac{2(5-3\varepsilon_0)}{3(1-\lambda)}-\frac{2(2-3\varepsilon_0)}{2-3\lambda}}_{\frac{2}{1-\lambda}}
\|(1+\theta)^{\frac{1-\lambda}{2}}\|^{\frac{2(2-3\varepsilon_0)}{2-3\lambda}}_{1,2}\dt\\
&\quad \le C+C\int_{0}^T \left(\int_{\Omega}
\frac{|\nabla \theta|^2}{(1+\theta)^{1+\lambda}}\dx \right)^{\frac{2-3\varepsilon_0}{2-3\lambda}}\dt.
\end{aligned}
$$
Consequently, assuming that $\lambda<\min\{\varepsilon_0, \frac 23\}$ (note that we in fact use finally $\lambda$ close to zero, therefore the upper bounds do not play any important role), we can use this estimate in \eqref{eq.e3}, which after a direct application of the Young inequality leads to
\begin{equation}
\begin{aligned}
\label{ep.theta.1}
&\int_{\Omega_T} \frac{|\na\theta|^2}{(1+\theta)^{\lambda+1}}\dx\dt \le C(\lambda)\qquad \textrm{for all }\lambda>0.
\end{aligned}
\end{equation}
In addition, using the above computation, we get
\begin{equation}
\begin{aligned}
\label{ep.theta.2}
&\int_{\Omega_T} |\theta|^{\frac53 -\lambda}+|\nabla \theta|^{\frac{5}{4}-\lambda}+\frac{|\na\theta|^2}{(1+\theta)^{\lambda+1}}\dx\dt \le C(\lambda).
\end{aligned}
\end{equation}
Note that the estimate holds in fact for any $\lambda >0$, however, blows up when $\lambda \to 0^+$.

\subsection{A~priori estimates for fluxes} Take any $q\in (1,2)$. Using Hypothesis (H3), we see that by virtue of the Cauchy--Schwarz and the Young inequality the following estimate holds true
$$
\begin{aligned}
&\int_{\Omega_T} |\mathfrak{q}_{\cc}|^q \dx \dt \le C\int_{\Omega_T} \left|\mathfrak{M}
\mathfrak{P}_{\vec{\ell}}\left(\nabla\vzeta + \frac{\vec{z}}{\theta}\nabla \varphi\right)\right|^q + \frac{|\vec{m}|^q |\nabla \theta|^q}{\theta^{2q}} \dx \dt\\
&\le C\int_{\Omega_T} |M(\theta)|^{\frac{q}{2}}\left(M(\theta)\left|
\mathfrak{P}_{\vec{\ell}}\left(\nabla\vzeta + \frac{\vec{z}}{\theta}\nabla \varphi\right)\right|^2\right)^{\frac{q}{2}}\dx\dt\\
&\quad  + C\int_{\{\theta\ge 1\}}\frac{|\vec{m}|^q}{(1+\theta)^{2q-\frac{q(1+\lambda)}{2}}} \left(\frac{|\nabla \theta|^2}{(1+\theta)^{1+\lambda}}\right)^{\frac{q}{2}}\dx\dt+ C\int_{\{\theta < 1\}}\frac{|\vec{m}|^q}{\theta^{\frac{q(2-\beta)}{2}}} \left(\frac{|\nabla \theta|^2}{\theta^{\beta+2}}\right)^{\frac{q}{2}}\dx\dt\\
&\le C\int_{\Omega_T} |M(\theta)|^{\frac{q}{2-q}}+M(\theta)\left|
\mathfrak{P}_{\vec{\ell}}\left(\nabla \vzeta + \frac{\vec{z}}{\theta}\nabla \varphi\right)\right|^2\dx\dt\\
&\quad  + C\int_{\{\theta\ge 1\}}\frac{|\vec{m}|^\frac{2q}{2-q}}{(1+\theta)^{\frac{4q}{2-q}-\frac{q(1+\lambda)}{2-q}}} +\frac{|\nabla \theta|^2}{(1+\theta)^{1+\lambda}}\dx\dt+ \int_{\{\theta < 1\}}\frac{|\vec{m}|^\frac{2q}{2-q}}{\theta^{\frac{q(2-\beta)}{2-q}}} +\frac{|\nabla \theta|^2}{\theta^{\beta+2}}\dx\dt\\
&\le C(\lambda)+C\int_{\Omega_T} |M(\theta)|^{\frac{q}{2-q}}+ C\int_{\{\theta\ge 1\}}\frac{|\vec{m}|^\frac{2q}{2-q}}{(1+\theta)^{\frac{4q}{2-q}-\frac{q(1+\lambda)}{2-q}}}\dx\dt+ \int_{\{\theta < 1\}}\frac{|\vec{m}|^\frac{2q}{2-q}}{\theta^{\frac{q(2-\beta)}{2-q}}} \dx\dt,
\end{aligned}
$$
where for the last inequality we used \eqref{eq.TE.s3}. Employing also assumption \eqref{hp.m} and setting $\lambda$ sufficiently small, the above inequality reduces to
$$
\begin{aligned}
&\int_{\Omega_T} |\mathfrak{q}_{\cc}|^q \dx \dt \le C+C\int_{\Omega_T} \theta^{(\frac53-\varepsilon_0)\frac{q}{2-q}}\chi_{\{\theta\ge 1\}}+\theta^{-\frac{q(\beta-\varepsilon_0)}{2-q}}\chi_{\{\theta\le 1\}} \dx\dt.
\end{aligned}
$$
Thus, using the a~priori estimates \eqref{ap.theta.1} and \eqref{ep.theta.2}, we see that
\begin{equation}\label{ap.qc}
\int_{\Omega_T}|\mathfrak{q}_{\cc}|^q\dx\dt \le C,
\end{equation}
for $q$ chosen by
\begin{equation} \label{qq1}
\begin{aligned}
q&=2 \quad \textrm{ if } \varepsilon_0 \ge \max\left\{\beta, \frac{5}{3}\right\},\\
q&<\min\left\{\frac{10}{10-3\varepsilon_0},\frac{2\beta}{2\beta-\varepsilon_0}\right\}\quad \textrm{ otherwise.}
\end{aligned}
\end{equation}

Second, we focus on estimates for $\vc{q}_e$. Using \eqref{qe}--\eqref{hp.mM} and  the Young inequality, we get
$$
\begin{aligned}
\int_{\Omega_T}|\vc{q}_e|^q\dx \dt &\le C\int_{\Omega_T}|\kappa|^q |\na\theta|^q + |\vec{m}|^q \left|\mathfrak{P}_{\vec{\ell}}\left( \na\vzeta + \frac{\vec{z}}{\theta}\na\varphi \right)\right|^q\dx\dt\\
&\le C\int_{\Omega_T}|\na\theta|^q\chi_{\{\theta \ge 1\}}+\frac{|\nabla \theta|^q}{\theta^{\beta q}}\chi_{\{\theta\le 1\}}\dx\dt\\
 &\quad + \int_{\Omega_T}\frac{|\vec{m}|^{\frac{2q}{2-q}}}{|M(\theta)|^{\frac{q}{2-q}}}  +M(\theta)\left|\mathfrak{P}_{\vec{\ell}}\left( \na\vzeta + \frac{\vec{z}}{\theta}\na\varphi \right)\right|^2\dx\dt.
 \end{aligned}
$$
Due to the assumption \eqref{hp.m} and the a~priori estimates \eqref{eq.TE.s3} and \eqref{ep.theta.2}, the above inequality reduces for all $q\in (1,5/4)$ to
$$
\begin{aligned}
\int_{\Omega_T}|\vc{q}_e|^q\dx \dt
&\le C(q)+ \int_{\Omega_T}\left(\frac{|\nabla \theta|^2}{\theta^{\beta +2}}\right)^{\frac{q}{2}}\left(\frac{1}{\theta^{\beta -2}}\right)^{\frac{q}{2}}\chi_{\theta\le 1}\dx\dt\\
 &\quad + \int_{\Omega_T}\theta^{\frac{q}{2-q}}\chi_{\theta>1} + \theta^{-\frac{q(\beta-\varepsilon_0)}{2-q}}\chi_{\theta\leq 1} \dx\dt\\
 &\le C(q)+ \int_{\Omega_T}\frac{|\nabla \theta|^2}{\theta^{\beta +2}} +\theta^{\frac{q}{2-q}}\chi_{\theta >1} + \theta^{-\frac{q(\beta-\varepsilon_0)}{2-q}}\chi_{\theta \leq 1}+\theta^{-\frac{q(\beta-2)}{2-q}}\chi_{\theta \le 1} \dx\dt.
 \end{aligned}
$$
Hence, repeating the same procedure as above, we deduce that provided $\varepsilon_0<\beta$,\footnote{In what follows, we will need that $\beta \ge 1$, therefore the case $\varepsilon_0 \ge 2\beta$ is not interesting for us.} for all $q$ fulfilling
$$
\begin{aligned}
q&<\min\left\{\frac54,\frac{2\beta}{2\beta-\varepsilon_0}\right\}
\end{aligned}
$$
there holds
\begin{equation} \label{ap.qe}
\int_{\Omega_T}|\vc{q}_e|^q \dx\dt \le C(q).
\end{equation}

\subsection{Estimates on the chemical potential and concentration}
Let us consider $q\in (1,2]$. Then using the triangle inequality, the H\"{o}lder inequality, the structural assumption \eqref{hp.m} and a~priori bounds \eqref{ap.phi.1} and \eqref{eq.TE.s3}, we can find
$$
\begin{aligned}
\int_{\Omega_T} |\mathfrak{P}_{\vec{\ell}}\nabla\vzeta|^q \dx \dt &\le C\int_{\Omega_T}\left|\mathfrak{P}_{\vec{\ell}}\left(\nabla\vzeta + \frac{\vec{z}}{\theta}\nabla \varphi\right)\right|^q +\left|\frac{\vec{z}}{\theta}\nabla \varphi\right|^q\dx\dt\\
&\le C+ C\int_{\Omega_T} \frac{1}{(M(\theta))^{\frac{q}{2-q}}} +\frac{1}{\theta^q}\dx\dt\\
&\le C+ C\int_{\Omega_T} \frac{1}{\theta^{\frac{q(\beta-\varepsilon_0)}{2-q}}} +\frac{1}{\theta^q}\dx\dt.
\end{aligned}
$$
Therefore, setting
\begin{equation}\label{aaa}
q=a:=\min\left\{\beta,\frac{2\beta}{2\beta-\varepsilon_0}\right\}
\end{equation}
and recalling the a~priori bound \eqref{ap.theta.1}, we finally obtain
\begin{equation}\label{ap.nabla.zeta.1}
\begin{aligned}
\int_{\Omega_T} |\mathfrak{P}_{\vec{\ell}}\nabla\vzeta|^a \dx \dt \le C+ C\int_{\Omega_T} \frac{1}{\theta^{\beta}}\dx\dt\le C.
\end{aligned}
\end{equation}
Hence, using the fact that we control the trace of $\mathfrak{P}_{\vec{\ell}}\,\vzeta$, see \eqref{eq.TE.s3}, and assumption \eqref{ass.d}, the Poincar\'{e} inequality yields that
\begin{equation}\label{ap.zeta.2}
\begin{aligned}
\int_{0}^T \|\mathfrak{P}_{\vec{\ell}}\,\vzeta\|^{a}_{1,a} \dt\le C.
\end{aligned}
\end{equation}
Consequently, it follows directly from \eqref{as.zeta2} that
\begin{equation}\label{ap.zeta.3}
\begin{aligned}
\int_{0}^T \|\vzeta\|^{a}_{a} \dt\le C.
\end{aligned}
\end{equation}
It is worth noticing here that the above estimate together with assumption \eqref{hp.sc.2} directly imply that each concentration $c_i$ is strictly positive almost everywhere in $\Omega_T$.
Finally, a direct computation together with assumption \eqref{hp.sc.1} implies
$$
C_1|\pa_{x_k}\cc|^2 \le -\sum_{i,j=1}^L\pa^2_{c_i c_j} s_{\cc}^*(\cc) \pa_{x_k} c_i \pa_{x_k}c_j=\pa_{x_k} \vzeta \cdot \pa_{x_k}\cc=\pa_{x_k} (\mathfrak{P}_{\vec{\ell}}\,\vzeta) \cdot \pa_{x_k}\cc,
$$
where the last identity follows from \eqref{c.in.G}. Thus, we have
$$
C_1|\nabla \cc|\le |\nabla (\mathfrak{P}_{\vec{\ell}}\,\vzeta)|
$$
and the a~priori bound \eqref{ap.zeta.2} yields
\begin{equation}\label{ap.nabla.c}
\begin{aligned}
\int_{0}^T \|\cc\|^{a}_{1,a} \dt\le C.
\end{aligned}
\end{equation}
Finally we estimate the boundary fluxes. We easily see that
\begin{equation} \label{bnd_flx}
\int_{\Gamma} |\vec{q}_{\cc\ \Gamma}|^2 + |q_{e\Gamma}|^2 + |q_{\varphi \Gamma}|^q \dS \dt \leq C,
\end{equation}
where $1\leq q <\infty$ arbitrary, and $C$ depends on the data of the problem.

\section{Weak solution, main results}

\subsection{Different types of weak solutions}

In this paper, we will consider three different types of the solution. The natural definition is connected with the weak formulation of the system of equations \eqref{eq.c}--\eqref{eq.E} with initial and boundary conditions \eqref{IC}--\eqref{bc.phiA}. Taking into account Hypotheses (H1)--(H10), we have
\begin{definition} \label{d_weak_E}
We say that $(\cc,\vv,e,\varphi)$ is a weak total energy solution to problem \eqref{eq.c}--\eqref{bc.phiA} in $\Omega_T$, provided
\begin{itemize}
\item $\cc \in G$ for a.a. $t \in [0,T)$, $\vv \in L^r(0,T;W^{1,r}_{0,\nu}(\Omega)) \cap L^\infty(0,T;L^2(\Omega;\R^3))$, $\vv \in L^3(\Omega_T;\R^3)$, $\Div \vv = 0$ a.e. in $\Omega_T$, $\varphi \in L^\infty(0,T; W^{1,q}(\Omega)) \cap W^{1,p}(\Omega_T)$ for all $q\in [1,\infty)$ and some $p>1$, $e \in L^\infty(0,T;L^1(\Omega))\cap L^q(0,T;W^{1,q}(\Omega))$ for some $q>1$, $1/e \in L^s(0,T;W^{1,s}(\Omega))$ for some $s >1$

\item  $\cc \in C_{weak} ([0,T];L^q(\Omega;\R^L))$ for some $q>1$,  $\vv \in C_{weak}([0,T]; L^2(\Omega;\R^3))$,  $e \in C_{weak} ([0,T]; L^1(\Omega))$ and the initial conditions $(\cc^{\ 0},\vv^0,e^0)$ are fulfilled in the weak sense

\item the weak formulation of the species equations holds true
\begin{equation} \label{w_c}
\int_{\Omega_T} \cc \cdot \pa_t \vec{\psi}+ (\cc \otimes  \vv + \mathfrak{q}_{\cc}): \na \vec{\psi} + \vec{r}\cdot \vec{\psi} \dx \dt + \int_\Omega \vec{c}^{\ 0}\cdot \vec{\psi}(0) \dx = \int_{\Gamma} \vec{q}_{\cc\ \Gamma} \cdot\vec{\psi} \dS\dt
\end{equation}
for all $\vec{\psi} \in C^\infty(\Omega_T;\R^L)$, $\vec{\psi}(T)=\vec{0}$

\item the weak formulation of the momentum equation holds true
\begin{equation} \label{w_v}
\begin{aligned}
&\int_{\Omega_T} \vv \cdot \pa_t \vc{u}+ (\vv\otimes \vv - \mathcal{S}): \mathcal{D}(\vc{u}) - Q\nabla \varphi \cdot \vc{u} \dx \dt \\
+ & \int_\Omega \vv^0 \cdot \vc{u}(0) \dx = \int_{\Gamma} \gamma \vv\cdot \vc{u} \dS \dt
\end{aligned}
\end{equation}
for all $\vc{u} \in C^\infty(\Omega_T;\R^3)$, $\vc{u}(T) =\vc{0}$, $\vc{u} \cdot \nu =0$ on $(0,T)\times \pa \Omega$, $\Div \vc{u} = 0$ in $\Omega_T$

\item the equation for the electrostatic potential is fulfilled in the strong sense
\begin{equation} \label{s_poi}
-\Delta \varphi = Q
\end{equation}
a.e. in $Q_T$, including the time instant $t=0$, and $\na \varphi \cdot \nu  = q_{\varphi\Gamma}$ a.e. on $\Gamma$

\item the weak formulation of the total energy balance holds true
\begin{equation} \label{w_E}
\begin{aligned}
&\int_{\Omega_T} E \pa_t \psi  +\Big(\Big(\frac{|\vv|^2}{2} + e + Q\varphi\Big) \vv
+ \varphi \sum_{i=1}^L z_i \vc{q}_{\cc}^i + \vc{q}_e - {\mathcal T}\vv -\varphi \na \pa_t \varphi\Big)\cdot \na \psi \dx\dt \\
+& \int_\Omega E(0) \psi(0) \dx
= \int_{\Gamma} \Big(\gamma |\vv|^2  + \varphi \vec{z}\cdot \vec{q}_{\cc \ \Gamma} + q_{e\Gamma} -\varphi \pa_t q_{\varphi\Gamma}\Big)\psi \dS \dt
\end{aligned}
\end{equation}
for all $\psi \in C^\infty(\Omega_T)$ with $\psi(T) =0$, where $\varphi_t$ solves \eqref{s_poi} differentiated with respect to time (in the weak sense) and $E(0) = \frac{|\vv|^2(0)}{2} + e(0) + \frac{|\na \varphi|^2(0)}{2}$, where the initial value of $\varphi$ is a solution to \eqref{s_poi} at the time instant $t=0$
\end{itemize}
\end{definition}

Note that the main trouble maker is the convective term in the total energy balance as for some $r$ (the parameter in the power law model for the stress tensor) the velocity may be not integrable in the third power. To this aim, we introduce another definition of weak solution

\begin{definition} \label{d_var_e}
We say that $(\cc,\vv,e,\varphi)$ is a variational energy solution to problem \eqref{eq.c}--\eqref{bc.phiA} in $\Omega_T$, provided
the functions $(\cc,\vv,e,\varphi)$ fulfill the integrability and continuity assumptions from Definition \ref{d_weak_E} with $\vv \in L^3(\Omega_T;\R^3)$ replaced by $\vv e \in L^1(\Omega_T;\R^3)$, the species equation, the momentum equation and the equation for the electrostatic potential are fulfilled in the same sense as in Definition \ref{d_weak_E} (\ref{w_c}--\ref{s_poi}), and the weak formulation of the total energy balance is replaced by the inequality for the internal energy balance
\begin{equation} \label{w_ineq_e}
\begin{aligned}
&\int_{\Omega_T} e \pa_t \psi +  \mathcal{S}:\mathcal{D}(\vv) \psi \dx \dt + \int_\Omega e^0 \psi(0) \dx \\
\leq & \int_{\Omega_T}  \vec{z}\cdot (\mathfrak{q}_{\cc} \nabla \varphi) \psi - (e\vv+\vc{q}_e)\cdot \na \psi \dx \dt +  \int_{\Gamma} q_{e\Gamma} \psi \dS \dt
\end{aligned}
\end{equation}
for all nonnegative $\psi \in C^\infty(\Omega_T)$, $\psi(T) =0$, and by the global total energy balance (i.e., the total energy balance integrated over $\Omega_T$)
\begin{equation} \label{tot_glo_E}
\int_\Omega E(t) \dx +  \int_{\Gamma} \gamma |\vv|^2 \psi + \varphi \vec{z}\cdot (\mathfrak{q}_{\cc\ \Gamma}) + q_{e\Gamma} -\varphi  \pa_t q_{\varphi\Gamma} \dS \dt \leq \int_\Omega E(0) \dx
 \end{equation}
for a.a. $t \in (0,T)$.
\end{definition}

Finally, in some cases we can even verify the weak formulation of the internal energy balance.

\begin{definition} \label{d_weak_e}
We say that $(\cc,\vv,e,\varphi)$ is a weak internal energy solution to problem \eqref{eq.c}--\eqref{bc.phiA} in $\Omega_T$, provided
the functions $(\cc,\vv,e,\varphi)$ fulfill the integrability and continuity assumptions from Definition \ref{d_weak_E} with $\vv \in L^3(\Omega_T;\R^3)$ replaced by $\vv e \in L^1(\Omega_T;\R^3)$, the species equation, the momentum equation and the equation for the electrostatic potential are fulfilled in the same sense as in Definition \ref{d_weak_E} (\ref{w_c}--\ref{s_poi}), and the weak formulation of the total energy balance is replaced by the weak formulation for the internal energy balance, i.e.
\begin{equation} \label{w_eq_e}
\begin{aligned}
&\int_{\Omega_T} e \pa_t \psi +  \mathcal{S}:\mathcal{D}(\vv) \psi \dx \dt + \int_\Omega e^0 \psi(0) \dx \\
&= \int_{\Omega_T} \vec{z} \cdot (\mathfrak{q}_{\cc}\nabla \varphi) \psi - (e\vv+\vc{q}_e)\cdot \na \psi \dx \dt + \int_{\Gamma} q_{e\Gamma} \psi \dS \dt
\end{aligned}
\end{equation}
for all $\psi \in C^\infty(\Omega_T)$, $\psi(T) =0$.
\end{definition}

\subsection{Main results}

In the subsequent sections we will prove the following results.

\begin{theorem} \label{T1}
Let $r>\frac 32$ and $1<\beta\leq 2$. Under the assumptions stated in Hypotheses (H1)--(H10)  there exists a variational energy solution to our problem in the sense of Definition \ref{d_var_e}.
\end{theorem}

\begin{theorem} \label{T2}
Let $r>\frac 95$ and $1<\beta\leq 2$. Under the assumptions stated in Hypotheses (H1)--(H10) there exists a weak total energy solution to our problem in the sense of Definition \ref{d_weak_E}.
\end{theorem}

\begin{theorem} \label{T3}
Let $r\ge \frac{11}{5}$ and $1<\beta\leq 2$. Under the assumptions stated in Hypotheses (H1)--(H10)  there exists a weak internal energy solution to our problem  in the sense if Definition \ref{d_weak_e}.
\end{theorem}

Let us mention that the results of Theorem \ref{T1} and Theorem \ref{T3} remain true also if we replace the slip boundary condition for the velocity by the homogeneous Dirichlet conditions. Under slight modifications of the definition of the solutions, the proof is almost identical. On the other hand, the result of Theorem \ref{T2} is based on the use of the Navier boundary conditions; we need to have an integrable pressure. Since  the weak total energy solution is physically the most relevant, we prefer to perform the proof with the Navier boundary conditions. Furthermore, if $r \geq \frac{11}{5}$, the total energy balance holds as well (provided we use the Navier boundary conditions for the velocity).

The bounds for $r$ are connected with the following observations. If $r\geq \frac{3}{2}$,  the term $\vv e$ belongs to $L^1(\Omega_T;\R^3)$. In order to prove the existence of the solution, we need compactness of this term which results in the strict inequality. Similarly, for $r\geq \frac 95$, the term $|\vv|^3$ is integrable and the strict inequality is needed to verify the compactness of the sequence of the approximate solutions. The limit $r\geq \frac{11}{5}$ ensures that the velocity can be used as a test function in the weak formulation of the momentum equation \eqref{w_v} which guarantees the information that $\mathcal{S}:\mathcal{D}(\vv)$ converges weakly in $L^1(\Omega_T)$.

Note that the natural limit coming form the integrability of the convective term for the shear-thinning fluid is $r>\frac 65$ which yields the compactness of the convective term. Replacing condition \eqref{hp.kappa} by
$$
C_1 \leq \frac{\kappa^*(\cc,\theta)}{(1+\theta^\alpha)(1+\theta^{-\beta})} \leq C_2
$$
with $\alpha >0$, and modifying accordingly assumptions \eqref{hp.m}, we could reach for a suitable choice of $\alpha$ that $e\sim \theta \in L^q(\Omega_T)$ for $q>2$ for any $r>\frac 65$. Then the term $e\vv \in L^1(\Omega_T;\R^3)$ for any $r>\frac 65$ and we would prove existence of a variational energy solution for any $r>\frac 65$.

Last but not least, let us underline that if the variational energy solution is sufficiently smooth, then it is actually a classical solution to our original problem. This follows exactly as in the case of compressible Navier--Stokes--Fourier system (see \cite{Fe_Book1}).

\subsection{Known results}
Modeling of ionized mixtures is very important for the design of devices, such as fuel cells
(see eg. \cite{NemecMarsiMican:2009,WeberNewma:2004}),
 and in biology (see eg. \cite{Eisen:1979,GuLaiMow:1998,KedemKatch:1961}).
This subject has been studied from the thermodynamical point of view for a long period.
The  approach used here relies on the concept of  the barycentric velocity and has been invented by Eckart
and Prigogine, see \cite{Eckar:1940,Prigo:1967}.
We use the linear treatment of the chemical reactions, as detailed in \cite{deGroMazur:1984}.
A similar treatment of chemical reactions appears already
in \cite{NunziWalsh:1979}.
The gradient structure of chemical reactions is emphasized in \cite{Mielk:2010}, where,
in contrast to our approach, the free energy approach is employed.
The diffusion matrix has been derived in \cite{deGen:1980}.
The thermodynamical treatment is based on \cite{Oetti:2009},
which presents the mixture model in light of the GENERIC framework; for more informations
see \cite{Grmel:2010, Oetti:2005}.

The mixture theory implicitly assumes that we can find every component of the mixture in every
macroscopic point of the domain.
This assumption was used in its full strength in rational thermodynamics, see for example
\cite{MiranSchim:2005,RajagTao:1995,Samoh:2007}.
The assumption might be relaxed in the phenomenological thermodynamics framework to the assumption that chemical
potential of each constituent is defined almost everywhere in the domain.
Another relaxation of the assumption was done in \cite{Magna:2009}.
A different approach was used to obtain results for mono-atomic gases. In this case macroscopic, equations are obtained as
a limit from the system of Boltzmann's equations.
This approach was used in \cite{FerziKaper:1972} on gases without ionization or
chemical reactions. Chemical reactions were added in \cite{Giova:1999}.

In \cite{Roubi:2006,Roubi:2007} models are presented and analysed, which combine the incompressible Navier--Stokes equations with
both Newtonian and non-Newtonian stress tensor, the Nernst--Planck equation for chemically reacting, electrically charged mixtures, a
balance equation for the temperature and the Poisson equation for the electrostatic potential.
Althought the models presented in the papers are quite general, no cross-diffusion coupling is considered in the equations
for concentrations and temperature. An isothermal version of such models can be found in \cite{BulicMalekRajag:2009}.
The compressible case have been analyzed in \cite{AbelsFeire:2007,FeirePetzeTrivi:2008,MuPoZa2015}.
Alternative models for similar physical situations can be found in \cite{EPalff:1997,KimLowen:2005}.

\section{Proof of the main theorems}

In what follows we consider the approximation of our problem. We will introduce six parameters and by passing subsequently to the limits as  parameters tend either to zero or infinity we finally prove existence of a solution to our original problem. We follow the ideas developed in \cite{BuHa15}; however, certain modifications are needed due to the presence of the electrostatic field and due to the fact that we extend the results for more general class of power-law fluids, including also the shear-thinning models.

\subsection{Approximate problem}

First, we denote for $\varepsilon$ and $\delta>0$
\begin{equation} \label{a1}
s_{\vec{c}}^{*,\varepsilon}(\vec{c}) := s^*_{\cc}(\cc) + \varepsilon \sum_{i=1}^L \ln c_i,
\end{equation}
and
\begin{equation} \label{a2}
s_{\cc}^{*,\varepsilon,\delta}(\cc):= \left\{
\begin{array}{rl}
s_{\cc}^{*,\varepsilon}(\cc) & \mbox{ for } i=1,2,\dots, L: \delta \leq c_i\leq \frac{2}{\delta} \\
\mbox{concave } & \mbox{ otherwise,}
\end{array}
\right.
\end{equation}
so that $s_{\cc}^{*,\varepsilon,\delta}\in C^2(\R^L)$ and for any $\cc \in \R^L$
\begin{equation}\label{a3}
\begin{array}{c}
\displaystyle -C_1(\delta) (|\cc|^2 +1) \leq s_{\cc}^{*,\varepsilon,\delta}(\cc) \leq -C_2(\delta) |\cc|^2 + C_3, \\
\displaystyle |\{\pa_{c_ic_j}^2 s_{\cc}^{*,\varepsilon,\delta}(\cc)\}_{i,j=1}^L| (1+|\cc|) + |\pa_{\cc} s_{\cc}^{*,\varepsilon,\delta}(\cc)| \leq C(\delta) (1+|\cc|), \\
\displaystyle \sum_{i,j=1}^L\pa^2_{c_i c_j} s_{\cc}^{*,\varepsilon,\delta}(\cc) x_i x_j \leq -C(\delta) |\vec{x}|^2, \vec{x}\in \R^L.
\end{array}
\end{equation}

We further introduce the approximate chemical potential as
\begin{equation} \label{a4}
\vzeta^{*,\varepsilon}(\cc) := - \pa_{\cc} s_{\cc}^{*,\varepsilon}(\cc), \quad  \vzeta^{*,\varepsilon,\delta}(\cc) := - \pa_{\cc} s_{\cc}^{*,\varepsilon,\delta}(\cc).
\end{equation}

Note that as in \cite{BuHa15} we may show that there exists $C>0$ independent of $\varepsilon$ such that
\begin{equation} \label{a5}
|\vzeta^{*,\varepsilon}(\cc)|^2 \leq C(1+ |\mathfrak{P}_{\vec{\ell}}\,\vzeta^{*,\varepsilon}(\cc)|^2),
\end{equation}
and
\begin{equation} \label{a6}
|\vzeta^{*}(\cc)| \leq C(1+ |\vzeta^{*,\varepsilon}(\cc)|)
\end{equation}
for all $\cc\in G$. Next, we set
\begin{equation} \label{a7}
s_{e}^{*,\delta}(e):= \left\{
\begin{array}{rl}
s_{e}^{*}(e) & \mbox{ for }\delta \leq  e \leq \frac{2}{\delta}, \\
\mbox{concave, increasing } & \mbox{ otherwise,}
\end{array}
\right.
\end{equation}
where $s_e^{*,\delta}\in C^2(\R)$,
\begin{equation} \label{a8}
|\pa_e s_e^{*,\delta}(e)|+ |\pa^2_{e e} s_e^{*,\delta}(e)|\leq C(\delta).
\end{equation}
The approximate temperature is defined as
\begin{equation} \label{a9}
\theta^{*,\delta}(e):= \frac{1}{\pa_e s^{*,\delta}(e)}.
\end{equation}
Next, for $0<\delta<\frac 12$ we introduce a cut-off function
$$
T_\delta(y) := \left\{
\begin{array}{rl}
0 & y \leq \delta \\
1 & 2\delta \leq y\leq \frac 1\delta \\
0 & \frac{2}{\delta} \leq y \\
\mbox{linear} & \mbox{otherwise},
\end{array}
\right.
$$
and
$${\mathfrak T}_\delta(\cc) = \prod_{i=1}^LT_\delta(c_i)
$$
for $\cc \in \R^L$. We are now prepared to define the fluxes and the reaction term
\begin{equation} \label{a10}
\begin{array}{l}
\displaystyle {\mathfrak q}^{*,\delta}_{\cc}(e,\cc,\theta,\nabla \vzeta,\nabla \theta,\nabla \varphi) := {\mathfrak T}_\delta(\cc) T_\delta(e) {\mathfrak q}_{\cc}^*(\cc,\theta,\nabla \vzeta, \nabla \theta,\nabla \varphi), \\
\displaystyle \vc{q}^{*,\delta}_{e}(e,\cc,\theta,\nabla \vzeta,\nabla \theta,\nabla \varphi) := {\mathfrak T}_\delta(\cc) T_\delta(e) \vc{q}_e^*(\cc,\theta,\nabla \vzeta, \nabla \theta,\nabla \varphi), \\
\displaystyle \vc{r}^{*,\delta} (e,\cc,\theta,\vzeta):= {\mathfrak T}_\delta(\cc) T_\delta(e) r^*(\cc,\theta,\vzeta),\\
\displaystyle Q^{*,\delta}(\cc) := {\mathfrak T}_\delta(\cc)  \vec{z} \cdot\vec{c},
\end{array}
\end{equation}
and the boundary fluxes
\begin{equation} \label{a11}
\begin{array}{l}
\vec{q}^{\ *,\delta}_{\cc\ \Gamma}(x,e,\cc,\theta,\vzeta,\varphi) := {\mathfrak T}_\delta(\cc) T_\delta(e) \vec{q}_{\cc\ \Gamma}^{\ *}(x,\cc,\theta,\vzeta,\varphi), \\
q^{*,\delta}_{e\Gamma}(x,e,\cc,\theta,\vzeta) := {\mathfrak T}_\delta(\cc) T_\delta(e) q_{e\Gamma}^*(x,\cc,\theta), \\
q^{*,\delta}_{\varphi\Gamma}(x,\varphi) :=  q_{\varphi\Gamma}^*(x,\varphi),\\
\gamma^{*,\delta}(e,\cc,\theta) :=  {\mathfrak T}_\delta(\cc) T_\delta(e) \gamma^*(\cc,\theta).
\end{array}
\end{equation}

Next, we introduce a basis of $W^{3,2}(\Omega) \cap W^{1,r}_{\nu,\Div}(\Omega)$ consisting of functions $\{\vc{w}_i\}_{i=1}^\infty$. Notice that by density this is also a basis in $W^{1,r}_{\nu,\Div}(\Omega)$.  We assume that this basis is moreover orthonormal in $L^2(\Omega)$ and denote by $\vc{W}^n$ the linear span of $\{\vc{w}_i\}_{i=1}^n$ and  by $\vc{P}_1^n$ the orthogonal projection from $W^{3,2}_{\Div}(\Omega;\R^3) \cap W^{1,r}_{0,\nu}(\Omega)$ to $\vc{W}^n$. Similarly, we introduce a basis in $W^{1,2}(\Omega)$, orthonormal in $L^2(\Omega)$, consisting of functions $\{u_i\}_{i=1}^\infty$ and denote by $U^m$ the linear span of $\{u_i\}_{i=1}^m$ and by $P_2^m$ the orthogonal projection from $W^{1,2}(\Omega)$ to $U^m$. Finally, we introduce a nonincreasing smooth function $\xi$: $[0,\infty) \to \R_{+,0}$ such that $\xi \equiv 1$ in $[0,1]$, $\xi \equiv 0$ in $[2,\infty)$ and $0\geq \xi'\geq -2$ in $\R_+$. For $k \in \N$ we define
$$
\xi^k(y) := \xi(y/k), \quad y \in [0,\infty).
$$

We can now formulate our approximate problem:  for fixed $k,l,m,n \in \N$ and $0<\delta<\varepsilon <1$ we look for functions
$\vv^{k,n,\varepsilon, \delta, m,l}$, $\cc^{\ k,n,\varepsilon, \delta, m,l}$, $e^{k,n,\varepsilon, \delta, m,l}$, $\varphi^{k,n,\varepsilon, \delta, m,l}$ such that
$$
\begin{array}{rl}
\displaystyle \vv^{k,n,\varepsilon, \delta, m,l}(t,x) & := \sum_{i=1}^n \alpha_i^{k,n,\varepsilon, \delta, m,l}(t) \vw_i(x), \\
\displaystyle \cc^{\ k,n,\varepsilon, \delta, m,l}(t,x) & := \sum_{i=1}^m \vec{\beta}_i^{\ k,n,\varepsilon, \delta, m,l} (t) u_i(x), \\
\displaystyle e^{k,n,\varepsilon, \delta, m,l}(t,x) & := \sum_{i=1}^l \sigma_i^{k,n,\varepsilon, \delta, m,l}(t) u_i(x),
\end{array}
$$
$\alpha_i^{k,n,\varepsilon, \delta, m,l}$: $[0,T] \to \R$, $\vec{\beta}^{\ k,n,\varepsilon, \delta, m,l}_i$: $[0,T]\to \R^L$ and $\sigma_i^{k,n,\varepsilon, \delta, m,l}$: $[0,T] \to \R$ are absolutely continuous functions, where
$$
\vv^{k,n,\varepsilon, \delta, m,l}(0) := \vc{P}^n_1(\vv^0), \ \vec{c}^{\ k,n,\varepsilon, \delta, m,l}(0) := (P^m_2(c^{\delta}_0)_1, \dots, P^m_2(c^{\delta}_0)_L), \ e^{k,n,\varepsilon, \delta, m,l}(0) = P^l_2(e_0^{\varepsilon, \delta})
$$
with
$$
\vec{c}_0^{\ \delta} := \frac{\cc^{\ 0}}{1+\delta L} + \frac{\delta \vec{\ell}}{1+\delta L}, \quad e_0^{\varepsilon,\delta}:= \min \{\varepsilon^{-1},e^0\} + \delta.
$$
Furthermore, we require on $(0,T)$
\begin{equation} \label{a12}
\begin{aligned}
& \int_\Omega \pa_t \cc^{\ k,n,\varepsilon, \delta, m,l}\cdot \vec{\psi} + \varepsilon \nabla \cc^{\ k,n,\varepsilon, \delta, m,l} : \nabla \vec{\psi} - \big(\cc^{\ k,n,\varepsilon, \delta, m,l}\otimes \vv^{k,n,\varepsilon, \delta, m,l} + {\mathfrak q}^{k,n,\varepsilon, \delta,m,l}_{\cc}\big) : \nabla \vec{\psi} \dx \\
& = -  \int_{\partial \Omega} \vec{q}^{\ k,n,\varepsilon, \delta,m,l}_{\cc\,\Gamma} \cdot \vec{\psi} \dS + \int_{\Omega} \vec{r}^{\ k,n,\varepsilon, \delta,m,l} \cdot \vec{\psi} \dx
\end{aligned}
\end{equation}
for all $\vec{\psi} \in \vec{U}^m$,
\begin{equation} \label{a13}
\begin{aligned}
& \int_\Omega \pa_t \vv^{k,n,\varepsilon, \delta, m,l}\cdot \vc{u} + {\mathcal S}^{k,n,\varepsilon, \delta, m,l} : \mathcal{D} (\vc{u}) - \xi^k(|\vv^{k,n,\varepsilon, \delta, m,l}|^2)(\vc{v}^{k,n,\varepsilon, \delta, m,l}\otimes \vv^{k,n,\varepsilon, \delta, m,l}) : \nabla \vc{u} \dx \\
& = -  \int_{\partial \Omega} \gamma^{k,n,\varepsilon, \delta,m,l} \vv^{k,n,\varepsilon, \delta, m,l}\cdot \vc{u} \dS - \int_{\Omega} Q^{k,n,\varepsilon, \delta, m,l} \nabla \varphi^{k,n,\varepsilon, \delta,m,l} \cdot \vc{u} \dx
\end{aligned}
\end{equation}
for all $\vc{u} \in \vc{W}^n$,
\begin{equation} \label{a14}
\begin{aligned}
& \int_\Omega \pa_t e^{k,n,\varepsilon, \delta, m,l} \psi + \varepsilon \nabla e^{k,n,\varepsilon, \delta, m,l} \cdot \nabla \psi - \big(e^{k,n,\varepsilon, \delta, m,l}\vv^{k,n,\varepsilon, \delta, m,l} + \vc{q}_e^{k,n,\varepsilon, \delta, m,l}\big) \cdot \nabla \psi  \dx \\
= - & \int_{\partial \Omega} q_{e\Gamma}^{k,n,\varepsilon, \delta, m,l} \psi  \dS + \int_{\Omega} \big(\mathcal{S}^{k,n,\varepsilon, \delta, m,l}: \mathcal{D}(\vv^{k,n,\varepsilon, \delta,m,l}) -\vec{z} \cdot (\mathfrak{q}^{k,n,\varepsilon, \delta, m,l}_{\cc} \nabla \varphi^{k,n,\varepsilon, \delta, m,l})\big) \psi \dx
\end{aligned}
\end{equation}
for all $\psi \in {U}^l$,
\begin{equation} \label{a15}
\int_\Omega \big(\nabla \varphi^{k,n,\varepsilon, \delta, m,l}\cdot\nabla \psi - Q\psi\big)\dx = \int_{\partial \Omega} q_{\varphi \Gamma}^{k,n,\varepsilon, \delta, m,l} \psi \dS
\end{equation}
for all $\psi \in W^{1,2}(\Omega)$, where
\begin{equation} \label{a16}
{\mathfrak q}_{\cc}^{k,n,\varepsilon, \delta, m,l} := {\mathfrak q}_{\cc}^{*,\delta} (e^{k,n,\varepsilon, \delta, m,l}, \cc^{\ k,n,\varepsilon, \delta, m,l}, \theta^{k,n,\varepsilon, \delta, m,l},\nabla \vzeta^{\ k,n,\varepsilon, \delta, m,l},\nabla \theta^{k,n,\varepsilon, \delta, m,l},\nabla \varphi^{k,n,\varepsilon, \delta, m,l}),
\end{equation}
\begin{equation} \label{a17}
{\vc{q}}_{e}^{k,n,\varepsilon, \delta, m,l} := \vc{q}_{e}^{*,\delta} (e^{k,n,\varepsilon, \delta, m,l}, \cc^{\ k,n,\varepsilon, \delta, m,l}, \theta^{k,n,\varepsilon, \delta, m,l}, \nabla \vzeta^{\ k,n,\varepsilon, \delta, m,l},\nabla \theta^{k,n,\varepsilon, \delta, m,l}, \nabla \varphi^{k,n,\varepsilon, \delta, m,l}),
\end{equation}
\begin{equation} \label{a18}
{\mathcal S}^{k,n,\varepsilon, \delta, m,l} := {\mathcal S}^* ( \cc^{\ k,n,\varepsilon, \delta, m,l}, \theta^{k,n,\varepsilon, \delta, m,l}, {\mathcal D}(\vv^{k,n,\varepsilon, \delta, m,l})),
\end{equation}
\begin{equation} \label{a19}
\vec{r}^{\ k,n,\varepsilon, \delta, m,l} := \vec{r}^{\ *,\delta} (e^{k,n,\varepsilon, \delta, m,l},\cc^{\ k,n,\varepsilon, \delta, m,l},\theta^{k,n,\varepsilon, \delta, m,l},\vzeta^{\ k,n,\varepsilon, \delta, m,l}),
\end{equation}
\begin{equation} \label{a20}
Q^{k,n,\varepsilon, \delta, m,l} = Q^{*,\delta} (\cc^{\ k,n,\varepsilon, \delta, m,l}),
\end{equation}
the boundary fluxes
\begin{equation} \label{a20a}
\begin{array}{l}
\displaystyle \vec{q}_{\cc\,\Gamma}^{\ k,n,\varepsilon, \delta, m,l} := \vec{q}^{\ *,\delta}_{\cc\,\Gamma}(x,e^{k,n,\varepsilon, \delta, m,l},\cc^{\ k,n,\varepsilon, \delta, m,l},\theta^{k,n,\varepsilon, \delta, m,l},\vzeta^{\ k,n,\varepsilon, \delta, m,l},\varphi^{k,n,\varepsilon, \delta, m,l}), \\
\displaystyle q_{e\Gamma}^{k,n,\varepsilon, \delta, m,l}:= q^{*,\delta}_{e\Gamma}(x,e^{k,n,\varepsilon, \delta, m,l},\cc^{\ k,n,\varepsilon, \delta, m,l},\theta^{k,n,\varepsilon, \delta, m,l},\vzeta^{\ k,n,\varepsilon, \delta, m,l}),  \\
\displaystyle q_{\varphi\Gamma}^{k,n,\varepsilon, \delta, m,l}:= q^{*,\delta}_{\varphi\Gamma}(x,\varphi^{k,n,\varepsilon, \delta, m,l}), \\
\displaystyle \gamma^{k,n,\varepsilon, \delta, m,l} := \gamma^{*,\delta} (\cc^{\ k,n,\varepsilon, \delta, m,l},e^{k,n,\varepsilon, \delta, m,l},\theta^{k,n,\varepsilon, \delta, m,l}),
\end{array}
\end{equation}
 and, in addition,
\begin{equation} \label{a21}
\theta^{k,n,\varepsilon, \delta, m,l} :=\max\{\delta^2,\theta^{*,\delta}(e^{k,n,\varepsilon, \delta, m,l})\},
\end{equation}
\begin{equation} \label{a22}
\vzeta^{\ k,n,\varepsilon, \delta, m,l}:= \Big(P^m_2 (\zeta^{*,\varepsilon,\delta}_1(\cc^{\ k,n,\varepsilon, \delta, m,l})), \dots, P^m_2 (\zeta^{*,\varepsilon,\delta}_L(\cc^{\ k,n,\varepsilon, \delta, m,l}))\Big).
\end{equation}

We proceed as follows. In the first step we show that for a fixed electrostatic potential problem \eqref{a12}--\eqref{a14} is uniquely solvable globally in time. Then using a suitable fixed-point argument, we verify that indeed, also problem \eqref{a12}--\eqref{a15} is solvable globally in time. In both cases, the most important information used in the proof are the a priori estimates. Having solved \eqref{a12}--\eqref{a15}, we let subsequently $l \to \infty$, $m\to \infty$, $\delta \to 0^+$, $\varepsilon \to 0^+$ and $n \to \infty$. Then, we rewrite the internal energy balance into the total energy balance (for $r\geq \frac 95$) and let finally $k \to \infty$ to prove Theorems \ref{T1}--\ref{T3}.

\subsection{Solvability of the approximate problem}

We start with the first step. We fix $\varphi_0 \in L^2(0,T;W^{1,2}(\Omega))$ and consider problem \eqref{a12}--\eqref{a14} with $\varphi^{k,n,\varepsilon, \delta, m,l}$ replaced by $\varphi_0$. In what follows, we skip the indices $k,n,\varepsilon, \delta, m,l$. We also suitably mollify all functions which are only continuous in $\cc$, $\vv$ and $e$ in such a way that the main properties required in Hypotheses (H1)--(H10) needed below remain valid, and denote the mollified functions by the index $\eta$.

Note that system \eqref{a12}--\eqref{a14} is locally in time solvable due to the classical theory of systems of ODE's (the system can be rewritten as system of the first order with the right-hand side Lipschitz continuous in $\cc$, $\vv$ and $e$). In order to verify that the solution is in fact global in time, we prove estimates of solutions which will exclude the possibility of the blow-up before the time instant $t=T$. Moreover, these estimates will also provide us with important piece of information in order to verify solvability of \eqref{a12}--\eqref{a15}.

We first take as test functions in \eqref{a12} the function $\cc$, in \eqref{a13} the function $\vv$ and in \eqref{a14} the function $e$. First, for the approximate velocity we have
\begin{equation} \label{b1}
\frac 12 \frac{\mbox{d}}{\mbox{d}t} \|\vv\|_{2}^2 + \|\mathcal{S}_\eta:\mathcal{D}(\vv)\|_1 + \int_{\pa \Omega} \gamma_\eta |\vv|^2 \dS = \int_{\Omega} Q \vv \cdot \na \varphi_0 \dx.
\end{equation}
Due to the definition of $Q$ we immediately have
\begin{equation} \label{b2}
\|\vv\|^2_{L^\infty(0,T;L^2(\Omega;\R^3))} + \|\nabla \vv\|^r_{L^r(\Omega_T;\R^9)} \leq C(\|\na \varphi_0\|_{L^2(\Omega_T;\R^3)}^2 + 1).
\end{equation}

Next, \eqref{a12} and \eqref{a14} yield
\begin{equation} \label{b3}
\begin{aligned}
\frac 12 \frac{\mbox{d}}{\mbox{d}t} \|\cc\|_{2}^2 + \eps \|\na \cc\|_2^2 &\leq \Big|\int_{\Omega}  \mathfrak{q}_{\vec{c},\eta} : \na \cc + \vec{r}_\eta \cdot \vec{c} \dx -\int_{\pa \Omega} \vec{q}_{\vec{c}\, \Gamma, \eta}\cdot \cc \dS \\
\frac 12 \frac{\mbox{d}}{\mbox{d}t} \|e\|_{2}^2 + \eps \|\na e\|_2^2&\leq \Big|\int_{\Omega} \vc{q}_{e,\eta} \cdot \na e + (\mathcal{S}_\eta:\mathcal{D}(\vv))e -\vec{z}\cdot (\mathfrak{q}_{\cc,\eta} \na \varphi_0)e \dx - \int_{\pa \Omega} q_{e\Gamma,\eta} e \dS.
\end{aligned}
\end{equation}

Summing up these two identities, using the fact that $e$, $\cc$ and $\vv$ are from a finite dimensional space, we end up with
\begin{equation} \label{b4}
\|(\cc;e)\|^2_{L^\infty(0,T;L^2(\Omega;\R^{L+1}))} + \|\nabla (\cc;e)\|^2_{L^2(\Omega_T;\R^{3(L+1)})} \leq C(1 + \|\varphi_0 \|^2_{L^2(0,T;W^{1,2}(\Omega;\R^3))}).
\end{equation}
This estimate together with \eqref{b2} implies that the solution exists on the whole time interval $(0,T)$. Furthermore, it is easy to see that the operator $L$ assigning
$$
\varphi_0 \mapsto (\vv,\cc,e) \mapsto \varphi_1,
$$
where $\varphi_1$ solves \eqref{a15} with $Q = Q^{*,\delta}(\cc(\varphi_0))$, $q_{\varphi\Gamma} = q_{\varphi\Gamma}^{*,\delta}(x,\varphi_1)$, is continuous from the space $L^2(0,T;W^{1,2}(\Omega))$ to itself. Moreover, as $Q$ is bounded and $\pa_t\cc$ can be estimated in $L^2(\Omega_T;\R^L)$, the operator is compact. In order to apply a version of the Schauder fixed point theorem (sometimes called also Schaefer's theorem), it is enough to verify that the possible fixed points
$$ s L(\varphi) =\varphi$$
are bounded independently of $s \in [0,1]$. Due to the form of equation \eqref{a15} it is an easy matter to verify also this step. Then we pass to the limit with the mollifying parameter $\eta \to 0^+$. Whence, we proved existence of a solution to our system \eqref{a12}--\eqref{a15}.

\subsection{Limit passage $l \to \infty$}

Next, we aim at letting $l \to \infty$. We denote all unknown functions with the upper index $l$. We first prove estimates independent of this parameter. Note that equation \eqref{a15} provides the estimate
\begin{equation} \label{b5}
\|\varphi^l\|_{L^\infty(0,T;W^{2,q}(\Omega))} \leq C(\delta).
\end{equation}
Next, similar reasoning as in \eqref{b1}--\eqref{b2} yields
\begin{equation} \label{b6}
\|\vv^l\|_{L^\infty(0,T;L^2(\Omega;\R^3))} + \|\na \vv^l\|_{L^r(\Omega_T;\R^9)} + \|\mathcal{S}^l\|_{L^{r'}(\Omega_T;\R^9)} \leq C(\delta).
\end{equation}
Using the fact that for $\vv^l$ and $\cc^{\ l}$ all norms are equivalent, it is not difficult to see that
\begin{equation} \label{b6a}
\|(\cc^{\ l};e^l)\|^2_{L^\infty(0,T;L^2(\Omega;\R^{L+1}))} + \|\nabla (\cc^{\ l};e^l)\|^2_{L^2(\Omega_T;\R^{3(L+1)})} \leq C(m, \delta).
\end{equation}
Moreover, we also have
\begin{equation} \label{b7}
\int_0^T \|\pa_t \cc^{\ l}\|_2^2 + \|\pa_t \vv^l\|_2^2 + \|\pa_t e^l\|_{-1,2}^2 \dt \leq C(m,\delta),
\end{equation}
where the lower index $_{-1,2}$ denotes the norm in $(W^{1,2}(\Omega))'$, the dual space to  $W^{1,2}(\Omega)$.
Therefore also
\begin{equation} \label{b7a}
\int_0^T \|\nabla \pa_t \varphi\|_{L^2(\Omega;\R^3)}^2 \dt \leq C.
\end{equation}
Next, estimates above imply
\begin{equation} \label{b8}
\begin{aligned}
\int_0^T & \|\mathfrak{q}_{\cc}^l\|_2^2 + \|\vc{q}_e^l\|_2^2 + \|\vec{q}^{\ l}_{\cc\ \Gamma}\|_{L^2(\pa \Omega;\R^L)}^2 + \|q_{e\Gamma}^l\|_{L^2(\pa \Omega)}^2 \\
+& \|q_{\varphi \Gamma}^l\|_{L^\infty (\pa \Omega)}^2  + \|\vzeta^{\ l}\|_{1,2}^2 + \|\mathcal{S}^l\|_{r'}^{r'} \dt \leq C(m).
\end{aligned}
\end{equation}
Combining estimates above with the fact that the velocity and the concentrations remain in finite dimensional spaces we easily deduce that (possibly for subsequences)
\begin{equation} \label{b9}
\begin{aligned}
\vv^l \to \vv \qquad & \mbox{strongly in }C([0,T];\vc{W}^n), \\
\pa_t \vv^l \to \pa_t \vv \qquad & \mbox{weakly in } L^2(\Omega_T;\R^3), \\
\cc^{\ l} \to \cc \qquad & \mbox{strongly in } C([0,T]; (U^m)^L), \\
\pa_t \cc^{\ l} \to \pa_t \cc \qquad & \mbox{weakly in } L^2(\Omega_T;\R^3), \\
e^l \to e  \qquad & \mbox{strongly in } L^2(\Omega_T), \\
e^l \to e \qquad & \mbox{weakly in } L^2(0,T;W^{1,2}(\Omega)), \\
\pa_t e^l \to \pa_t e \qquad & \mbox{weakly in } L^2(0,T;W^{-1,2}(\Omega)), \\
\varphi^l \to \varphi \qquad & \mbox{strongly in } L^q(0,T;W^{1,q}(\Omega)), 1\leq q <\infty, \\
\varphi^l \to \varphi \qquad & \mbox{weakly in } L^q(0,T;W^{2,q}(\Omega)), 1\leq q<\infty, \\
\mathfrak{q}_{\cc}^l \to \mathfrak{q}_{\cc} \qquad & \mbox{weakly in } L^2(0,T;L^2(\Omega;\R^{3})), \\
\vc{q}_e^l \to \vc{q}_e \qquad & \mbox{weakly in } L^2(0,T;L^2(\Omega;\R^{3L})), \\
\vec{q}_{\cc\ \Gamma}^{\ l} \to \vec{q}_{\cc\ \Gamma} \qquad & \mbox{weakly in } L^2(0,T;L^2(\pa \Omega;\R^{L})), \\
{q}_{e\Gamma}^{l} \to {q}_{e\Gamma} \qquad & \mbox{weakly in } L^2(0,T;L^2(\pa \Omega)), \\
{q}_{\varphi\Gamma}^{ l} \to {q}_{\varphi\Gamma} \qquad & \mbox{weakly}*\mbox{ in } L^2(0,T;L^\infty(\pa \Omega)), \\
\vzeta^{\, l} \to \vzeta \qquad & \mbox{weakly in } L^2(0,T;W^{1,2}(\Omega;\R^L)), \\
\mathcal{S}^l \to \mathcal{S} \qquad & \mbox{weakly in } L^{r'}(\Omega_T;\R^{9}),
\end{aligned}
\end{equation}
and the limit functions fulfill for a.e. $t \in (0,T)$
\begin{equation} \label{b10}
\begin{aligned}
 \int_\Omega \pa_t \cc\cdot \vec{\psi} + \varepsilon \nabla \cc : \nabla \vec{\psi} - \big(\cc\otimes \vv + {\mathfrak q}_{\cc}\big) : \nabla \vec{\psi} \dx
= -  \int_{\partial \Omega} \vec{q}_{\cc\,\Gamma} \cdot \vec{\psi} \dS + \int_{\Omega} \vec{r} \cdot \vec{\psi} \dx
\end{aligned}
\end{equation}
for all $\vec{\psi} \in \vec{U}^m$,
\begin{equation} \label{b11}
\begin{aligned}
 \int_\Omega \pa_t \vv\cdot \vc{u} + {\mathcal S} : \mathcal{D} (\vc{u}) - \xi^k(|\vv|^2)(\vc{v}\otimes \vv) : \nabla \vc{u}  \dx
= -  \int_{\partial \Omega} \gamma\vv\cdot \vc{u} \dS - \int_{\Omega} Q\nabla \varphi \cdot \vc{u} \dx
\end{aligned}
\end{equation}
for all $\vc{u} \in \vc{W}^n$,
\begin{equation} \label{b12}
\begin{aligned}
& \langle \pa_t e, \psi\rangle_{W^{-1,2}(\Omega);W^{1,2}(\Omega)} + \int_{\Omega}\varepsilon \nabla e\cdot \nabla \psi - \big(e\vv + \vc{q}_e\big) \cdot \nabla \psi   \dx  \\
&= - \int_{\partial \Omega} q_{e\Gamma} \psi  \dS + \int_{\Omega} \big(\mathcal{S}: \mathcal{D}( \vv) - \vec{z} \cdot (\mathfrak{q}_{\cc} \nabla \varphi)\big) \psi \dx
\end{aligned}
\end{equation}
for all $\psi \in W^{1,2}(\Omega)$, $\lim_{t\to 0^+} e(t) = e_{0}^{\eps,\delta}$ in $L^2(\Omega)$ while for $\cc$ and $\vv$ the initial conditions are fulfilled in the same sense as above,
\begin{equation} \label{b13}
\int_\Omega \nabla \varphi\cdot\nabla \psi - Q\psi\dx = \int_{\partial \Omega} q_{\varphi \Gamma} \psi \dS
\end{equation}
for all $\psi \in W^{1,2}(\Omega)$. It is an easy matter to verify at this moment that
\begin{equation} \label{b14}
{\mathfrak q}_{\cc} = {\mathfrak q}_{\cc}^{*,\delta} (e, \cc,\theta,\nabla \vzeta,\nabla \theta,\nabla \varphi),
\end{equation}
\begin{equation} \label{b15}
{\vc{q}}_{e} = \vc{q}_{e}^{*,\delta} (e, \cc,\theta,\nabla \vzeta,\nabla \theta, \nabla \varphi),
\end{equation}
\begin{equation} \label{b16}
{\mathcal S} = {\mathcal S}^* (\cc,\theta,{\mathcal D}(\vv)),
\end{equation}
\begin{equation} \label{b17}
\vec{r} = \vec{r}^{\ *,\delta} (e,\cc,\theta,\vzeta),
\end{equation}
\begin{equation} \label{b18}
Q = Q^{*,\delta} (\cc),
\end{equation}
the boundary fluxes
\begin{equation} \label{b19}
\begin{array}{l}
\displaystyle \vec{q}_{\cc\,\Gamma} = \vec{q}^{\ *,\delta}_{\cc\,\Gamma}(x,e,\cc,\theta,\vzeta,\varphi), \\
\displaystyle q_{e\Gamma}= q^{*,\delta}_{e\Gamma}(x,e,\cc,\theta,\vzeta),  \\
\displaystyle q_{\varphi\Gamma}= q^{*,\delta}_{\varphi\Gamma}(x,\varphi), \\
\displaystyle \gamma = \gamma^{*,\delta} (e,\cc,\theta),
\end{array}
\end{equation}
 and, in addition,
\begin{equation} \label{b20}
\theta =\max\{\delta^2,\theta^{*,\delta}(e)\},
\end{equation}
\begin{equation} \label{b21}
\vzeta= \Big(P^m_2 (\zeta^{*,\varepsilon,\delta}_1(\cc)), \dots, P^m_2 (\zeta^{*,\varepsilon,\delta}_L(\cc))\Big).
\end{equation}

\subsection{Limit passage $m \to \infty$}

First, note that we can show the minimum principle for $e^m$. Using as test function in \eqref{b12} the function $\psi = \min\{0,e^m-\delta\}$ and recalling that for $e^m \leq \delta$ the fluxes are zero due to the cut-off function $T_\delta(e)$, we get, thanks to the form of the initial condition, that
$$
\frac{{\rm d}}{{\rm d}t} \|\min\{0,e^m-\delta\}\|_2^2 \leq 0,
$$
hence $e^m \geq \delta$ a.e. in $\Omega_T$. Therefore, for $\delta$ sufficiently small, we have
$$
\theta^m = \theta^{*,\delta}(e^m).
$$
Note that estimates \eqref{b5}--\eqref{b6} remain valid (with $l$ replaced by $m$). Next we  would like to obtain similar estimates as above in order to let $m \to \infty$. Even though inspired by \cite{BuHa15}, we have to modify the procedure as we are not able to estimate the internal energy independently of $m$ and the estimate of the total energy is  not yet available at this moment. Anyway, we first compute an estimate of the internal energy and later on, we will estimate the term which will remain on the right-hand side.

Let us choose $\psi = \chi_{[0,t]\times\Omega}$ (the characteristic function of $[0,t]\times \Omega$) in \eqref{b12}. Indeed, we are not allowed to do so directly, but after smoothing this function in time we get
\begin{align}\label{int.e}
\int_{\Omega}e^{m}(t)\dx &= \int_{\Omega}e_{0}^{\eps,\delta}\dx -\int_0^T\int_{\Gamma}q_{e\Gamma}^m\dS\dtau
+\int_{\Omega_T}\mathcal{S}^m : \mathcal{D}(\vv^{m}) - \vec{z}\cdot(\mathfrak{q}_{\cc}^m\na\varphi^{m})\dx\dtau.
\end{align}
Since $\vc{q}_{\cc}^{i,m} = -T_{\delta}(e^m)\mathfrak{T}_{\delta}(\vec{c}^{\ m})\big( \sum_{j=1}^{L}M_{ij}(\na \zeta^m_{j} + (\theta^m)^{-1}z_{j}\na\varphi^m )
-m_{i}\na((\theta^m)^{-1}) \big)$ and $\theta^{m}$ is strictly positive,
we can write $|\mathfrak{q}_{\cc}^m|\leq C_{1} + C_{2}T_{\delta}(e^m)\mathfrak{T}_{\delta}(\vec{c}^{\ m})(|\na\vzeta^{\ m}| + |\na \theta^{m}|)$. Thus
\begin{align}\label{int.e.2}
& \int_{\Omega}e^{m}(t)\dx \leq C\left( 1 + \int_{\Omega_T}T_{\delta}(e^m)\mathfrak{T}_{\delta}(\vec{c}^{\ m})(|\na (\vzeta^{\ m})| + |\na \theta^{m}|)\dx\dtau \right).
\end{align}
Now, choose $\psi = \frac{\pa s_{e}^{*,\delta}}{\pa e}(e^{m})$ in \eqref{b12}. It follows
\begin{align*}
& \frac{d}{dt}\int_{\Omega}s_{e}^{*,\delta}(e^{m})\dx = \left\langle\frac{\pa e^{m}}{\pa t}, \frac{\pa s_{e}^{*}}{\pa e}(e^{m})\right\rangle
= -\eps\int_{\Omega}\na e^{m}\cdot\na\frac{1}{\theta^{m}}\dx + \int_{\Omega}(e^{m}\vv^{m} + \vc{q}_{e}^m)\cdot\na\frac{1}{\theta^{m}}\dx\\
&\qquad -\int_{\pa\Omega}q_{e\Gamma}^m\frac{1}{\theta^{m}}\dS
+ \int_{\Omega} \frac{1}{\theta^{m}}\mathcal{S}^m:\mathcal{D} (\vv^{m}) - \frac{1}{\theta^{m}}\vec{z}\cdot(\mathfrak{q}_{\cc}^m\na\varphi^{m})  \dx .
\end{align*}
It holds
$$ -\eps\int_{\Omega}\na e^{m}\cdot\na\frac{1}{\theta^{m}}\dx = -\eps\int_{\Omega}\frac{\pa^{2} s_{e}^{*,\delta}}{\pa e^{2}}|\na e^{m}|^{2}dx\geq 0 $$
because $s_{e}^{*,\delta}$ is concave. Moreover, as $\Div \vv^m =0$,
\begin{align*}
\int_{\Omega}e^{m}\vv^{m}\cdot\na\frac{1}{\theta^{m}}\dx = 0.
\end{align*}
Furthermore,
\begin{align*}
\int_{\Omega}\vc{q}_{e}^m\cdot\na\frac{1}{\theta^{m}}\dx
&= \int_{\Omega}T_{\delta}(e^m)\mathfrak{T}_{\delta}(\vec{c}^{\ m})\frac{\kappa|\na\theta^{m}|^{2}}{|\theta^{m}|^{2}}\dx\\
&\qquad -\int_{\Omega}T_{\delta}(e^m)\mathfrak{T}_{\delta}(\vec{c}^{\ m})\sum_{i=1}^{L}m_{i}\left( \na\zeta_{i}^m + \frac{z_{i}}{\theta^{m}}\na\varphi^{m} \right)\cdot
\na\frac{1}{\theta^{m}}\dx.
\end{align*}
What's more, since $a(a-b)\geq -b^{2}/2$ for any $a,b\in\R$,
\begin{align}\label{qeG.1}
-\int_{\pa\Omega}q_{e\Gamma}^m\frac{1}{\theta^{m}}\dS &=
\int_{\pa\Omega}\kappa_{\Gamma}\frac{1}{\theta^{m}}\left( \frac{1}{\theta^{m}} - \frac{1}{\theta^{\Gamma}} \right)\dS
\geq -\frac{1}{2}\int_{\pa\Omega}\frac{\kappa_{\Gamma}}{(\theta^{\Gamma})^{2}}\dS \geq -C.
\end{align}
Then, observe that
\begin{align}\label{Snav.1}
& \int_{\Omega}\frac{1}{\theta^{m}}\mathcal{S}^m:\mathcal{D}(\vv^{m})\dx \geq 0,
\end{align}
and
\begin{align*}
 -\int_{\Omega}\frac{1}{\theta^{m}}(\vec{z}\cdot\mathfrak{q}_{\cc}^m\na\varphi^{m}) \dx
 =&\int_{\Omega}T_{\delta}(e^m)\mathfrak{T}_{\delta}(\vec{c}^{\ m}) \mathfrak{M}\Big( \na\vzeta^{\ m}  +\frac{\vec{z}}{\theta^{m}}\na \varphi^{m} \Big):
\frac{\vec{z}}{\theta^{m}}\,\na\varphi^{m} \\ + &\vec{m}\cdot \vec{z}\, \na\frac{1}{\theta^{m}}\cdot\na\varphi^{m}\frac{1}{\theta^{m}}\dx .
\end{align*}
Putting together the previous relations yields
\begin{align}\label{int.se}
& \frac{{\rm d}}{{\rm d}t}\int_{\Omega}s_{e}^{*,\delta}(e^{m})\dx \geq -C
+ \int_{\Omega}T_{\delta}(e^m)\mathfrak{T}_{\delta}(\vec{c}^{\ m})\left(
\frac{\kappa|\na\theta^{m}|^{2}}{|\theta^{m}|^{2}} + \sum_{i,j=1}^{L}\mathfrak{M}:(\vec{z}\otimes\vec{z})\left|\frac{\na\varphi^m}{\theta^{m}}\right|^{2}
\right)\dx\\
& +\int_{\Omega}T_{\delta}(e^m)\mathfrak{T}_{\delta}(\vec{c}^{\ m})\left(
-\sum_{i=1}^L m_i  \na(\zeta_{i}^{m})\cdot\na\frac{1}{\theta^{m}}
+ \sum_{i,j=1}^{L}M_{ij} \na\zeta_{j}^{m}\cdot\frac{z_{i}}{\theta^{m}}\na\varphi^{m} \right)\dx .\nonumber
\end{align}
Next, take $\vec\psi =-\vzeta^{\ m} =  -P_{2}^m(\vzeta^{\ *}(\vec c^{\ m}))$ in \eqref{b10}. As it is an appropriate test function (therefore the projection was needed), it follows
\begin{align*}
& \frac{{\rm d}}{{\rm d}t}\int_{\Omega}s_{\cc}^{*,\eps,\delta}(\vec {c}^{\ m})\dx =
-\int_{\Omega}\pa_{t}\vec{c}^{\ m}\cdot \vzeta^{\ *}(\vec {c}^{\ m})\dx = -\int_{\Omega}\pa_{t}\vec{c}^{\ m}\cdot P_{2}^m(\vzeta^{\ *}(\vec {c}^{\ m}))\dx=\\
&\int_{\Omega}\eps \na\vec{c}^{\ m} : \na (P_{2}^m(\vzeta^{\ *}(\vec {c}^{\ m})))\
-(\vec{c}^{\ m}\otimes \vv^m):\na(P_{2}^m(\vzeta^{\ *}(\vec {c}^{\ m}))) - \mathfrak{q}_{\cc}:\na(P_{2}^m(\vzeta^{\ *}(\vec {c}^{\ m})))\dx\\
&\qquad +\int_{\pa\Omega}\vec{q}_{\vec{c}\ \Gamma}^{\ m}\cdot P_{2}^m(\vzeta^{\ *}(\vec {c}^{\ m}))\dS
- \int_{\Omega}\vec{r}^{\ m}\cdot P_{2}^m(\vzeta^{\ *}(\vec {c}^{\ m}))\dx .
\end{align*}
The orthogonality of the basis functions $(u_{i})_{i\in\N}$ in $L^{2}(\Omega)$ and $W^{1,2}(\Omega)$ as well as the strict concavity of $s_{\cc}^{*,\eps,\delta}$ imply
\begin{align*}
& \eps\int_{\Omega}\na\vec c^{\ m} : \na (P_{2}^m(\vzeta^{\ *}(\vec c^{\ m})))\dx \\
= &\eps\int_{\Omega}\na\vec c^{\ m} : \na (\vzeta^{\ *}(\vec c^{\ m}))\dx
= -\eps\int_{\Omega}\na\vec c^{\ m} : \frac{\pa^{2}s_{\cc}^{*,\eps,\delta}}{\pa \cc^{\ 2}}\na \cc^{\ m}\dx
\geq \eps C\|\na \cc^{\ m}\|_{2}^{2}.
\end{align*}
Moreover, the boundedness of $P_{2}^m$ in $L^{2}(\Omega)$ and the fact that $|\cc^{\ m}|^{2}\leq C(1 + |s_{\cc}^{*,\eps,\delta}(\cc^{\ m})| )$ lead to
\begin{align*}
& -\int_{\Omega}(\vec{c}^{\ m}\otimes\vc{v}^m):\na(P_{2}^m(\vzeta^{\ *}(\vec c^{\ m})))\dx
\geq -\|\vv^{m}\|_{\infty}\|\cc^{\ m}\|_{2}\|\na P_{2}^m(\vzeta^{\ *}(\cc^{\ m}))\|_{2}\\
&\qquad\geq -C\|\vv^{m}\|_{\infty}\|\cc^{\ m}\|_{2}\|\na\vzeta^{\ *}(\cc^{\ m})\|_{2}
\geq -C\|\vv^{m}\|_{\infty}(1+\|s_{\cc}^{*,\eps,\delta}\|_{1}^{1/2})(1+\|\na \cc^{\ m}\|_{2}).
\end{align*}
Furthermore
\begin{align*}
& - \int_{\Omega}\mathfrak{q}^m_{\cc}:\na(P_{2}^m(\vzeta^{\ *}(\vec c^{\ m})))\dx \\
&= \int_{\Omega}T_{\delta}(e^m)\mathfrak{T}_{\delta}(\cc^{\ m})\Big(
\sum_{i,j=1}^{L}M_{ij}\Big( \na\big(P_{2}^m(\zeta_{j}^{*}(\cc^{\ m}))\big)
+ \frac{z_{j}}{\theta^{m}}\na\varphi^m \Big)\cdot\na\big(P_{2}^m(\zeta_{i}^{*}(\cc^{\ m}))\big)\\
&+\sum_{i=1}^{L}m_{i}\na\frac{1}{\theta^{m}}\cdot\na\big(P_{2}^m(\zeta_{i}^{*}(\cc^{\ m}))\big)\Big)\dx\\
&= \int_{\Omega}T_{\delta}(e^m)\mathfrak{T}_{\delta}(\cc^{\ m})\sum_{i,j=1}^{L}M_{ij}\na\big(P_{2}^m(\zeta_{i}^{*}(\cc^{\ m}))\big)\cdot\na\big(P_{2}^m(\zeta_{j}^{*}(\cc^{\ m}))\big)\dx\\
& +\int_{\Omega}T_{\delta}(e^m)\mathfrak{T}_{\delta}(\cc^{\ m})\Big(
\sum_{i,j=1}^{L}M_{ij}\frac{z_{j}}{\theta^{m}}\na\varphi^m\cdot\na\big(P_{2}^m(\zeta_{i}^{*}(\cc^{\ m}))\big)
\\
&+\sum_{i=1}^{L}m_{i}\na\frac{1}{\theta^{m}}\cdot\na\big(P_{2}^m(\zeta_{i}^{*}(\cc^{\ m}))\big)\Big)\dx .
\end{align*}
What's more, again due to the elementary fact that $a(a-b)\geq -b^{2}/2$ for any $a,b\in\R$,
\begin{align}\label{qcG.1}
& \int_{\pa\Omega}\vec{q}_{\vec{c}\ \Gamma}^{\ m}\cdot P_{2}^m(\vzeta^{\ *}(\cc^{\ m}))\dS = \int_{\pa\Omega}\vec{q}_{\vec{c}\ \Gamma}^{\ m}\cdot \vzeta^{\ m}\dS \\
\nonumber & \qquad = \sum_{j=1}^{L}\int_{\pa\Omega}T_{\delta}(e^m)\mathfrak{T}_{\delta}(\cc^{\ m})D_{ij}\left(\zeta_{j}^{m} - \zeta_{j}^{\Gamma} + \frac{z_{j}}{\theta^{\Gamma}}(\varphi^m - \varphi^{\Gamma})\right)
\zeta_{j}^{m}\dx\\
\nonumber
&\qquad\geq -\frac{1}{2}\sum_{j=1}^{L}\int_{\pa\Omega}T_{\delta}(e^m)\mathfrak{T}_{\delta}(\cc^{\ m})\left(
\zeta_{j}^{\Gamma} - \frac{z_{j}}{\theta^{\Gamma}}(\varphi^m - \varphi^{\Gamma})\right)^{2}\dx\geq -C.
\end{align}
Moreover, by assumption,
\begin{align*}
& - \int_{\Omega}\vec{r}^{\ m}\cdot \vzeta^{\ m}\dx\geq 0.
\end{align*}
So, collecting all these inequalities yields
\begin{align}\label{int.sc}
& \frac{{\rm d}}{{\rm d}t} \int_{\Omega}s_{\cc}^{*,\eps,\delta}(\vec c^{\ m})\dx\geq -C
+\eps C\|\na \cc^{\ m}\|_{2}^{2}
-C\|\vv^{m}\|_{\infty}(1+\|s_{\cc}^{*,\eps,\delta}\|_{1}^{1/2})(1+\|\na \cc^{\ m}\|_{2})\\
\nonumber
&\qquad +\int_{\Omega}T_{\delta}(e^m)\mathfrak{T}_{\delta}(\cc^{\ m})\sum_{i,j=1}^{L}M_{ij}\na\zeta_{i}^{m}\cdot\na\zeta_{j}^{m}\dx\\
\nonumber
&\qquad +\int_{\Omega}T_{\delta}(e^m)\mathfrak{T}_{\delta}(\cc^{\ m})\Big(
\sum_{i,j=1}^{L}M_{ij}\frac{z_{j}}{\theta^{m}}\na\varphi^m\cdot\na\zeta_{i}^{m}
+\sum_{i=1}^{L}m_i\na\frac{1}{\theta^{m}}\cdot\na\zeta_{i}^{m}\Big)\dx  .
\end{align}
We get, summing \eqref{int.se} and \eqref{int.sc}
\begin{align*}
 \frac{{\rm d}}{{\rm d}t} &\int_{\Omega}s_{e}^{*,\delta}(e^{m}) + s_{\cc}^{*,\eps,\delta}(\cc^{\ m})\dx \\
&\geq -C +\eps C\|\na \cc^{\ m}\|_{2}^{2} -C\|\vv^{m}\|_{\infty}(1+\|s_{\cc}^{*,\eps,\delta}\|_{1}^{1/2})(1+\|\na \cc^{\ m}\|_{2})\\
& + \int_{\Omega}T_{\delta}(e^m)\mathfrak{T}_{\delta}(\vec{c}^{\ m})\Big(
\frac{\kappa|\na\theta^{m}|^{2}}{|\theta^{m}|^{2}} \\
& +\sum_{i,j=1}^{L}M_{ij}\left( \na\zeta_{i}^{\ m} + \frac{z_{i}}{\theta^{m}}\na\varphi^{m} \right)\cdot
\left( \na\zeta_{j}^{\ m} + \frac{z_{j}}{\theta^{m}}\na\varphi^{m} \right)
\Big)\dx.
\end{align*}
Young's inequality and the fact that of $s_{e}^{*}(e)$ is nonnegative
(see \eqref{prop.se2}) imply
\begin{align*}
& -C + \eps C\|\na \cc^{\ m}\|_{2}^{2} -C\|\vv^{m}\|_{\infty}(1+\|s_{\cc}^{*,\eps,\delta}(\cc^{\ m})\|_{1}^{1/2})(1+\|\na \cc^{\ m}\|_{2})\\
&\qquad\geq\eps \|\na \cc^{\ m}\|_{2}^{2} - \eps^{-1}C\left( 1 + \int_{\Omega}s_{\cc}^{*,\eps,\delta}(\cc^{\ m})\dx \right)\\
&\qquad\geq \eps \|\na \cc^{\ m}\|_{2}^{2} - \eps^{-1}C_{1}\left( 1 + \int_{\Omega}s_{\cc}^{*,\eps,\delta}(\cc^{\ m}) + s_{e}^{*,\delta}(e^{m})\dx \right).
\end{align*}
So we conclude
\begin{align}\label{eq.F}
& \frac{{\rm d}}{{\rm d}t} F + \eps^{-1}C_{1}F\geq
\eps \|\na \cc^{\ m}\|_{2}^{2}
+ \int_{\Omega}T_{\delta}(e^m)\mathfrak{T}_{\delta}(\vec{c}^{\ m})\frac{\kappa|\na\theta^{m}|^{2}}{|\theta^{m}|^{2}}\dx\\
\nonumber
& + \int_{\Omega}T_{\delta}(e^m)\mathfrak{T}_{\delta}(\vec{c}^{\ m})
\sum_{i,j=1}^{L}M_{ij}\left( \na\zeta_{i}^{m} + \frac{z_{i}}{\theta^{m}}\na\varphi^{m} \right)\cdot
\left( \na\zeta_{j}^{m} + \frac{z_{j}}{\theta^{m}}\na\varphi^{m} \right)\dx,
\end{align}
where
$$
F(t) \equiv 1 + \int_{\Omega}s_{e}^{*,\delta}(e^{m}(t)) + s_{\cc}^{*,\eps,\delta}(\cc^{\ m}(t))\dx .
$$
We get the following upper bound
\begin{align*}
&\eps \int_{0}^{t}\int_{\Omega}|\na \cc^{\ m}|^{2}\dx\dtau
+ \int_{0}^{t}\int_{\Omega}T_{\delta}(e^m)\mathfrak{T}_{\delta}(\vec{c}^{\ m})\frac{\kappa|\na\theta^{m}|^{2}}{|\theta^{m}|^{2}}\dx\dtau\\
\nonumber
& + \int_{0}^{t}\int_{\Omega}T_{\delta}(e^m)\mathfrak{T}_{\delta}(\vec{c}^{\ m})
\sum_{i,j=1}^{L}M_{ij}\left( \na\zeta_{i}^{m} + \frac{z_{i}}{\theta^{m}}\na\varphi^{m} \right)\cdot
\left( \na\zeta_{j}^{m} + \frac{z_{j}}{\theta^{m}}\na\varphi^{m} \right)\dx\dtau\\
\nonumber
&\qquad\leq e^{C t/\eps}\left(1 + \int_{\Omega}(s_{e}^{*,\delta}(e^{m}(t)) + s_{\cc}^{*,\eps,\delta}(\cc^{\ m}(t)))\dx\right),\qquad t>0.
\end{align*}
We know that $s_{e}^{*,\delta}(e)\leq K(1+e)$. Let $\widehat {\cc} \in\R^{L}_{+}$ be generic but fixed. Since $s_{\cc}^{*,\eps,\delta}$ is concave,
$$
s_{\cc}^{*}(\cc)\leq s_{\cc}^{*}(\widehat {\cc}) + \pa_{\cc} s_{\cc}^{*}(\widehat {\cc})\cdot (\cc - \widehat{\cc} ) \leq K(1 + |\cc|),\qquad \cc\in\R^{L}_{+},
$$
for some constant $K>0$ depending on $\widehat {\cc}$. As a consequence, from the above bounds and Young's inequality, we obtain
\begin{align}\label{up}
&\eps \int_{0}^{t}\int_{\Omega}|\na \cc^{\ m}|^{2}\dx\dtau
+ \int_{0}^{t}\int_{\Omega}T_{\delta}(e^m)\mathfrak{T}_{\delta}(\vec{c}^{\ m})\frac{\kappa|\na\theta^{m}|^{2}}{|\theta^{m}|^{2}}\dx\dtau\\
\nonumber
& + \int_{0}^{t}\int_{\Omega}T_{\delta}(e^m)\mathfrak{T}_{\delta}(\vec{c}^{\ m})
\sum_{i,j=1}^{L}M_{ij}\left( \na\zeta_{i}^{m} + \frac{z_{i}}{\theta^{m}}\na\varphi^{m} \right)\cdot
\left( \na\zeta_{j}^{m} + \frac{z_{j}}{\theta^{m}}\na\varphi^{m} \right)\dx\dtau\\
\nonumber
&\qquad\leq e^{C t/\eps}\left(\lambda^{-1} + \lambda\int_{\Omega}(|e^{m}(t)|^{2} + |\cc^{\ m}(t)|^{2})\dx\right),\qquad t>0,
\end{align}
for any $\lambda\in (0,1)$.

Let us choose $\vec\psi = \vec c^{\ m}$ in \eqref{b10} and $\psi=e^{m}$ in \eqref{b12}. It follows
\begin{align}\label{tmp.ce}
& \frac{{\rm d}}{{\rm d}t}\int_{\Omega}|e^{m}|^{2} + |\cc^{\ m}|^{2}\dx + \eps\int_{\Omega}|\na \cc^{\ m}|^{2} + |\na e^{m}|^{2}\dx\leq
\int_{\Omega}\vc{q}_{e}^m\cdot\na e^{m}\dx\\
\nonumber
&+\int_{\pa\Omega}q_{e\Gamma}^m e^m\dS + \int_{\Omega}e^{m}\mathcal{S}^m:\mathcal{D} (\vv^{m})\dx
-\int_{\Omega}e^{m}\vec{z}\cdot(\mathfrak{q}_{\cc}^m\na\varphi^{m})\dx \\
\nonumber
&+ \int_{\Omega}\mathfrak{q}_{\cc}^m:\na \cc^{\ m}\dx-\int_{\pa\Omega}\vec{q}_{\cc \ \Gamma}^{\ m}\cdot\vec{c}^{\ m}\dS + \int_{\Omega}\vec{r}^{\ m}\cdot\vec{c}^{\ m}\dx .
\end{align}
Let us estimate the terms on the right-hand side of \eqref{tmp.ce}. The lower bound for $\theta^{m}$ and Young's inequality imply
\begin{align*}
& \int_{\Omega}\vc{q}_{e}^m\cdot\na e^{m}\dx
= -\int_{\Omega}T_{\delta}(e^m)\mathfrak{T}_{\delta}(\cc^{\ m})\kappa\na\theta^{m}\cdot\na e^{m}\dx\\
&\qquad -\int_{\Omega}T_{\delta}(e^m)\mathfrak{T}_{\delta}(\cc^{\ m})\sum_{i=1}^{L} m_{i}\left( \na\zeta_{i}^m + \frac{z_{i}}{\theta^{m}}\na\varphi^{m} \right)\cdot\na e^{m}\dx\\
&\leq K\int_{\Omega}T_{\delta}(e^m)\mathfrak{T}_{\delta}(\cc^{\ m})\left| \na\vzeta^{\ m} + \frac{\vec{z}}{\theta^{m}}\na\varphi^{m} \right|^{2}\dx .
\end{align*}
Again, Young's inequality implies ($\eta\in(0,1)$ will be specified later)
\begin{align*}
&\int_{\pa\Omega}q_{e\Gamma}^m e^m\dS = -\int_{\Omega}\kappa\left(\frac{1}{\theta^{m}} - \frac{1}{\theta^{\Gamma}} \right)e^m\dS\leq
K\left(\eta^{-1}+\eta\int_{\Omega}|\na e^{m}|^{2}\dx\right),\\
&\int_{\Omega}e^{m}\mathcal{S}^m:\mathcal{D} (\vv^{m})\dx \leq K\left(\eta^{-1}+\eta\int_{\Omega}|\na e^{m}|^{2}\dx\right).
\end{align*}
Furthermore, since $e/\theta(e)$ is bounded as $e\to\infty$ and $\theta^{m}$ is strictly positive, we deduce
\begin{align*}
& -\int_{\Omega}e^{m}\vec{z}\cdot(\mathfrak{q}_{\cc}^m\na\varphi^{m})\dx =
\int_{\Omega}T_{\delta}(e^m)\mathfrak{T}_{\delta}(\cc^{\ m})e^{m}\sum_{i,j=1}^{L}z_{i}M_{ij}\left( \na\zeta_{j}^m + \frac{z_{j}}{\theta^{m}}\na\varphi^{m} \right)\na\varphi^{m}\dx\\
&\qquad -\int_{\Omega}T_{\delta}(e^m)\mathfrak{T}_{\delta}(\cc^{\ m})\sum_{i=1}^{L}m_{i}z_{i}\frac{e^{m}}{(\theta^{m})^{2}}\na\theta^{m}\cdot \na\varphi^{m}\dx\\
&\leq K\left(\eta^{-1} + \eta\int_{\Omega} |\na e^{m}|^{2}\dx
+\int_{\Omega}T_{\delta}(e^m)\mathfrak{T}_{\delta}(\cc^{\ m})\left| \na\vzeta^{\ m} + \frac{\vec{z}}{\theta^{m}}\na\varphi^{m} \right|^{2}\dx \right).
\end{align*}
Moreover,
\begin{align*}
& \int_{\Omega}\mathfrak{q}_{\cc}^m:\na \cc^{\ m}\dx
= -\int_{\Omega}T_{\delta}(e^m)\mathfrak{T}_{\delta}(\cc^{\ m})\sum_{i,j=1}^{L}M_{ij}\left( \na\zeta_{j}^m + \frac{z_{j}}{\theta^{m}}\na\varphi^{m} \right)\cdot\na c_{i}^{m}\dx\\
&-\int_{\Omega}T_{\delta}(e^m)\mathfrak{T}_{\delta}(\cc^{\ m})\sum_{i=1}^{L}m_{i}\na\frac{1}{\theta^{m}}\cdot\na c^{m}_{i}\dx \\
&\leq K\left( \eta^{-1} + \eta^{-1}\int_{\Omega}|\na \cc^{\ m}|^{2}\dx + \eta\int_{\Omega}|\na e^{m}|^{2}\dx \right).
\end{align*}
Furthermore,
\begin{align*}
& -\int_{\pa\Omega}\vec{q}_{\cc\ \Gamma}^{\ m}\cdot\vec{c}^{\ m}\dS = \int_{\pa\Omega}\mathfrak{D}\left( \vzeta^{\ m}-\vzeta^{\ \Gamma} + \frac{\vec{z}}{\theta^{m}}(\varphi^{m}-\varphi^{\Gamma}) \right)\cdot \cc^{\ m}\dS\\
&\leq K\left( 1 + \int_{\Omega}|\na \cc^{\ m}|^{2} + |\cc^{\ m}|^2 \dx \right),
\end{align*}
and
\begin{align*}
&  \int_{\Omega}\vec{r}^{\ m}\cdot\vec{c}^{\ m}\dx\leq K\left( 1 + \int_{\Omega}|\cc^{\ m}|^{2}\dx \right).
\end{align*}
Putting all the previous inequalities together yields, for any $\eta\in (0,1)$
\begin{align*}
& \frac{{\rm d}}{{\rm d}t} \int_{\Omega}|e^{m}|^{2} + |\cc^{\ m}|^{2}\dx + \eps\int_{\Omega}|\na \cc^{\ m}|^{2} + |\na e^{m}|^{2}\dx\\
& \leq K_{1}\int_{\Omega}|e^{m}|^{2} + |\cc^{\ m}|^{2}\dx
+K_{2}\left(\eta^{-1} + \eta^{-1}\int_{\Omega}|\na \cc^{\ m}|^{2}\dx + \eta\int_{\Omega}|\na e^{m}|^{2} \dx\right)\\
&\qquad + K_{3}\int_{\Omega}T_{\delta}(e^m)\mathfrak{T}_{\delta}(\cc^{ m})\left| \na\vzeta^{\ m} + \frac{\vec{z}}{\theta^{m}}\na\varphi^{m} \right|^{2}\dx .
\end{align*}
Choosing $\eta<\eps/K_{2}$ leads to
\begin{align*}
& \frac{{\rm d}}{{\rm d}t} \int_{\Omega}|e^{m}|^{2} + |\cc^{\ m}|^{2}\dx
\leq K_{1}\int_{\Omega}|e^{m}|^{2} + |\cc^{\ m}|^{2}\dx + K_{2}'\int_{\Omega}|\na \cc^{\ m}|^{2}\dx \\
&\qquad + K_{3}\int_{\Omega}T_{\delta}(e^m)\mathfrak{T}_{\delta}(\cc^{\ m})\left| \na\vzeta^{\ m} + \frac{\vec{z}}{\theta^{m}}\na\varphi^{m} \right|^{2}\dx .
\end{align*}
Applying Gronwall's inequality we deduce
\begin{align}\label{bound.ec2}
& \int_{\Omega}|e^{m}(t)|^{2} + |\cc^{\ m}(t)|^{2}\dx \leq K e^{K_{1}t} \left(1 + \int_{0}^{t}\int_{\Omega}|\na \cc^{\ m}|^{2}\dx\dtau\right.\\
\nonumber
&\left.\qquad + \int_{0}^{t}\int_{\Omega}T_{\delta}(e^m)\mathfrak{T}_{\delta}(\cc^{\ m})\left| \vzeta^{\ m} + \frac{\vec{z}}{\theta^{m}}\na\varphi^{m} \right|^{2}\dx\dtau \right),
\qquad t>0.
\end{align}
Putting \eqref{up} and \eqref{bound.ec2} together and choosing $\lambda>0$ small enough, we conclude
\begin{align}\label{up.2}
&\sup_{t \in (0,T)} \int_\Omega |e^m(t)|^2 + |\cc^{\ m}|^2 \dx +  \int_{\Omega_T}\eps |\na \cc^{\ m}|^{2}
+ T_{\delta}(e^m)\mathfrak{T}_{\delta}(\vec{c}^{\ m})\frac{\kappa|\na\theta^{m}|^{2}}{|\theta^{m}|^{2}}\dx\dtau\\
\nonumber
& + \int_{\Omega_T}T_{\delta}(e^m)\mathfrak{T}_{\delta}(\vec{c}^{\ m} )
\left| \na\vzeta^{\ m} + \frac{\vec{z}}{\theta^{m}}\na\varphi^{m} \right|\dx\dtau \leq K.
\end{align}
Returning back to \eqref{tmp.ce} we also have
\begin{equation} \label{up22}
\int_{\Omega_T} |\na e^m|^2 \dx \dt \leq K.
\end{equation}

Next, as above, we may show the following estimates of the time derivative
\begin{equation} \label{c1}
\int_0^T \|\pa_t \cc^{\ m} \|_{-1,2}^2 + \|\pa_t e^m\|_{-1,2}^2 + \|\pa_t \vv^m\|_{2}^2 +
\|\nabla \varphi_t\|_{2}^2 \dt \leq C
\end{equation}
and consequently
\begin{equation} \label{c2}
\begin{aligned}
\vv^m \to \vv \qquad & \mbox{strongly in }C([0,T];\vc{W}^n), \\
\pa_t \vv^m \to \pa_t \vv \qquad & \mbox{weakly in } L^2(\Omega_T;\R^3), \\
\cc^{\ m} \to \cc \qquad & \mbox{strongly in } L^2(\Omega_T;\R^L), \\
\cc^{\ m} \to \cc \qquad & \mbox{weakly in } L^2(0,T; W^{1,2}(\Omega;\R^L)), \\
\pa_t \cc^{\ m} \to \pa_t \cc \qquad & \mbox{weakly in } L^2(0,T; W^{-1,2}(\Omega;\R^L), \\
e^m \to e  \qquad & \mbox{strongly in } L^2(\Omega_T), \\
e^m \to e \qquad & \mbox{weakly in } L^2(0,T;W^{1,2}(\Omega)), \\
\pa_t e^m \to \pa_t e \qquad & \mbox{weakly in } L^2(0,T;W^{-1,2}(\Omega)), \\
\varphi^m \to \varphi \qquad & \mbox{strongly in } L^q(0,T;W^{1,q}(\Omega)), 1\leq q <\infty, \\
\varphi^m \to \varphi \qquad & \mbox{weakly in } L^q(0,T;W^{2,q}(\Omega)), 1\leq q<\infty, \\
\mathfrak{q}_{\cc}^m \to \mathfrak{q}_{\cc} \qquad & \mbox{weakly in } L^2(0,T;L^2(\Omega;\R^{3})), \\
\vc{q}_e^m \to \vc{q}_e \qquad & \mbox{weakly in } L^2(0,T;L^2(\Omega;\R^{3L})), \\
\vec{q}_{\cc\ \Gamma}^{\ m} \to \vec{q}_{\cc\ \Gamma} \qquad & \mbox{weakly in } L^2(0,T;L^2(\pa \Omega;\R^{L})), \\
{q}_{e\Gamma}^{m} \to {q}_{e\Gamma} \qquad & \mbox{weakly in } L^2(0,T;L^2(\pa \Omega)), \\
{q}_{\varphi\Gamma}^{m} \to {q}_{\varphi\Gamma} \qquad & \mbox{weakly}*\mbox{ in } L^2(0,T;L^\infty(\pa \Omega)), \\
\vzeta^{\, m} \to \vzeta \qquad & \mbox{weakly in } L^2(0,T;W^{1,2}(\Omega;\R^L)), \\
\mathcal{S}^m \to \mathcal{S} \qquad & \mbox{weakly in } L^{r'}(\Omega_T;\R^{9}).
\end{aligned}
\end{equation}
The limit functions fulfill
\begin{equation} \label{c3}
\begin{aligned}
& \langle\pa_t \cc,\vec{\psi}\rangle_{W^{-1,2}(\Omega);W^{1,2}(\Omega)}  + \int_{\Omega} \varepsilon \nabla \cc : \nabla \vec{\psi} - \big(\cc\otimes \vv + {\mathfrak q}_{\cc}\big) : \nabla \vec{\psi} \dx \\
&= -  \int_{\partial \Omega} \vec{q}_{\cc\,\Gamma} \cdot \vec{\psi} \dS + \int_{\Omega} \vec{r} \cdot \vec{\psi} \dx
\end{aligned}
\end{equation}
a.e. in $(0,T)$ and for all $\vec{\psi} \in W^{1,2}(\Omega;\R^L)$, $\lim_{t\to 0^+} \cc(t) = \cc_{0}^{\ \delta}$ in $L^2(\Omega;\R^L)$,
\begin{equation} \label{c4}
\begin{aligned}
& \int_\Omega \pa_t \vv\cdot \vc{u} + {\mathcal S} : \mathcal{D}(\vc{u}) - \xi^k(|\vv|^2)(\vc{v}\otimes \vv) : \nabla \vc{u}  \dx \\
&= -  \int_{\partial \Omega} \gamma\vv\cdot \vc{u} \dS - \int_{\Omega} Q\nabla \varphi \cdot \vc{u} \dx
\end{aligned}
\end{equation}
for all $\vc{u} \in \vc{W}^n$,
\begin{equation} \label{c5}
\begin{aligned}
& \langle\pa_t e,\psi\rangle_{W^{-1,2}(\Omega);W^{1,2}(\Omega)} + \int_{\Omega} \varepsilon \nabla e\cdot \nabla \psi - \big(e\vv + \vc{q}_e\big) \cdot \nabla \psi  \dx  \\
&= - \int_{\partial \Omega} q_{e\Gamma} \psi  \dS  + \int_0^T\int_{\Omega} \big(\mathcal{S}: \nabla \vv - \vec{z} \cdot (\mathfrak{q}_{\cc} \nabla \varphi)\big) \psi \dx
\end{aligned}
\end{equation}
a.e. in $(0,T)$ and for all $\psi \in W^{1,2}(\Omega)$, $\lim_{t\to 0^+} e(t) = e_{0}^{\eps,\delta}$ in $L^2(\Omega)$. For $\vv$, the initial condition is fulfilled in the same sense as in the previous step. Furthermore
\begin{equation} \label{c6}
\int_\Omega \nabla \varphi\cdot\nabla \psi - Q\psi\dx = \int_{\partial \Omega} q_{\varphi \Gamma} \psi \dS
\end{equation}
for all $\psi \in W^{1,2}(\Omega)$. As in the previous step, it is not difficult to verify that
\begin{equation} \label{c7}
{\mathfrak q}_{\cc} = {\mathfrak q}_{\cc}^{*,\delta} (e,\cc,\theta, \nabla \vzeta,\nabla \theta,\nabla \varphi),
\end{equation}
\begin{equation} \label{c8}
{\vc{q}}_{e} = \vc{q}_{e}^{*,\delta} (e, \cc,\theta,\nabla \vzeta,\nabla \theta, \nabla \varphi),
\end{equation}
\begin{equation} \label{c9}
{\mathcal S} = {\mathcal S}^* (\cc, \theta,{\mathcal D}(\vv)),
\end{equation}
\begin{equation} \label{c10}
\vec{r} = \vec{r}^{\ *,\delta} (e,\theta,\cc,\vzeta),
\end{equation}
\begin{equation} \label{c11}
Q = Q^{*,\delta} (\cc),
\end{equation}
the boundary fluxes
\begin{equation} \label{c12}
\begin{array}{l}
\displaystyle \vec{q}_{\cc\,\Gamma} = \vec{q}^{\ *,\delta}_{\cc\,\Gamma}(x,e,\cc,\theta,\vzeta,\varphi), \\
\displaystyle q_{e\Gamma}= q^{*,\delta}_{e\Gamma}(x,e,\cc,\theta,\vzeta),  \\
\displaystyle q_{\varphi\Gamma}= q^{*,\delta}_{\varphi\Gamma}(x,\varphi), \\
\displaystyle \gamma = \gamma^{*,\delta} (e,\cc,\theta),
\end{array}
\end{equation}
 and, in addition,
\begin{equation} \label{c13}
\theta =\theta^{*,\delta}(e),
\end{equation}
\begin{equation} \label{c14}
\vzeta= \vzeta^{\ *,\varepsilon,\delta}(\cc).
\end{equation}
Since the proof of the attainment of the initial conditions is rather standard we omit the details here.

\subsection{Limit $\delta\to 0$}
First, recall that exactly as in the previous section, we keep the minimum principle for the internal energy (and for the temperature). Hence, $e^\delta \geq \delta$ and $\theta^\delta=\theta^{*,\delta}(e^\delta)\geq C\delta$. Similarly, we can now establish the minimum principle for the concentrations.

To this aim, we use in \eqref{c3} as test function $\psi_i = \delta_{ij} \min\{0,c_j-\delta\}$ for $i=1,2,\dots,L$; recalling the cut-off in the fluxes for $c_i\leq \delta$ and the fact that $\Div \vv^\delta=0$, we deduce
$$
c_j \geq \delta, \qquad j=1,2,\dots, L.
$$
Furthermore, fixing an arbitrary function $u \in W^{1,2}(\Omega)$ and using as test function in \eqref{c3} $\vec{\psi} = u\vec{\ell}$, we get
\begin{equation} \label{d0}
\langle \pa_t (\cc^{\ \delta}\cdot \vec{\ell}),u\rangle_{W^{-1,2}(\Omega),W^{1,2}(\Omega)} + \eps \int_{\Omega} \na (\cc^{\ \delta}\cdot \vec{\ell})\cdot \na u - (\cc^{\ \delta}\cdot \vec{\ell}) \vv^\delta\cdot\na u \dx = 0
\end{equation}
a.e. in $(0,T)$. As $u\in W^{1,2}(\Omega)$ is arbitrary and $\cc_0^{\ \delta}\cdot \vec{\ell} \equiv 1$, we get
$$
\cc^{\ \delta}\cdot \vec{\ell} = 1
$$
a.e. in $\Omega_T$. Therefore $\cc^{\ \delta} \in G$ a.e. in $\Omega_T$ uniformly with respect to $\delta$, which means that for $\delta$ sufficiently small the "cut-off" $\mathfrak{T}_\delta (\cc) \equiv 1$ a.e. in $\Omega_T$. Moreover, we have
$$
\sup_{t \in (0,T)} \|\cc^{\ \delta}\|_{L^\infty(\Omega;\R^L)} \leq 1
$$
which implies that the definitions of $\vzeta^{\ \delta}$ and $\vzeta^{\ \eps,\delta}$ coincide.
We also keep in mind that estimates \eqref{b5}--\eqref{b6} (with $l$ replaced by $\delta$) still remain valid. Note, however, that due to the $L^\infty$-bound of $\cc^{\ \delta}$ the constants on the right-hand side of these estimates are actually independent of any parameter.

Due to the fact that we have already passed to the limit in the Galerkin approximations for the internal energy balance and the balance of species, we can now deduce an analogue of the entropy inequality (on the approximate level). Next, we also deduce the total energy balance and the sum of these two identities will allow us to obtain the required estimates independent of $\delta$. Moreover, we will be able to pass to the limit in these identities (getting, however, rather inequalities than equalities) and use them in order to get estimates in the following limit passages.

First, let us use as test function in \eqref{c5} $\psi: = \frac{1}{\theta^\delta} = \frac{1}{\theta^{*,\delta}(e^\delta)}$. Recall that this function belongs to $W^{1,2}(\Omega) \cap L^\infty(\Omega)$ and therefore it is an appropriate test function. We have
\begin{equation} \label{d1}
\begin{aligned}
&\frac{{\rm d}}{{\rm d}t} \int_\Omega s^{*,\delta}_e (e^\delta)\dx  + \int_\Omega \eps \pa^2_{e^2} s^{*,\eps}_e(e^\delta) |\na e^\delta|^2 - (e^\delta\vv^\delta + \vc{q}^\delta_e)\cdot \na \frac{1}{\theta^\delta} \dx \\
&=- \int _{\pa \Omega} q_{e\Gamma}^\delta\frac{1}{\theta^\delta} \dS + \int_\Omega \frac{\mathcal{S}^\delta:\mathcal{D}(\vv^\delta)}{\theta^\delta} - \frac{\vec{z}}{\theta^\delta}\cdot (\mathfrak{q}^{\ \delta}_{\cc} \na \varphi^\delta) \dx.
\end{aligned}
\end{equation}
Next, we use as test function in \eqref{c3} the function $\vec{\psi}:= -\vzeta^\delta = \pa_{\cc}s_{\cc}^{*,\delta}(\cc^{\ \delta})$. It yields
\begin{equation} \label{d2}
\begin{aligned}
&\frac{{\rm d}}{{\rm d}t} \int_\Omega s^{*,\delta}_{\cc}(\cc^{\ \delta})\dx + \int_\Omega \eps \nabla \cc^{\ \delta}: \na \vzeta^{\ \delta} + (\cc^{\ \delta} \otimes \vv^{\delta}+ \mathfrak{q}^{\ \delta}_{\cc}): \na \vzeta^{\ \delta} \dx \\
&= \int_{\pa \Omega} \vec{q}^{\ \delta}_{\cc\ \Gamma} \cdot \vzeta^{\ \delta} \dS - \int_\Omega \vec{r}^{\ \delta} \cdot  \vzeta^{\ \delta} \dx.
\end{aligned}
\end{equation}
Summing up \eqref{d1} and \eqref{d2} and using the specific form of the fluxes we end up with
\begin{equation}\label{d3}
\begin{aligned}
&\frac{{\rm d}}{{\rm d}t} \int_\Omega s^{*,\delta}_e (e^\delta) +s^{*,\delta}_{\cc}(\cc^{\ \delta})\dx + \eps \int_{\Omega} \pa^2_{e^2} s^{*,\eps}_e(e^\delta) |\na e^\delta|^2 - \nabla \cc^{\ \delta}: \na \vzeta^{\ \delta} \dx \\
& = \int_\Omega \frac{\mathcal{S}^\delta:\mathcal{D}(\vv^\delta)}{\theta^\delta} - \vec{r}^{\ \delta} \cdot \vzeta^{\ \delta} + T_\delta(e^\delta) \frac{\kappa |\na \theta^\delta|^2}{(\theta^\delta)^2} \dx \\
& + \int_\Omega T_\delta(e^\delta) \mathfrak{M}\Big(\na \vzeta^{\ \delta} + \frac{\vec{z}}{\theta^\delta} \na \varphi^\delta\Big)\Big(\na \vzeta^{\ \delta} + \frac{\vec{z}}{\theta^\delta} \na \varphi^\delta\Big) \dx \\
& + \int_{\pa \Omega} T_\delta(e^\delta) \kappa_\Gamma \Big(\frac {1} {\theta^\delta} - \frac {1} {\theta^\Gamma}\Big)\frac {1} {\theta^{\delta}} + T_\delta(e^\delta) \mathfrak{D} \Big(\vzeta^{\ \delta} -\vzeta^{\ \Gamma} + \frac{\vec{z}}{\theta^\Gamma} (\varphi^\delta - \varphi^\Gamma)\Big)\vzeta^\delta \dS.
\end{aligned}
\end{equation}
Next we deal with the total energy balance. We test the equation for the internal energy \eqref{c5} by $\chi_{[0,t]\times \Omega}$ (indeed, after suitable regularization in time), the momentum equation \eqref{c4} by $\vv^\delta$ and the "balance of the total charge"  by $\varphi$. The internal energy balance reads
\begin{equation} \label{d4}
\begin{aligned}
\frac{{\rm d}}{{\rm d}t} \|e^\delta\|_1 + \int_{\pa \Omega} q_{e\Gamma}^\delta \dS = \int_\Omega \mathcal{S}^\delta :\mathcal{D}(\vv^\delta) - \vec{z}\cdot (\mathfrak{q}^{\delta}_{\cc} \nabla \varphi^\delta) \dx,
\end{aligned}
\end{equation}
the kinetic energy balance is
\begin{equation} \label{d5}
\begin{aligned}
\frac 12 \frac{{\rm d}}{{\rm d}t} \|\vv^\delta\|_2^2 + \int_\Omega \mathcal{S}^\delta :\mathcal{D}(\vv^\delta)\dx  + \int_{\pa \Omega} \gamma|\vv^\delta|^2 \dS = -\int_\Omega Q^\delta \nabla \varphi^\delta \cdot \vv^\delta \dx.
\end{aligned}
\end{equation}
The total charge balance with the test function $\varphi^\delta$ has the form
\begin{equation} \label{d6}
\begin{aligned}
&\langle \pa_t Q^\delta,\varphi\rangle_{W^{-1,2}(\Omega);W^{1,2}(\Omega)}+ \int_\Omega \eps \na Q^\delta \cdot \na \varphi - Q^\delta \na \varphi^\delta \cdot \vv^\delta - \vec{z} \cdot (\mathfrak{q}^{\delta}_{\cc} \na \varphi^\delta) \dx  \\
= & \int_{\pa \Omega} \vec{q}^{\ \delta}_{\cc\ \Gamma} \cdot \vec{z} \varphi \dS + \int_\Omega \vec{r}^{\ \delta} \cdot \vec{z} \varphi \dx.
\end{aligned}
\end{equation}
Using \eqref{c6}, Hypothesis (H1) and the form of $Q^\delta=\vec{z} \cdot \vec{c}^{\ \delta}$, we get after summing \eqref{d4}--\eqref{d6}
\begin{equation} \label{d7}
\begin{aligned}
& \frac{{\rm d}}{{\rm d}t} \Big(\|e^\delta\|_1 + \frac 12 \|\vv^\delta\|_2^2 + \frac 12 \|\na \varphi^\delta\|_2^2 + \frac 12 \int_{\pa \Omega} \lambda^\Gamma |\varphi^\delta|^2 \dS \Big)  + \int_{\pa \Omega} \gamma|\vv^\delta|^2 \dS + \int_\Omega \eps |Q^\delta|^2 \dx = \\
&  \int_{\pa \Omega} T_\delta(e^\delta) \kappa^\Gamma \left( \frac{1}{\theta^{\delta}} - \frac{1}{\theta^{\Gamma}} \right) + \eps Q^\delta \lambda^\Gamma (\varphi^\delta-\varphi^\Gamma) + T_\delta(e^\delta)\mathfrak{D} \Big(\vzeta^{\ \delta} -\vzeta^{\ \Gamma} + \frac{\vec{z}}{\theta^\Gamma}(\varphi^\delta -\varphi^\Gamma)\vec{z} \varphi^\delta \dS.
\end{aligned}
\end{equation}
Subtracting \eqref{d3} from \eqref{d7} we end up with
\begin{equation} \label{d8}
\begin{aligned}
& \frac{{\rm d}}{{\rm d}t} \Big(\|e^\delta\|_1 + \frac 12 \|\vv^\delta\|_2^2 + \frac 12 \|\na \varphi^\delta\|_2^2 -\int_\Omega s^{*,\delta}_e (e^\delta) +s^{*,\delta}_{\cc}(\cc^{\ \delta})\dx  + \frac 12 \int_{\pa \Omega} \lambda^\Gamma |\varphi^\delta|^2 \dS \Big) \\
& +\int_\Omega \eps |Q^\delta|^2 -  \eps\pa^2_{e^2} s^{*,\eps}_e(e^\delta) |\na e^\delta|^2 + \eps\nabla \cc^{\ \delta}: \na \vzeta^{\ \delta} +\frac{\mathcal{S}^\delta:\mathcal{D}(\vv^\delta)}{\theta^\delta}  + T_\delta(e^\delta) \frac{\kappa |\na \theta^\delta|^2}{(\theta^\delta)^2}   \dx \\
& + \int_\Omega T_\delta(e^\delta) \mathfrak{M}\Big(\na \vzeta^{\ \delta} + \frac{\vec{z}}{\theta^\delta} \na \varphi^\delta\Big)\Big(\na \vzeta^{\ \delta} + \frac{\vec{z}}{\theta^\delta} \na \varphi^\delta\Big) \dx \\
& + \int_{\pa \Omega} \gamma|\vv^\delta|^2  + T_\delta(e^\delta) \kappa_\Gamma \Big(\frac {1} {\theta^\delta} - \frac {1} {\theta^\Gamma}\Big)^2 \dS  \\
& + \int_{\pa \Omega} T_\delta(e^\delta) \mathfrak{D} \Big(\vzeta^{\ \delta} -\vzeta^{\ \Gamma} + \frac{\vec{z}}{\theta^\Gamma} (\varphi^\delta - \varphi^\Gamma)\Big) \cdot \Big(\vzeta^{\ \delta} -\vzeta^{\ \Gamma} + \frac{\vec{z}}{\theta^\Gamma} (\varphi^\delta - \varphi^\Gamma)\Big) \dS \\
& = \int_{\Omega} \vec{r}^{\ \delta} \cdot \vzeta^{\ \delta} \dx -\int_{\pa \Omega} T_\delta(e^\delta) \kappa^\Gamma \left( \frac{1}{\theta^{\delta}} - \frac{1}{\theta^{\Gamma}} \right)\left(\frac{1}{\theta^{\Gamma}}-1\right) + \eps Q^\delta \lambda^\Gamma (\varphi^\delta-\varphi^\Gamma) \dS \\
& + \int_{\pa \Omega} T_\delta(e^\delta)\mathfrak{D} \Big(\vzeta^{\ \delta} -\vzeta^{\ \Gamma} + \frac{\vec{z}}{\theta^\Gamma}(\varphi^\delta -\varphi^\Gamma)\Big)\cdot \Big(\vzeta^{\ \Gamma} +\frac{\vec{z}}{\theta^\Gamma} \varphi^\delta\Big) \dS \\
& + \int_{\pa \Omega} T_\delta(e^\delta)\mathfrak{D} \Big(\vzeta^{\ \delta} -\vzeta^{\ \Gamma} + \frac{\vec{z}}{\theta^\Gamma}(\varphi^\delta -\varphi^\Gamma)\Big)\cdot \frac{\vec{z}} \varphi^\delta\Big(1 -\frac{1}{\theta^\Gamma} \Big) \dS.
\end{aligned}
\end{equation}
Using our Hypotheses we read from \eqref{d8} the following estimates uniform in $\delta$ (but also in $n$, $\eps$ and $k$)
\begin{equation} \label{d9}
\begin{aligned}
& \|e^\delta\|_{L^\infty(0,T;L^1(\Omega))} + \|\vv^\delta\|_{L^\infty(0,T;L^2(\Omega;\R^3))} + \|\varphi^\delta\|_{L^\infty(0,T;W^{1,2}(\Omega))} + \eps \Big\|\frac{\na e^\delta}{e^\delta}\Big\|_{L^2(\Omega_T;\R^3)}  \\
& + \eps \|\nabla \cc^{\ \delta}\|_{L^2(\Omega_T;\R^{3L})} + \int_{\Omega_T} T_\delta(e_\delta) M(\theta) \Big|\mathfrak{P}_{\vec{\ell}}\Big(\na \vzeta^{\ \delta} + \frac{\vec{z}}{\theta^\delta} \na \varphi^\delta\Big)\Big|^2 + T_\delta(e^\delta) \frac{\kappa|\na \theta^\delta|^2}{(\theta^\delta)^2} \dx \dt
\\
& +\int_\Gamma  T_\delta(e^\delta) \Big|\mathfrak{P}_{\vec{\ell}} \Big(\vzeta^{\ \delta} -\vzeta^{\ \Gamma} + \frac{\vec{z}}{\theta^\Gamma}(\varphi^\delta -\varphi^\Gamma) \Big)\Big|^2 + \frac{\overline{\kappa}}{(\theta^\delta)^2} \dS  \dt \leq C.
\end{aligned}
\end{equation}
Furthermore, we may use $e^\delta$ as test function in \eqref{c5} the function  and get
\begin{equation} \label{d10}
\|e^\delta\|_{L^\infty(0,T;L^2(\Omega))} + \eps \|\na e^\delta\|_{L^2(0,T;L^2(\Omega))} \leq C.
\end{equation}
Finally, for the time derivatives we have
\begin{equation} \label{d11}
\begin{aligned}
&\|\pa_t e^\delta\|_{L^2(0,T;W^{-1,2}(\Omega))} + \|\pa_t \vv^\delta \|_{L^2(Q_T;\R^3)}  + \|\pa_t \cc^{\ \delta} \|_{(L^2(0,T;W^{1,2}(\Omega;\R^L)) \cap L^{a'}(0,T;W^{1,a'}(\Omega;\R^L)))'} \\
& + \|\na \varphi^\delta\|_{(L^2(0,T;W^{1,2}(\Omega;\R^L)) \cap L^{a'}(0,T;W^{1,a'}(\Omega;\R^L)))'} \leq C,
\end{aligned}
\end{equation}
where $a = \min\{\beta, \frac{2\beta}{2\beta-\eps_0}\}$, see \eqref{aaa}, and $a' = \frac{a}{a-1}$. Note that the reason for worse estimate of the time derivative of $\cc$ (and, consequently, of $\na \varphi$) is connected with the integrability of $\mathfrak{q}_{\cc}^\delta$. More precisely, the most restrictive term is $m \na \frac{1}{\theta^\delta}$.
We may now repeat the computations from Subsections 3.4 and 3.5 to get the same estimates on the  integrability of the fluxes now denoted by $\delta$. We have shown that for $\delta \to 0^+$ we have (possibly for subsequences)
\begin{equation} \label{d12}
\begin{aligned}
\vv^\delta \to \vv \qquad & \mbox{strongly in }C([0,T];\vc{W}^n), \\
\pa_t \vv^\delta \to \pa_t \vv \qquad & \mbox{weakly in } L^2(\Omega_T;\R^3), \\
\cc^{\ \delta} \to \cc \qquad & \mbox{strongly in } L^q(\Omega_T;\R^L), 1\leq q<\infty\\
\cc^{\ \delta} \to \cc \qquad & \mbox{weakly in } L^2(0,T; W^{1,2}(\Omega;\R^L)), \\
\pa_t \cc^{\ \delta} \to \pa_t \cc \qquad & \mbox{weakly in } (L^2(0,T;W^{1,2}(\Omega;\R^L)) \cap L^{a'}(0,T;W^{1,a'}(\Omega;\R^L)))', \\
e^\delta \to e  \qquad & \mbox{strongly in } L^2(\Omega_T), \\
e^\delta \to e \qquad & \mbox{weakly in } L^2(0,T;W^{1,2}(\Omega)), \\
\pa_t e^\delta \to \pa_t e \qquad & \mbox{weakly in } L^2(0,T;W^{-1,2}(\Omega)), \\
\varphi^\delta \to \varphi \qquad & \mbox{strongly in } L^q(0,T;W^{1,q}(\Omega)), 1\leq q <\infty, \\
\varphi^\delta \to \varphi \qquad & \mbox{weakly in } L^q(0,T;W^{2,q}(\Omega)), 1\leq q<\infty, \\
\mathfrak{q}_{\cc}^{\ \delta} \to  \mathfrak{q}_{\cc} \qquad & \mbox{weakly in } L^q(Q_T; \R^{3L}), q \mbox{ from \ (\ref{qq1})}, \\
\vc{q}_e^\delta \to \vc{q}_e \qquad & \mbox{weakly in } L^a(Q_T; \R^{3}), a \mbox{ from (\ref{aaa})}, \\
\vec{q}_{\cc \ \Gamma}^{\ \delta} \to \vec{q}_{\cc \ \Gamma} \qquad & \mbox{weakly in } L^2(\Gamma; \R^{L}), \\
q_{e\Gamma}^\delta \to q_{e\Gamma} \qquad & \mbox{weakly in } L^2(\Gamma), \\
q_{\varphi\Gamma}^\delta \to q_{\varphi\Gamma} \qquad & \mbox{weakly in } L^2(0,T; L^\infty(\pa\Omega), \\
\vzeta^{\ \delta} \to \vzeta \qquad & \mbox{weakly in } L^a(0,T;W^{1,a}(\Omega;\R^L)), \\
\mathcal{S}^\delta \to \mathcal{S} \qquad & \mbox{weakly in } L^{r'}(\Omega_T),
\end{aligned}
\end{equation}
and the limit functions fulfill
\begin{equation} \label{d13}
\begin{aligned}
& \langle\pa_t \cc,\vec{\psi}\rangle_{W^{-1,a}(\Omega);W^{1,a'}(\Omega)}  + \int_{\Omega} \varepsilon \nabla \cc : \nabla \vec{\psi} - \big(\cc\otimes \vv + {\mathfrak q}_{\cc}\big) : \nabla \vec{\psi} \dx \\
&= -  \int_{\partial \Omega} \vec{q}_{\cc\,\Gamma} \cdot \vec{\psi} \dS + \int_{\Omega} \vec{r} \cdot \vec{\psi} \dx
\end{aligned}
\end{equation}
a.e. in $(0,T)$ and for all $\vec{\psi} \in W^{1,a'}(\Omega;\R^L)$, $\lim_{t\to 0^+} \cc\,(t) = \cc^{\ 0}$ in $L^2(\Omega;\R^L)$,
\begin{equation} \label{d14}
\begin{aligned}
& \int_\Omega \pa_t \vv\cdot \vc{u} + {\mathcal S} : \mathcal{D}(\vc{u}) - \xi^k(|\vv|^2)(\vc{v}\otimes \vv) : \nabla \vc{u}  \dx \\
&= -  \int_{\partial \Omega} \gamma\vv\cdot \vc{u} \dS - \int_{\Omega} Q\nabla \varphi \cdot \vc{u} \dx
\end{aligned}
\end{equation}
 for all $\vc{u} \in \vc{W}^n$, the initial condition for $\vv$ is fulfilled in the same sense as above,
\begin{equation} \label{d15}
\begin{aligned}
& \langle\pa_t e,\psi\rangle_{W^{-1,2}(\Omega);W^{1,2}(\Omega)} + \int_{\Omega} \varepsilon \nabla e\cdot \nabla \psi - \big(e\vv + \vc{q}_e\big) \cdot \nabla \psi  \dx  \\
&= - \int_{\partial \Omega} q_{e\Gamma} \psi  \dS  + \int_0^T\int_{\Omega} \big(\mathcal{S}: \nabla \vv - \vec{z} \cdot (\mathfrak{q}_{\cc} \nabla \varphi)\big) \psi \dx
\end{aligned}
\end{equation}
a.e. in $(0,T)$ and for all $\psi \in W^{1,2}(\Omega)$, $\lim_{t\to 0^+} e(t) = e^{0}$ in $L^2(\Omega)$,
\begin{equation} \label{d16}
\int_\Omega \nabla \varphi\cdot\nabla \psi - Q\psi\dx = \int_{\partial \Omega} q_{\varphi \Gamma} \psi \dS
\end{equation}
for all $\psi \in W^{1,2}(\Omega)$, where
\begin{equation} \label{d17}
{\mathfrak q}_{\cc} = {\mathfrak q}_{\cc}^{*} (e, \cc,\theta,\nabla \vzeta,\nabla \theta,\nabla \varphi),
\end{equation}
\begin{equation} \label{d18}
{\vc{q}}_{e} = \vc{q}_{e}^{*} (e,\cc,\theta,\nabla \vzeta,\nabla \theta, \nabla \varphi),
\end{equation}
\begin{equation} \label{d19}
{\mathcal S} = {\mathcal S}^* (\cc,\theta,{\mathcal D}(\vv)),
\end{equation}
\begin{equation} \label{d20}
\vec{r} = \vec{r}^{\ *} (e,\cc,\theta,\vzeta),
\end{equation}
\begin{equation} \label{d21}
Q = Q^{*} (\cc),
\end{equation}
the boundary fluxes
\begin{equation} \label{d22}
\begin{array}{l}
\displaystyle \vec{q}_{\cc\,\Gamma} = \vec{q}^{\ *}_{\cc\,\Gamma}(x,e,\cc,\theta,\vzeta,\varphi), \\
\displaystyle q_{e\Gamma}= q^{*}_{e\Gamma}(x,e,\cc,\theta,\vzeta),  \\
\displaystyle q_{\varphi\Gamma}= q^{*}_{\varphi\Gamma}(x,\varphi), \\
\displaystyle \gamma = \gamma^{*} (e,\cc,\theta).
\end{array}
\end{equation}
Furthermore,
\begin{equation} \label{d22a}
\theta = \theta^*(e),
\end{equation}
\begin{equation} \label{d22b}
\vzeta = \vzeta^{\ *}(\cc).
\end{equation}

Note, however, that statements \eqref{d13}--\eqref{d22} require more careful proof than in the previous limit passages, where the situation was much simpler. On the other hand, we may follow directly the approach from \cite{BuHa15}. The main problem is connected with possible degeneracy of the temperature and concentrations and we have to control the size of the subset of $\Omega_T$, where some of these quantities are less than say $\delta_0$. Due to the estimates above, the size of this singular set can be controlled by $C/\ln \delta_0$ and this allows to show the limit passages in the fluxes.

Note also that we may pass to the limit in the entropy identity to get the entropy inequality

\begin{equation}\label{d23}
\begin{aligned}
&\int_\Omega (s^{*}_e (e) +s^{*}_{\cc}(\cc))(t)\dx  \geq  \int_\Omega (s^{*}_e (e) +s^{*}_{\cc}(\cc))(0)\dx \\
&-\eps \int_{\Omega_T} \pa^2_{e^2} s^{*,\eps}_e(e) |\na e|^2 - \nabla \cc: \na \vzeta \dx\dt \\
& + \int_{\Omega_T} \frac{\mathcal{S}:\mathcal{D}(\vv)}{\theta} - \vec{r} \cdot \vzeta + \frac{\kappa |\na \theta|^2}{(\theta)^2}  + \mathfrak{M}\Big(\na \vzeta + \frac{\vec{z}}{\theta} \na \varphi\Big)\Big(\na \vzeta + \frac{\vec{z}}{\theta} \na \varphi\Big) \dx \dt\\
& + \int_{\Gamma}  \kappa_\Gamma \Big(\frac {1} {\theta} - \frac {1} {\theta^\Gamma}\Big)\frac {1} {\theta} + \mathfrak{D} \Big(\vzeta -\vzeta^{\ \Gamma} + \frac{\vec{z}}{\theta^\Gamma} (\varphi - \varphi^\Gamma)\Big)\vzeta \dS \dt.
\end{aligned}
\end{equation}

\subsection{Limit passage $n \to \infty$ and $\eps \to 0^+$} We take e.g. $\eps_n = \frac{1}{n}$ and consider the limit passage $n \to \infty$. Let us first collect estimates independent of $n$. Recalling that
\begin{equation} \label{e1}
\|\cc^{\ n}\|_{L^\infty(\Omega_T;\R^L)} \leq 1
\end{equation}
we can read from \eqref{d16}
\begin{equation} \label{e2}
\|\varphi^n\|_{L^\infty (0,T;W^{2,q}(\Omega))} \leq C(q)
\end{equation}
for any $1\leq q<\infty$ as well from \eqref{d14} with test function $\vv^n$
\begin{equation} \label{e3}
\|\vv^n\|_{L^\infty(0,T;L^2(\Omega;\R^3))} + \|\na \vv^n \|_{L^r(\Omega_T;\R^9)} + \|\mathcal{S}^n \|_{L^{r'}(\Omega_T;\R^9)}  \leq C.
\end{equation}
Next, we read from \eqref{d9}
\begin{equation} \label{e4}
\begin{aligned}
& \sup_{t\in (0,T)} \|e^n(t)\|_{L^1(\Omega_T)}+\int_{Q_T} M(\theta) \Big|\mathfrak{P}_{\vec{\ell}} \Big(\na \vzeta^{\ n} + \frac{\vec{z}}{\theta^n} \na \varphi^n\Big)\Big|^2  + \frac{\kappa |\nabla \theta^n|^2}{|\theta^n|^2} \dx \dt \\
& + \int_\Gamma  \Big|\mathfrak{P}_{\vec{\ell}} \Big(\vzeta^{\ n} -\vzeta^{\ \Gamma} + \frac{\vec{z}}{\theta^\Gamma}(\varphi^n -\varphi^\Gamma) \Big)\Big|^2 + \frac{\overline{\kappa}}{|\theta^n|^2} \dS \dt
\leq C.
\end{aligned}
\end{equation}
Further, we may repeat computations between \eqref{eq.e1}--\eqref{ep.theta.2} to conclude that for any $\lambda >0$
\begin{equation} \label{e5}
\int_{Q_T} |\theta^n|^{\frac 53 -\lambda} + |\na \theta^n|^{\frac 54 -\lambda} + \frac{|\na \theta^n|^2}{(1+\na\theta^n)^{1+\lambda}} \dx \dt \leq C(\lambda)
\end{equation}
and
\begin{equation} \label{e6}
\|\mathfrak{q}_{\cc}^n \|_{L^{q_1}(\Omega_T;\R^{3L})} + \|\vc{q}_e^n\|_{L^{q_2}(\Omega_T;\R^3)} + \|\vzeta^{\ n}\|_{L^{q_3}(0,T;W^{1,q_3}(\Omega;\R^L))} + \|\cc^{\ n}\|_{L^{q_3}(0,T;W^{1,q_3}(\Omega;\R^L))} \leq C,
\end{equation}
where
$$
q_1 < \min\Big\{\frac{10}{10-3\eps_0}, \frac{2\beta}{2\beta-\eps_0}\Big\}, \quad q_2 = \min\Big\{\frac 54, \frac{2\beta}{2\beta-\eps_0}\Big\}, \quad q_3 = \min\Big\{\beta, \frac{2\beta}{2\beta-\eps_0}\Big\}.
$$
Next, we compute the estimates of the time derivatives. We can show
\begin{equation} \label{e7}
\begin{aligned}
&\|\pa_t \vv^n\|_{(L^{r}(0,T;W^{1,r}_{0,\nu}(\Omega) \cap W^{1,r}_{\Div}(\Omega)))'} + \|\pa_t \cc^{\ n}\|_{(L^2(0,T;W^{1,2}(\Omega)) \cap L^{a'}(0,T;W^{1,a'}(\Omega)))'} \\
&+ \|\pa_t e^n\|_{L^1(0,T; (W^{2,3}(\Omega))^*)} + \|\pa_t \nabla \varphi^n\|_{(L^2(0,T;W^{1,2}(\Omega)) \cap L^{a'}(0,T;W^{1,a'}(\Omega)))'}\leq C.
\end{aligned}
\end{equation}

Therefore, we have (with a possible choice of a subsequence)
\begin{equation} \label{e8}
\begin{aligned}
\vv^n \to \vv \qquad & \mbox{weakly in }L^r(0,T;W^{1,r}(\Omega;\R^3)), \\
\vv^n \to \vv \qquad & \mbox{weakly}^* \mbox{ in }L^\infty(0,T;L^2(\Omega;\R^3)), \\
\pa_t \vv^n \to \pa_t \vv \qquad & \mbox{weakly in } (L^{r}(0,T;W^{1,r}_{0,\nu}(\Omega) \cap W^{1,r}_{\Div}(\Omega)))', \\
\cc^{\ n} \to \cc \qquad & \mbox{strongly in } L^q(\Omega_T;\R^L), 1\leq q<\infty,\\
\cc^{\ n} \to \cc \qquad & \mbox{weakly in } L^{q_3}(0,T; W^{1,{q_3}}(\Omega;\R^L)), \\
\pa_t \cc^{\ n} \to \pa_t \cc \qquad & \mbox{weakly in } (L^2(0,T;W^{1,2}(\Omega;\R^L)) \\
& \qquad  \cap L^{a'}(0,T;W^{1,a'}(\Omega;\R^L)))', \\
e^n \to e  \qquad & \mbox{strongly in } L^q(\Omega_T),  q<\frac 53,\\
e^n \to e \qquad & \mbox{weakly in } L^q(0,T;W^{1,q}(\Omega)), q< \frac 54,\\
\pa_t e^n \to \pa_t e \qquad & \mbox{weakly in } \mathcal{M}(0,T;(W^{3,2}(\Omega))^*), \\
\varphi^n \to \varphi \qquad & \mbox{strongly in } L^q(0,T;W^{1,q}(\Omega)), 1\leq q <\infty, \\
\varphi^n \to \varphi \qquad & \mbox{weakly in } L^q(0,T;W^{2,q}(\Omega)), 1\leq q<\infty, \\
\mathfrak{q}_{\cc}^{\ n} \to  \mathfrak{q}_{\cc} \qquad & \mbox{weakly in } L^{q_1}(Q_T; \R^{3L}), \\
\vc{q}_e^n \to \vc{q}_e \qquad & \mbox{weakly in } L^{q_2}(Q_T; \R^{3}),  \\
\vec{q}_{\cc \ \Gamma}^{\ n} \to \vec{q}_{\cc \ \Gamma} \qquad & \mbox{weakly in } L^2(\Gamma; \R^{L}), \\
q_{e\Gamma}^n \to q_{e\Gamma} \qquad & \mbox{weakly in } L^2(\Gamma), \\
q_{\varphi\Gamma}^n \to q_{\varphi\Gamma} \qquad & \mbox{weakly in } L^2(0,T; L^\infty(\pa\Omega), \\
\vzeta^{\ n} \to \vzeta \qquad & \mbox{weakly in } L^a(0,T;W^{1,a}(\Omega;\R^L)), \\
\mathcal{S}^n \to \mathcal{S} \qquad & \mbox{weakly in } L^{r'}(\Omega_T).
\end{aligned}
\end{equation}
The limit functions satisfy
\begin{equation} \label{e9}
\begin{aligned}
& \langle\pa_t \cc,\vec{\psi}\rangle_{W^{-1,a}(\Omega);W^{1,a'}(\Omega)}  - \int_{\Omega}  \big(\cc\otimes \vv + {\mathfrak q}_{\cc}\big) : \nabla \vec{\psi} + \vec{r} \cdot \vec{\psi} \dx
= -  \int_{\partial \Omega} \vec{q}_{\cc\,\Gamma} \cdot \vec{\psi} \dS
\end{aligned}
\end{equation}
a.e. in $(0,T)$ and for all $\vec{\psi} \in W^{1,a'}(\Omega;\R^L)$; $\lim_{t\to 0^+} \cc(t) = \cc^{0}$ in $L^q(\Omega;\R^L)$, $1\leq q<\infty$ arbitrary,
\begin{equation} \label{e10}
\begin{aligned}
& \langle\pa_t \vv,\vc{u}\rangle_{W^{-1,{r'}}(\Omega);W^{1,r}(\Omega)}  + \int_\Omega   {\mathcal S} : \mathcal{D}(\vc{u}) - \xi^k(|\vv|^2)(\vc{v}\otimes \vv) : \nabla \vc{u}  \dx \\
&= -  \int_{\partial \Omega} \gamma\vv\cdot \vc{u} \dS - \int_{\Omega} Q\nabla \varphi \cdot \vc{u} \dx
\end{aligned}
\end{equation}
a.e. in $(0,T)$, for all $\vc{u} \in W^{1,r}_{0,\nu}(\Omega) \cap W^{1,r}_{\Div}(\Omega)$, $\lim_{t\to 0^+} \vv(t) = \vv^{0}$ in $L^2(\Omega;\R^3)$,
\begin{equation} \label{e11}
\begin{aligned}
& \langle\pa_t e,\psi\rangle_{W^{-1,2}(\Omega);W^{1,2}(\Omega)} - \int_{\Omega} \big(e\vv + \vc{q}_e\big) \cdot \nabla \psi -\big(\mathcal{S}: \mathcal{D}(\vv) - \vec{z} \cdot (\mathfrak{q}_{\cc} \nabla \varphi)\big) \psi \dx  \\
&= - \int_{\partial \Omega} q_{e\Gamma} \psi  \dS
\end{aligned}
\end{equation}
a.e. in $(0,T)$ and for all $\psi \in W^{1,2}(\Omega)$, $\lim_{t\to 0^+} e(t) = e^{0}$ in $L^1(\Omega)$. However, this must be shown more carefully. Next,
\begin{equation} \label{e12}
\int_\Omega \nabla \varphi\cdot\nabla \psi - Q\psi\dx = \int_{\partial \Omega} q_{\varphi \Gamma} \psi \dS.
\end{equation}
Further, the main difficulty is to show that
\begin{equation} \label{e13}
\mathcal{S} = \mathcal{S}^*(\cc,\theta,\mathcal{D}(\vv))
\end{equation}
as well as the fact that $\mathcal{S}^n :\mathcal{D}(\vv^n) \to \mathcal{S}^*(\cc,\theta,\mathcal{D}(\vv)): \mathcal{D}(\vv)$ weakly in $L^1(\Omega_T)$.
As there is still a cut-off in the  convective term, we can verify the limit passage using the Minty trick only. More precisely, due to estimates above we know that
\begin{equation} \label{e14}
\limsup_{n\to \infty} \int_{Q_T} \mathcal{S}^n: \mathcal{D}(\vv^n)\dx\dt \leq \int_{Q_T} \mathcal{S}: \mathcal{D}(\vv)\dx\dt
\end{equation}
as well as
\begin{equation} \label{e15}
\mathcal{S}^*(\cc^{\ n}, \theta^n,\mathcal{B}) \to \mathcal{S}^*(\cc, \theta,\mathcal{B})
\end{equation}
strongly in $L^{r'}(0,T;L^{r'}(\Omega;\R^9))$ for any  symmetric matrix-valued function $\mathcal{B}$ being in the space  $L^{r}(0,T;L^{r}(\Omega;\R^9))$. The monotonicity of the stress tensor implies
\begin{equation} \label{e16}
\int_{\Omega_T} \big(\mathcal{S}^n-\mathcal{S}^*(\cc^{\ n},\theta^n,\mathcal{B})\big):(\mathcal{D}(\vv^n)-\mathcal{B})\dx\dt \geq 0;
\end{equation}
whence
\begin{equation} \label{e17}
\int_{\Omega_T} \big(\mathcal{S}-\mathcal{S}^*(\cc,\theta,\mathcal{B})\big):(\mathcal{D}(\vv)-\mathcal{B})\dx\dt \geq 0
\end{equation}
for all $\mathcal{B} \in L^{r'}(0,T;L^{r'}(\Omega;\R^9))$, symmetric. The Minty trick gives us
\begin{equation} \label{e18}
\mathcal{S} = \mathcal{S}^*(\cc,\theta,\mathcal{D}(\vv))
\end{equation}
as well as
\begin{equation} \label{e19}
(\mathcal{S}^n -\mathcal{S}^*(\cc^{\ n},\theta^n,\mathcal{D}(\vv))):(\mathcal{D}(\vv^n)-\mathcal{D}(\vv)) \to 0
\end{equation} strongly in $L^1(Q_T)$. This implies
\begin{equation} \label{e20}
\mathcal{S}^n :\mathcal{D}(\vv^n) \to \mathcal{S}:\mathcal{D}(\vv)
\end{equation}
weakly in $L^1(Q_T)$. We may use this information to conclude that $\pa_t e^n \to \pa_t e$ weakly in $L^1(0,T;(W^{2,3}(\Omega))')$ and therefore $e \in C([0,T];(W^{2,3}(\Omega))')$. Moreover, using this information and the weak formulation of the internal energy balance we conclude
\begin{equation} \label{e21}
\int_{\Omega} (e(t)-e_0)\psi \dx \to 0
\end{equation}
for $t\to 0^+$ for all $\psi \in C(\overline{\Omega})$, i.e.
\begin{equation} \label{e22}
e(t) \to e_0
\end{equation}
weakly$^*$ in $\mathcal{M}(\Omega)$. Then, exactly as in \cite{BuHa15}, using the Biting lemma, we may conclude that
\begin{equation} \label{e23}
e(t)\to e_0 \quad \mbox{for } t\to 0^+ \mbox{ weakly in } L^1(\Omega).
\end{equation}

Next, we employ the assumption \eqref{hp.sc.2} to verify that
\begin{equation} \label{e24}
\vzeta = \vzeta^*(\cc)
\end{equation}
and
\begin{equation} \label{e25}
\vzeta^{\ n} \to \vzeta
\end{equation}
strongly in $L^1(\Omega_T;\R^L)$. Recalling the previous limit passage we can show that
$$
\mathfrak{T}_\delta(\cc^{\ n})\vzeta^{\ n} \to \mathfrak{T}_\delta(\cc) \vzeta
$$
strongly in $L^1(\Omega_T;\R^L)$ for any fix $\delta >0$. However, due to \eqref{hp.sc.2}, inequality \eqref{a6} and the control of $\vzeta^{\ n}$ in $L^{q_3}(\Omega_T;\R^L)$, we may control uniformly the smallness of the set, where $c^n_i$, $=1,2,\dots,L$, are small. This finishes the proof of \eqref{e24}.

To conclude this limit passage, let us note that we may pass to the limit also in the entropy inequality to get
\begin{equation}\label{e26}
\begin{aligned}
&\int_\Omega (s^{*}_e (e) +s^{*}_{\cc}(\cc))(t)\dx  \geq  \int_\Omega (s^{*}_e (e) +s^{*}_{\cc}(\cc))(0)\dx \\
& + \int_{\Omega_T} \frac{\mathcal{S}:\mathcal{D}(\vv)}{\theta} - \vec{r} \cdot \vzeta + \frac{\kappa |\na \theta|^2}{(\theta)^2}  + \mathfrak{M}\Big(\na \vzeta + \frac{\vec{z}}{\theta} \na \varphi\Big)\Big(\na \vzeta + \frac{\vec{z}}{\theta} \na \varphi\Big) \dx \dt\\
& + \int_{\Gamma}  \kappa_\Gamma \Big(\frac {1} {\theta} - \frac {1} {\theta^\Gamma}\Big)\frac {1} {\theta} + \mathfrak{D} \Big(\vzeta -\vzeta^{\ \Gamma} + \frac{\vec{z}}{\theta^\Gamma} (\varphi - \varphi^\Gamma)\Big)\vzeta \dS \dt.
\end{aligned}
\end{equation}

\subsection{Limit passage $k \to \infty$}

We may deduce almost the same set of estimates independent of $k$ as in the previous subsection. We have

\begin{equation} \label{f1}
\begin{aligned}
\vv^k \to \vv \qquad & \mbox{weakly in }L^r(0,T;W^{1,r}(\Omega;\R^3)), \\
\vv^k \to \vv \qquad & \mbox{weakly}^* \mbox{ in }L^\infty(0,T;L^2(\Omega;\R^3)), \\
\pa_t \vv^k \to \pa_t \vv \qquad & \mbox{weakly in } (L^{r}(0,T;W^{1,r}_{0,\nu}(\Omega) \cap W^{1,r}_{\Div}(\Omega)) \cap L^{\frac{5r}{5r-6}}(0,T;W^{1,\frac{5r}{5r-6}}_{0,\nu}(\Omega)))', \\
\cc^{\ k} \to \cc \qquad & \mbox{strongly in } L^q(\Omega_T;\R^L), 1\leq q<\infty,\\
\cc^{\ k} \to \cc \qquad & \mbox{weakly in } L^{q_3}(0,T; W^{1,{q_3}}(\Omega;\R^L)), \\
\pa_t \cc^{\ k} \to \pa_t \cc \qquad & \mbox{weakly in } (L^2(0,T;W^{1,2}(\Omega;\R^L)) \cap L^{a'}(0,T;W^{1,a'}(\Omega;\R^L)))', \\
e^k \to e  \qquad & \mbox{strongly in } L^q(\Omega_T),  q<\frac 53,\\
e^k \to e \qquad & \mbox{weakly in } L^q(0,T;W^{1,q}(\Omega)), q< \frac 54,\\
\pa_t e^k \to \pa_t e \qquad & \mbox{weakly in } \mathcal{M}(0,T;(W^{3,2}(\Omega))^*), \\
\varphi^k \to \varphi \qquad & \mbox{strongly in } L^q(0,T;W^{1,q}(\Omega)), 1\leq q <\infty, \\
\varphi^k \to \varphi \qquad & \mbox{weakly in } L^q(0,T;W^{2,q}(\Omega)), 1\leq q<\infty, \\
\mathfrak{q}_{\cc}^{\ k} \to  \mathfrak{q}_{\cc} \qquad & \mbox{weakly in } L^{q_1}(Q_T; \R^{3L}), \\
\vc{q}_e^k \to \vc{q}_e \qquad & \mbox{weakly in } L^{q_2}(Q_T; \R^{3}),  \\
\vec{q}_{\cc \ \Gamma}^{\ k} \to \vec{q}_{\cc \ \Gamma} \qquad & \mbox{weakly in } L^2(\Gamma; \R^{L}), \\
q_{e\Gamma}^k \to q_{e\Gamma} \qquad & \mbox{weakly in } L^2(\Gamma), \\
q_{\varphi\Gamma}^k \to q_{\varphi\Gamma} \qquad & \mbox{weakly in } L^2(0,T; L^\infty(\pa\Omega), \\
\vzeta^{\ k} \to \vzeta \qquad & \mbox{weakly in } L^a(0,T;W^{1,a}(\Omega;\R^L)), \\
\mathcal{S}^k \to \mathcal{S} \qquad & \mbox{weakly in } L^{r'}(\Omega_T).
\end{aligned}
\end{equation}
The only difference with respect to \eqref{e8} is a slightly worse information concerning the time derivative of $\vv$ connected with the convective term in the momentum equation. To proceed, we have to distinguish three cases: \newline
$\bullet \  r\geq \frac{11}{5}$ \newline
$\bullet \ \frac 95 < r <\frac {11}{5}$ \newline
$\bullet \ \frac 32 <r\leq \frac {9}{5}$

\subsubsection {The case $r\geq \frac {11}{5}$}
This case is the simplest one as we may proceed exactly as in the previous limit passage. The reason for it is that we may still use the velocity as test function in the limit of the momentum equation which is the most important piece of information to deduce that
\begin{equation} \label{f2}
\mathcal{S}^k:\mathcal {D}(\vv^k) \to \mathcal{S}^*(\cc,\theta,\mathcal{D}(\vv)):\mathcal{D}(\vv)
\end{equation}
weakly in $L^1(\Omega_T)$. Therefore we use the procedure described in the previous subsection to finish the proof of Theorem \ref{T3}.

\subsubsection{The case $\frac 95 < r <\frac {11}{5}$}

The main difference in this case is that we are not anymore able to justify \eqref{f2} and we have to use the definition of the weak solution based on the total energy balance. Therefore, before passing to the limit in the equations, we first write down the approximate total energy balance. We proceed similarly as above with the difference that we aim at weak formulation, not just the balance of the total energy. To this aim, for $\psi$ arbitrary sufficiently smooth function, we use as test function in \eqref{e10} the function $\vv^k \psi$, in \eqref{e9} we use as test function $\vec{z}\varphi \psi$ and sum up these two identities with \eqref{e11}. We integrate the resulted identity over time and get
\begin{equation} \label{f3}
\begin{aligned}
&\int_0^T \int_\Omega E^k \pa_t \psi \dx \dt + \int_0^T \int_\Omega \Big(\Big(|\vv|^2\xi_k(|\vv^k|^2) -\frac 12 \Xi_k(|\vv^k|^2) + e^k + Q^k\varphi^k\Big) \vv^k \\
+ &\varphi^k \vec{z} \cdot \mathfrak{q}_{\cc}^{k} + \vc{q}_e^k - \mathcal {S}^k\vv^k  + p^k\vv^k-\varphi^k \na \pa_t \varphi^k\Big)\cdot \na \psi \dx\dt
+ \int_\Omega E^k(0) \psi(0) \dx \\
= &\int_0^T \int_{\pa \Omega} \Big(\gamma^k |\vv^k|^2 \psi + \varphi^k \vec{z}\cdot \vec{q}_{\cc \ \Gamma}^{\ k} + q_{e\Gamma}^k -\varphi \pa_t q_{\varphi\Gamma}^k\Big)\psi \dS \dt,
\end{aligned}
\end{equation}
where $\Xi_k$ is a primitive function to $\xi_k$, $E^k = e^k + \frac 12 |\vv^k|^2 + \frac 12 |\na \varphi^k|^2$, $Q^k = \sum_{i=1}^L z_i c_i^k$ and $p^k$ is the pressure due to the divergence-less constraint on $\vv^k$. Note that the fact that we deal with the slip boundary conditions for the velocity plays an important role here; the pressure is namely known to be integrable, see \eqref{def.p1}--\eqref{ap.pres}. Estimates \eqref{ap.pres} (with $p_i$ replaced by $p_i^k$, $i=1,\dots,4$) holds true. Note finally that the meaning of the fluxes with the upper index $k$ is similar as in the previous limit passages.

We may now pass to the limit in the weak formulation \eqref{e9}, \eqref{e10}, \eqref{f3} and \eqref{e12}. As
$$
|\vv|^2\xi_k(|\vv^k|^2) -\frac 12 \Xi_k(|\vv^k|^2) \to \frac 12 |\vv|^2
$$
strongly in $L^q(\Omega_T)$ for any $q<\frac{5r}{6}$ and $\vv^k \to \vv$ weakly in $L^\frac{3r}{3-r}(\Omega_T;\R^3)$, the limit passage in the convective term of the total energy (see \eqref{f3}) is possible for $r>\frac 95$. All other terms in the weak formulation of the total energy, concentration equation, momentum equation and the equation for the electrostatic potential are simpler. We rewrite the time derivative and integrate the equalities over time. We get

\begin{equation} \label{f4}
\begin{aligned}
& -\int_{\Omega_T} \cc \cdot \pa_t \vec{\psi} \dx \dt - \int_\Omega \cc^{\ 0} \cdot \vec{\psi}(0) \dx + \int_{\Omega_T} - \big(\cc\otimes \vv + {\mathfrak q}_{\cc}\big) : \nabla \vec{\psi} \dx \dt \\
&= -  \int_{\Gamma} \vec{q}_{\cc\,\Gamma} \cdot \vec{\psi} \dS \dt + \int_{\Omega_T} \vec{r} \cdot \vec{\psi} \dx \dt
\end{aligned}
\end{equation}
for all $\vec{\psi} \in C^\infty(\Omega_T;\R^L)$, $\vec{\psi}(T) =\vec{0}$,
\begin{equation} \label{f5}
\begin{aligned}
&  -\int_{\Omega_T} \vv \cdot \pa_t \vc{u} \dx \dt - \int_\Omega \vv^{0} \cdot \vc{u}(0) \dx + \int_{\Omega_T} {\mathcal S} : \mathcal{D}(\vc{u}) - (\vc{v}\otimes \vv) : \nabla \vc{u}  \dx \dt \\
&= -  \int_{\Gamma} \gamma\vv\cdot \vc{u} \dS \dt - \int_{\Omega_T} Q\nabla \varphi \cdot \vc{u} \dx \dt
\end{aligned}
\end{equation}
for all $\vc{u} \in C^\infty(\Omega_T;\R^3)$, $\vc{u}(T) = 0$, $\vc{u}\cdot \nu = 0$ on $\partial \Omega$, $\Div \vc{u} =0$ in $\Omega_T$,
\begin{equation} \label{f6}
\begin{aligned}
&\int_0^T \int_\Omega E \pa_t \psi \dx \dt + \int_0^T \int_\Omega \Big(\Big(\frac 12 |\vv|^2 + e + Q\varphi\Big) \vv \\
+ &\varphi \vec{z} \cdot \mathfrak{q}_{\cc} + \vc{q}_e - \mathcal {S}\vv  + p\vv \mathcal{I} -\varphi \na \pa_t \varphi\Big)\cdot \na \psi \dx\dt
+ \int_\Omega E(0) \psi(0) \dx \\
=& \int_0^T \int_{\pa \Omega} \Big(\gamma |\vv|^2 \psi + \varphi \vec{z}\cdot \vec{q}_{\cc \ \Gamma} + q_{e\Gamma} -\varphi \pa_t q_{\varphi\Gamma}\Big)\psi \dS \dt
\end{aligned}
\end{equation}
for all $\psi \in C^\infty(\Omega_T)$, $\psi(T) =0$,
\begin{equation} \label{f7}
\int_\Omega \nabla \varphi\cdot\nabla \psi - Q\psi\dx = \int_{\partial \Omega} q_{\varphi \Gamma} \psi \dS
\end{equation}
for all $\psi \in W^{1,2}(\Omega)$.

We would like to identify the limits of the fluxes, especially we would like to show that
\begin{equation} \label{f8}
\mathcal{S} = \mathcal{S}^*(\cc,\theta,\mathcal{D}(\vv)).
\end{equation}
However, we meet here a new difficulty. Since we cannot use as test function in the limit momentum equation the function $\vv$, the idea based on the theory of the monotone operators is not any more applicable. We have to replace it by another method, based on the Lipschitz truncation method. Strictly speaking, for the parameter $r$ in the present range we could apply another method, based on the $L^\infty$-truncation, but this method works only up to $r=\frac 85$ and our aim is to consider in the next case $\frac 32<r<\frac 85$. For this reason we present here only the method based on the Lipschitz truncation. This methods has the origin in the work \cite{FrMaSt} (for steady case) and its parabolic version were obtained in \cite{DiRuWo} and for Navier slip boundary condition and general Orlicz growth in \cite{BuGwMaSw}. Here, we use the very improved version of this method, which is introduced in \cite{BrDiSc}.


Note first that it is not difficult to justify (by using the strong convergence of $\cc^{\ k}$ and $\theta^k$) that
$$
\mathcal{S}^k  - \mathcal{S}^*(\cc,\theta,\mathcal{D}(\vv^k)) \to 0
$$
weakly in $L^{r'}(\Omega_T;\R^9)$ and even strongly in $L^1(\Omega_T; \R^9)$, therefore we will replace \eqref{f8} by considering the
\begin{equation} \label{f9}
\lim_{k\to \infty} \mathcal{S}_0^k-\mathcal{S}^*(\cc,\theta,\mathcal{D}(\vv)),
\end{equation}
where we denoted
$\mathcal{S}_0^k = \mathcal{S}^*(\cc,\theta,\mathcal{D}(\vv^k))$.

We will use  Theorem 2.2 and Corollary 2.4 from \cite{BrDiSc}. We first introduce certain notation. For $\alpha >0$ we say that $Q = I\times B\subset\R\times \R^3$ is an $\alpha$-parabolic cylinder, if $r_I = \alpha r_B^2$,  where $r_I$ is the radius of the interval $I$ and $r_B$ the radius of the ball $B$. By ${\mathcal Q}^\alpha$ we denote the set of all $\alpha$-parabolic cylinders. For $\kappa>0$ we denote $\kappa Q$ the scaled cylinder $\kappa Q = (\kappa I) \times (\kappa B)$, where $\kappa B$ is the scaled ball with the same center, similarly $\kappa I$. Then $\alpha$-parabolic maximal operators ${\mathcal M}^\alpha$ and ${\mathcal M}^\alpha_s$, $s\in [1,\infty)$ are defined
$$
\begin{array}{c}
\displaystyle ({\mathcal M}^\alpha f)(t,x) := \sup_{Q' \in {\mathcal Q}^\alpha; (t,x) \in Q'} \frac{1}{|Q'|} \int_{Q'} |f(s,y)|\, {\mathrm d}s \, {\mathrm d}y, \\
\displaystyle ({\mathcal M}^\alpha_s f)(t,x):= \Big(({\mathcal M}^\alpha |f|^s)(t,x)\Big)^{\frac 1s}.
\end{array}
$$
For $\lambda$, $\alpha>0$ and $\sigma>1$ we define
$$
 {\mathcal O}^\alpha_\lambda (\vc{z}) := \Big\{ (t,x); ({\mathcal M}^\alpha_\sigma(\xi_{\frac 13 Q_0}|\na^2 \vc{z}|)>\lambda \cap ({\mathcal M}^\alpha_\sigma(\xi_{\frac 13 Q_0}|\partial_t \vc{z}|)>\lambda\Big\},
$$
where $\xi$ is a suitable cut-off function and $\vc{z} \sim \na^{-1}\vv$; for more precise definition of $\vc{z}$ see the proof of Theorem 2.2 in \cite{BrDiSc}.

We have (see Theorem 2.2 and Corollary 2.4 in \cite{BrDiSc})
\begin{prop} \label{4.1}
Let $1<r<\infty$, $r$, $r'>\sigma$. Let $\vc{w}^k$ and $\mathcal{G}^k$ satisfy
$$
\langle \partial_t \vc{w}^k,\vc{u} \rangle = \langle\Div \mathcal{G}^k,\vc{u}\rangle
$$
for all $\vc{u} \in C^\infty_{0,\rm{div}}(Q_0)$, $Q_0= I_0 \times B_0 \subset \R \times \R^3$. Assume that $\vc{w}^k$ is a weak null sequence in $L^r(I_0;W^{1,r}(B_0;\R^3))$ and a strong null sequence in $L^\sigma(Q_0;\R^3)$ and bounded in $L^\infty(I_0;L^\sigma(B_0;\R^3))$. Further assume that $\mathcal{G}^k = \mathcal{G}^k_1 + \mathcal{G}^k_2$ such that $\mathcal{G}^k_1$ is a weak null sequence in $L^{r'}(Q_0;\R^9)$ and $\mathcal{G}^k_2$ converges strongly to zero in $L^\sigma(Q_0;\R^9)$. Then there exists a double sequence $\{\lambda_{m,k}\}\subset \R^+$ and $m_0 \in \N$ with
\begin{itemize}
\item[(a)] $2^{2^m} \leq \lambda_{m,k} \leq 2^{2^{m+1}}$ \newline
such that the double sequence $\vc{w}^{m,k} := \vc{w}^{\lambda_{m,k}}_{\alpha_{m,k}} \in L^1(Q_0;\R^3)$, $\alpha_{m,k} := \lambda_{m,k}^{2-r}$ and ${\mathcal O}_{m,k}:={\mathcal O}_{\lambda_{m,k}}^{\alpha_{m,k}}$ defined above satisfy for all $m\geq m_0$
\item[(b)] $\vc{w}_{m,k} \in L^s(\frac 14 I_0;W^{1,s}_{0,\rm{div}}(\frac 16 B_0;\R^3))$ for all $s <\infty$ and $\operatorname{supp}\, \vc{w}_{m,k} \subset \frac 16 Q_0$
\item[(c)] $\vc{w}_{m,k} = \vc{w}^k$ a.e. on $\frac 18 Q_0 \setminus {\mathcal O}_{m,k}$
\item[(d)] $\|\nabla \vc{w}_{m,k} \|_{L^\infty(\frac 14 (Q_0);\R^3)} \leq c \lambda_{m,k}$
\item[(e)] $\vc{w}_{m,k} \to 0$ in $L^\infty(\frac 14 Q_0;\R^3)$ for $k \to \infty$ and $m$ fixed
\item[(f)] $\nabla \vc{w}_{m,k} \rightharpoonup^* 0$ in $L^\infty(\frac 14 Q_0;\R^9)$ for $k \to \infty$ and $m$ fixed
\item[(g)] $\limsup _{k\to \infty} \lambda^q_{m,k} |{\mathcal O}_{m,k}| \leq c 2^{-m}$
\item[(h)] $\limsup_{k\to \infty}\Big| \int_{Q_0}\mathcal{G}^k: \nabla\vc{w}_{m,k}\dx\dt\Big| \leq c \lambda_{m,k}^r |{\mathcal O}_{m,k}|$
\item[(i)] Additionally, let $\zeta \in C^\infty_0(\frac 16 Q_0)$ with $\chi_{\frac 18 Q_0} \leq \zeta \leq \chi_{\frac 16 Q_0}$. Let $\vc{w}^k$ be uniformly bounded in $L^\infty(I_0;L^\sigma(B_0;\R^3))$, then for every $\mathcal{K} \in L^{r'}(\frac 16 Q_0;\R^9)$
$$
\limsup_{k\to \infty} \Big| \Big(\int_{Q_0}(\mathcal{G}^k_1+\mathcal{K}):\nabla \vc{w}^k\Big)\zeta \chi_{{\mathcal O}_{m,k}^C}\dx\dt\Big| \leq c 2^{-\frac mr}.
$$
\end{itemize}
\end{prop}

We apply this theorem to our problem; cf. Theorem 3.1 in \cite{BrDiSc}. We denote $\vc{w}^k = \vv^k - \vv$.  Then
$$
\begin{array}{c}
\vc{w}^k \rightharpoonup {\mathbf 0} \quad \mbox{ in } L^r(0,T;W^{1,r}_{\textrm{div}}(\Omega)), \\
\vc{w}^k \to {\mathbf 0} \quad \mbox{ in }
L^{2\sigma}(Q_T;\R^3), \\
\vc{w}^k \rightharpoonup^* {\mathbf 0} \quad \mbox{ in }L^\infty(0,T;L^2(\Omega;\R^3)),
\end{array}
$$
see \eqref{f1}. Further
$$
\int_0^T \int_{\Omega} \vc{w}^k \cdot \partial_t \vc{u} \dx\dt= \int_0^T \int_\Omega \mathcal{G}^k :\mathcal{D}(\vc{u})\dx\dt
$$
for all $\vc{u} \in C^\infty_{0}(\Omega_T;\R^3)$, where $\mathcal{G}^k = \mathcal{G}^k_1+ \mathcal{G}^k_2$ with
$$
\begin{array}{c}
\mathcal{G}^k_1 = \mathcal{S}^k_0-\mathcal{S} \\
\mathcal{G}^k_2 = -(\vv^k\otimes \vv^k) \xi(|\vv^k|^2)+ (\vv\otimes \vv) +Q^k\nabla \varphi_k -Q \nabla \varphi +\mathcal{S}^k_0-\mathcal{S}^k.
\end{array}
$$
We have $\|\mathcal{G}^k_1\|_{L^{r'}(0,T;L^{r'}(\Omega;\R^9))} \leq C$ and $\mathcal{G}^k_2 \to \mathbf{0}$ in $L^{\sigma_1}(\Omega_T)$ for some $\sigma_1 >1$ see \eqref{f3}.

Take now $Q\subset\subset(0,T)\times \Omega$. Due to properties mentioned above, assumptions of Proposition \ref{4.1} (i) are fulfilled. Hence,  plugging in $\mathcal{K} = -\mathcal{S}^*(\cc,\theta, \mathcal{D}(\vv)) + \mathcal{S}$ we have for $\zeta \in C^\infty_0(\frac 16 Q_0)$
$$
\limsup_{k\to \infty}\Big|\int_0^T \int_{\Omega}
\big(\mathcal{G}^k_1+\mathcal{K}\big):\mathcal{D}(\vc{w}^k) \zeta \chi_{{\mathcal O}_{m,k}^C}\dx\dt \Big| \leq C 2^{-\frac mr}.
$$
Therefore, using the fact that $\vc{w}^k$ coincides with $(\vv^k-\vv)$ on ${\mathcal O}_{m,k}^C$ and also the definition of $\mathcal{G}^k_1$, we get
$$
\limsup_{k\to \infty}\Big|\int_0^T \int_{\Omega}
\big(\mathcal{S}^k_0-\mathcal{S}^*(\cc,\theta, \mathcal{D}(\vv))\big):\mathcal{D}(\vv^k-\vv) \zeta \chi_{{\mathcal O}_{m,k}^C}\dx\dt \Big| \leq C 2^{-\frac mr}.
$$
This however due to the monotonicity of $\mathcal{S}^*$ implies that
$$
\limsup_{k\to \infty}\int_0^T \int_{\Omega}
\Big|\big(\mathcal{S}^k_0-\mathcal{S}^*(\cc,\theta, \mathcal{D}(\vv))\big):\mathcal{D}(\vv^k-\vv)\Big| \zeta \chi_{{\mathcal O}_{m,k}^C}\dx\dt \leq C 2^{-\frac mr}.
$$
Next, we take arbitrary $\mu \in (0,1)$ and  by virtue of H\"older's inequality and Proposition~\ref{4.1} (g) we deduce that
\begin{align*}
&\limsup_{k\to \infty}\int_0^T \int_{\Omega}\Big|(
\mathcal{S}^k_0-\mathcal{S}^*(\cc,\theta, \mathcal{D}(\vv))) :\mathcal{D}(\vv^k-\vv) \Big|^\mu\zeta \chi_{{\mathcal O}_{m,k}}\dx\dt \\
&\leq C\limsup_{k\to \infty} |{\mathcal O}_{k,m}|^{1-\mu} \leq C 2^{-(1-\mu)\frac mr}.
\end{align*}
Thus, combining both estimates we observe that
$$
\limsup_{k\to \infty}\int_0^T \int_{\Omega}\Big|(
\mathcal{S}^k_0-\mathcal{S}^*(\cc,\theta, \mathcal{D}(\vv))) :\mathcal{D}(\vv^k-\vv) \Big|^\mu \leq   C 2^{-(1-\mu)\frac mr}.
$$
Taking $\lim_{m\to \infty}$ the right-hand side tends to zero and consequently we have that
$$
\Big|(
\mathcal{S}^k_0-\mathcal{S}^*(\cc,\theta, \mathcal{D}(\vv))) :\mathcal{D}(\vv^k-\vv) \Big|^\mu \to 0
$$
in $L^1 (0,T; L^1(\Omega))$. Thus, using the Egorov theorem, we find that for any $\varepsilon>0$ there exists $\Omega_T^{\varepsilon}$ such that
$$
(
\mathcal{S}^k_0-\mathcal{S}^*(\cc,\theta, \mathcal{D}(\vv))) :\mathcal{D}(\vv^k-\vv) \to 0
$$
uniformly in $\Omega_T^{\varepsilon}$ and $|\Omega_T \setminus \Omega_T^{\varepsilon}|\le \varepsilon$, which due to the weak convergence of $\vv^k$ to $\vv$ also implies that $\mathcal{S}_0^k : \mathcal{D}(\vv^k-\vv)\to 0$ weakly in $L^{1}(\Omega_T^{\varepsilon})$. Due to the monotonicity of $\mathcal{S}^*$ and \eqref{f9}, we can apply the Minty method to finally obtain that
$$
\mathcal{S} = \mathcal{S}^*(\cc,\theta,\mathcal{D}(\vv)) \qquad \textrm{ a.e. in } \Omega_T^{\varepsilon}.
$$
Since, $\varepsilon>0$ is arbitrary, we can finally let it to zero and due to the fact that $|\Omega_T \setminus \Omega_T^{\varepsilon}|\le \varepsilon$, we see that the above identity holds a.e. in $\Omega_T$
Therefore the proof of \eqref{f8} is finished.


Having shown \eqref{f8}, we may easily repeat the arguments from the preceding section, including the validity of the entropy inequality. Note that due to the form of the weak formulation of the total energy it is not difficult to verify the information about the continuity of $E$. Theorem \ref{T2} is proved. Notice finally that the weak formulation of the total energy holds for any $r>\frac 95$, including the situation covered by Theorem \ref{T3}. However, we my deduce the weak formulation of the total energy after the limit passage as we may test the momentum equation by $\vv \psi$.

\subsubsection{The case $\frac 32 <r\leq \frac 95$}

In this situation we do not have enough information about the integrability of the velocity to justify neither the limit passage in the weak formulation of the internal energy balance nor in the weak formulation of the total energy balance. The way out is to pass to the limit in the weak formulation of the internal energy balance with only non-negative test functions, receiving only inequality in limit, and pass to the limit in the total energy balance (weak formulation tested by a constant test function).

In the latter case, approximating first the characteristic function of $[0,t]\times \Omega$ and then letting $k \to \infty$ in \eqref{f3} the total energy inequality
\begin{equation} \label{f50}
\begin{aligned}
&  \int_\Omega E(t) \psi \dx  +  \int_0^t \int_{\pa \Omega} \Big(\gamma |\vv|^2 \psi + \varphi \vec{z}\cdot \vec{q}_{\cc \ \Gamma} + q_{e\Gamma} -\varphi \pa_t q_{\varphi\Gamma}\Big)\psi \dS \dt \leq \int_{\Omega} E(0) \dx.
\end{aligned}
\end{equation}
Next, taking the limit $k\to \infty$ in the weak formulation of the internal energy balance yields due to the weak lower semicontinuity of $\int_{\Omega_T}\mathcal{S}^k :\mathcal{D}(\vv^k) \dx \dt$ inequality \eqref{w_ineq_e}. Note that the limit $r>\frac 32$ is due to the integrability of the term $\vv e$. To pass to the limit in the momentum equation, we use the Lipschitz truncation method exactly as in the case $r>\frac 95$. The rest is the same as in the previous case. Theorem \ref{T1} is proved.

Finally note that it would be more natural to replace the inequality in the internal energy balance by the inequality in the entropy inequality (with nonnegative test functions). This is, however, not so easy as we do not have enough information about the time integrability of some terms. The trouble makers are terms with $\theta$ in the negative powers. Therefore we can dispose  with the entropy inequality tested only by a constant function.

\medskip

\noindent {\bf Acknowledgement:} The first and the second authors (MB and MP) were partially supported by the Czech Science Foundation (grant no. 16-03230S). The third author (NZ) acknowledges support from the Austrian Science Fund (FWF), grants P22108, P24304, W1245. The paper was mostly written during the visits of NZ in Prague and MP and MB in Vienna in the frame of the project MOBILITY 7AMB15AT005. The authors acknowledge also this support.

\def\cprime{$'$}



\end{document}